\documentclass[12pt]{article}
 \usepackage{amscd}
\usepackage{epstopdf}
\usepackage{color}
\usepackage{latexsym}
\usepackage{epsf}
\usepackage{multicol}
\usepackage{ifpdf}
\usepackage{amsmath,amsfonts,amssymb,epsfig,subfigure}
\usepackage{mathrsfs}
\usepackage{stmaryrd}
\usepackage{graphicx,amsmath}
\usepackage{latexsym}
\usepackage{verbatim}
\usepackage{listings}
\usepackage[framed,numbered,autolinebreaks,useliterate]{mcode}
\ifpdf
\usepackage[%
  pdftitle={Instructions for use of the document class
    elsart},%
  pdfauthor={},%
  pdfsubject={The preprint document class elsart},%
  pdfkeywords={instructions for use, elsart, document class},%
  pdfstartview=FitH,%
  bookmarks=true,%
  bookmarksopen=true,%
  breaklinks=true,%
  colorlinks=true,%
  linkcolor=blue,anchorcolor=blue,%
  citecolor=blue,filecolor=blue,%
  menucolor=blue,pagecolor=blue,%
  urlcolor=blue]{hyperref}
\else
\usepackage[%
  breaklinks=true,%
  colorlinks=true,%
  linkcolor=blue,anchorcolor=blue,%
  citecolor=blue,filecolor=blue,%
  menucolor=blue,pagecolor=blue,%
  urlcolor=blue]{hyperref}
\fi
\makeatletter
\def\elsartstyle{%
    \def\normalsize{\@setfontsize\normalsize\@xiipt{14.5}}
    \def\small{\@setfontsize\small\@xipt{13.6}}
    \let\footnotesize=\small
    \def\large{\@setfontsize\large\@xivpt{18}}
    \def\Large{\@setfontsize\Large\@xviipt{22}}
    \skip\@mpfootins = 18\p@ \@plus 2\p@
    \normalsize
} \@ifundefined{square}{}{}
\makeatother

\newtheorem{theorem}{Theorem}[section]
\newtheorem{assp}{Assumption}
\newtheorem{lemma}[theorem]{Lemma}
\newtheorem{rem}[theorem]{Remark}
\newtheorem{expl}[theorem]{Example}
\newtheorem{cor}[theorem]{Corollary}

\makeatletter \@addtoreset{equation}{section}

\allowdisplaybreaks
\newcommand{\E}{\mathbb{E}}
\newcommand{\PP}{\mathbb{P}}
\newcommand{\RR}{\mathbb{R}}
\newcommand{\Se}{\mathbb{S}}
\newcommand{\dis}{\displaystyle}
\def\nn{\nonumber}

\def\F{{\cal F}}

\def\trace{\hbox{\rm trace}}
 \def\f{\varphi} \def\r{\rho}\def\e{\varepsilon}
\def\lf{\left} \def\rt{\right}\def\t{\triangle} \def\ra{\rightarrow}
\def\la{\label}\def\be{\begin{equation}}  \def\ee{\end{equation}}
\def\SS{{\mathbb{S}}}

\usepackage[text={7in,9in},centering]{geometry}

\begin{document}


\title{Explicit approximations for nonlinear switching diffusion systems in finite and infinite horizons\thanks{This work was supported by National Natural Science Foundation of China (11171056, 11471071, 11671072), the Natural Science Foundation of Jilin Province (No.20170101044JC), the Education Department of Jilin Province (No.JJKH20170904KJ).}}
\author{Hongfu Yang, Xiaoyue Li\thanks{Corresponding author,~~Email address: lixy209@nenu.edu.cn. (Xiaoyue Li)}}
 \date{School of Mathematics and Statistics, Northeast Normal University,
Changchun, Jilin, 130024, China}
\maketitle

\begin{abstract}
Focusing on hybrid diffusion dynamics involving continuous dynamics as well as discrete events, this article investigates the explicit approximations for nonlinear switching diffusion systems   modulated by a Markov chain. Different kinds of easily implementable explicit schemes have been proposed to approximate the dynamical behaviors of switching diffusion systems with local Lipschitz continuous drift and diffusion coefficients in both finite and infinite intervals. Without additional restriction conditions except those which guarantee the exact solutions posses their dynamical properties, the numerical solutions converge  strongly to the exact solutions in finite horizon, moreover, realize the approximation of long-time dynamical properties including the moment boundedness, stability and ergodicity. Some simulations and examples are provided to support the theoretical results and demonstrate the validity of the approach.

\end{abstract}

\section{Introduction}\label{s-w}

The switching diffusion systems (SDSs) modulated by  Markov chains involving
   continuous dynamics and discrete events, have drawn more and more attention to many researches. Much of the study originated from applications arising from
  biological systems, financial engineering, manufacturing systems, wireless communications  (see, e.g., \cite{Shen2013, Mao06, So, Ta, yz09}  and the references therein). Compared with those of the subsystems the  dynamics of SDSs are seemingly   much different. For instance, considering a  predator-prey ecosystem switching between two environments randomly,  Takeuchi et al. in \cite{Ta} revealed that both subsystems  develop periodically but switching between them makes them neither permanent nor dissipative. Pinsky and his coauthors in \cite{Pin1, Pin2} provided several interesting examples to show that the switching system is recurrent (resp. transient) even if its subsystems are transient (resp. recurrent). Due to the coexistence of continuous dynamics and discrete events, the dynamics of SDSs  are  full of uncertainty  and challenge.

Since solving SDSs is almost unavailable, numerical scheme or approximation techniques  become   viable alternatives.
The explicit  Euler-Maruyama (EM)  scheme  is  popular for approximating  diffusion systems   and   SDSs
 with global Lipschitz  coefficients \cite{Mao06,  Kloeden, yz09}. However, the coefficients of many important diffusion systems and SDSs are only locally Lipschitz  and  superlinear (see, e.g., \cite{Hutzenthaler15, 16, yz09, So} and the references therein).  Hutzenthaler et al. \cite[Theorem 2.1]{16} showed that the  absolute moments of the EM approximation for a large class of diffusion systems with superlinear  growth coefficients  diverge to infinity at a finite time point $T\in (0, \infty)$. The implicit EM scheme is better than the explicit EM scheme in that its numerical solutions  converge strongly to the exact solutions of diffusion and switching diffusion systems with the one-sided Lipschitz drift coefficient and global Lipschitz diffusion coefficient (see Higham et al. \cite{Hi02}, Mao and Yuan \cite[p.134-153]{Mao06}). Nevertheless, additional computational efforts are required for its implementation since the solution of an algebraic equation has to be found  before each iteration.
Due to the    advantages of explicit schemes (e.g., simple  structure and cheap computational cost),  a few modified EM methods
have been developed for   diffusion systems with nonlinear  coefficients including the tamed EM method  \cite{Hutzenthaler12,Hutzenthaler15,Sabanis13,Sabanis16},   the tamed Milstein method \cite{Wang13},  the stopped EM method \cite{LiuW13} and the truncated EM method  \cite{Mao20152}. These modified EM methods have shown their abilities to approximate the solutions of nonlinear diffusion systems.
 However, to the best of our knowledge these  methods are not developed, even unavailable for a large class of nonlinear SDSs.
 For instance, Hamilton \cite{James}  remarked that the economy may either be in a fast growth or slow growth phase  with the regime switching   governed by the outcome of a Markov   chain.  Consider a two-dimensional nonlinear stochastic volatility model  switching randomly between   a fast growth  phase
\begin{align}\la{sex1}
\mathrm{d}X(t)=2.5X(t)\Big(1-|X(t)|\Big)\mathrm{d}t+\left(
  \begin{array}{ccc}
    -1 & \sqrt{2}\\
     \sqrt{2}  & 1\\
  \end{array}
\right)|X(t)|^{3/2}\mathrm{d}B(t),
\end{align}
 and a slow growth phase
\begin{align}\la{sex2}
\mathrm{d}X(t)= \Big((1, 2)^T- X(t)\Big)\mathrm{d}t+ \left(
  \begin{array}{ccc}
    0.2 & -0.5\\
    1 &  0.4\\
  \end{array}
\right)|X(t)|\mathrm{d}B(t),~~~~~~
\end{align}
modulated by a Markov chain $r(t)$.
This SDS
is one of popular volatility models  used for  pricing option in Finance  (see, e.g., \cite{James, Shen2013, Goard, So} and the references therein),
especially, for pricing VIX options.  Based on its importance and no closed-form,
one of our aims is to construct the explicit scheme available for this kind of  SDSs and to study its convergence in the $p$th moment as well as the rate of convergence.

  On the other hand,  long-time behaviors of SDSs are also the major concerns in stochastic processes, systems theory, control  and optimization (see, e.g., the monographs \cite{Mao06,yz09} and the references therein).
 So far, the dynamical properties of SDSs are investigated deeply including stochastic stability (see, e.g., \cite{Mao06,yz09,Khasminskii2007,shao}), invariant densities (see, e.g., \cite{Bakhtin12,Bakhtin15}),  recurrence and transience (see, e.g., \cite{yz09, Shao15}) and so on.  Although the finite-time convergence is one of the fundamental concerns,  how to approximate  long-time behaviors of SDSs  is   significant and challenging.
Recently,  Higham et al. \cite{Higham07} by using the EM scheme, Mao and Yuan \cite[pp.229-249]{Mao06}  by using the backward EM (BEM) scheme showed respectively that  the numerical solutions  preserve the mean-square exponential stability of SDSs with  globally  Lipschitz continuous coefficients.   Mao et al. \cite{Mao2011} gave a counter-example  that the EM numerical solutions  don't share the underlying
almost surely exponential stability of  SDSs  with the nonlinear growth drift term,
but the  BEM numerical solutions do.
Mao et al. \cite{Mao2005}, Yuan and Mao \cite{Yuan3} and  Bao et al. \cite{Bao2016} made use of  the EM  method with a constant stepsize to approximate the underlying invariant measure of SDSs with linear growth  coefficients while
Yin and Zhu \cite[p.159-179]{yz09} did that by using  the EM scheme with the  decreasing step sizes.
  In the above mentioned works,  the diffusion coefficients of   SDSs are always required to be globally Lipschitz continuous.   For the further development of numerical schemes for SDSs, we refer  readers to \cite{Nguyen17}, for example, and the references therein. 
  Accommodating  many applications although these systems are more realistically addressing the demands, the nontraditional setup makes the discrete approximations of the long-time behaviors  for    nonlinear SDSs  more difficult.
  Thus,  in order to close the  gap, the other  aim  is  to construct an appropriate explicit  scheme   for  a large class of nonlinear SDSs  such that the numerical solutions realize  the underlying infinite-time dynamical properties, such as the $p$th moment ($p>0$) boundedness,  stability   as well as  approximate the underlying  invariant measure.

 Motivated by  the truncation idea in \cite{Mao20152},  together with the novel approximation technique, we have constructed a new  explicit scheme  and get the convergence with $1/2$ order rate  in  finite  time interval.  Then  we go further to improve the scheme according to the structure of the SDSs 
  such that it is easily  implementable for approximating    the underlying invariant measure{ (resp. stability)}.
 The schemes proposed in this paper  are obviously different from those  of
  \cite{Mao06,Sabanis13, Mao20152,Zhou2015}.
  More precisely, the numerical solutions at the grid points are modified before each iteration according to the growth rate of the drift and diffusion coefficients such that the numerical solutions keep  the underlying excellent properties of the exact solutions of SDSs. Our contributions  are as follows.
\begin{itemize}
\item[$\bullet$]  We construct an easily implementable scheme  for the SDSs with only local Lipschitz  drift and diffusion coefficients and establish finite-time  moment convergence results.  The rate of convergence is also estimated under slightly stronger conditions.

\item[$\bullet$] Using novel techniques (i.e., combining the Lemma \ref{L:1} with the Perron-Frobenius theorem (see \cite[p.6]{Chen2007}) to construct appropriate Lyapunov functions depending on the states and analyzing their asymptotic properties),
 we obtain the criterion on the existence and uniqueness of  invariant measures of the SDSs as well as the $p$th moment ($p>0$) and the almost surely exponential stability; see  Theorem  \ref{yth3.1} and Corollary \ref{CT:F_0}.

\item[$\bullet$]  We reconstruct the explicit scheme which numerical solutions admit a unique  invariant measure  converging  to the underlying one in the Wasserstein distance.  The restrictions on the  coefficients by using the EM scheme
(c.f. \cite{Yuan3,Bao2016,Liu}) are relaxed. Thus this scheme is more suitable for the measure approximation of   nonlinear SDSs.


 \item[$\bullet$]  Without extra restrictions the numerical solutions of the appropriate  explicit scheme   stay in step of dynamical properties with the exact solutions.
\end{itemize}

The rest of the paper is arranged as follows. Section \ref{n-p} begins with notations and preliminaries on  the properties of the exact solutions. Section \ref{s-c}  constructs an explicit scheme,  and yields the convergence in $p$th moment and the  rate.  Two examples are given to illustrate  the availability of this scheme.
 Section \ref{Inv} focuses on the analysis of  invariant measures.
The other explicit scheme is constructed preserving the stability in distribution, which numerical invariant measure approximates  the underlying one in the Wasserstein distance.
 Several numerical experiments are presented  to illustrate the results.
 Section \ref{s6} gives some concluding remarks to conclude the paper.


\section{Notations and preliminaries}\label{n-p}
Throughout this paper,  we  use the following notations. Let $d$,  $m$ and $n$ denote finite positive integers,  $|\cdot|$ denote the Euclidean norm in $\RR^n:=\RR^{n\times 1}$ and the trace norm in $\RR^{n\times d}$.
 If $A$ is a vector or matrix, its transpose is denoted by $A^T$ and its trace norm is denoted by $|A|=\sqrt{\mathrm{trace}(A^TA)}$.  For vectors  or matrixes $A$ and $B$ with compatible dimensions, $A B$ denotes the usual matrix multiplication. If $A$ is a symmetric matrix, denote by $\lambda_{\max}(A)$ and $\lambda_{\min}(A)$ its largest and smallest eigenvalue, respectively.
 For any
$c=(c(1), \dots, c(m))\left(\mathrm{ or }=(c_1, \dots, c_m)\right)$
 define $\dis\hat{c}=\min_{1\leq i\leq m}c(i)  (   c_i$) and $\dis\check{c}=\max_{1\leq i\leq m}c(i)( c_i)$. For any $a, b\in \mathbb{R}$,
 $a\vee b:=\max\{a,b\}$, and $a\wedge b:=\min\{a,b\}$.
If $\mathbb{D}$ is a set, its indicator function is denoted by $I_{\mathbb{D}}$, namely $I_{\mathbb{D}}(x)=1$ if $x\in \mathbb{D}$ and $0$ otherwise.  Let $C_i$ and $C$ denote two generic positive real constants respectively, whose value may change in different appearances, where $C_i$ is  dependent on $i$ and $C$ is independent of  $i$.

 Let $( \Omega,~\cal{F}$,~$\PP )$ be a complete  probability space, and $\mathbb{E}$ denotes the expectation corresponding to $\PP$. Let~$B(t)$ be a $d$-dimensional Brownian motion defined on this probability space. Suppose that  $\{r(t)\}_{t\geq0}$ is a right-continuous Markov chain with finite  state space $\mathbb{S}=\{1, 2, \dots, m\}$ and independent of the Brownian motion $B(\cdot)$. Suppose ${\{ {\cal{F}}_{t}\}} _{t \geq 0}$ is a filtration defined on this probability space satisfying the usual conditions (i.e., it is right continuous and $\mathcal{F}_0$ contains all $\mathbb{P}$-null sets) such that $B(t)$ and $r(t)$
are ${\cal{F}}_{t}$ adapted.
  The   generator of $\{r(t)\}_{t\geq0}$ is  denoted by $\Gamma=(\gamma_{ij})_{m\times m}$,  so that for a sufficiently small $\delta>0$,
\begin{align*}
 \mathbb{P}\{r(t+\delta)=j|r(t)=i\}=\left\{
\begin{array}{ll}
\gamma_{ij}\delta+o(\delta),&~~\mathrm{if}~~i\neq j,\\
1+\gamma_{ii}\delta+o(\delta),&~~\mathrm{if}~~i=j,
\end{array}
\right.
\end{align*}
where $ {o}(\delta)$ satifies $\lim_{\delta\rightarrow 0} o(\delta)/\delta=0$. Here $\gamma_{ij}\geq 0$ is the transition rate from $i$ to $j$
if $i\neq j$ while
  $
\gamma_{ii}=-\sum_{i\neq j}\gamma_{ij}.
$
It is well known that almost every sample path  of $r(t)$ is  right-continuous step functions with a finite number of simple jumps in any finite interval of   $\mathbb{R}_+:=[0, +\infty)$(c.f. \cite{Mao2011}).

In this paper, we consider the  $n$-dimensional  SDS described by a  hybrid stochastic differential equation (HSDE)
\be\label{e1}
\mathrm{d}X(t)=f(X(t),r(t))\mathrm{d}t +g(X(t),r(t))\mathrm{d}B(t)
\ee
  with an initial value $(X(0), r(0))=(x_0, \ell)\in \mathbb{R}^n\times \mathbb{S}$, where the drift and diffusion terms
$$
  f: ~\RR^n\times \mathbb{S}    \rightarrow \RR^{n}, ~~~~ ~~~~g: ~\RR^n\times \mathbb{S}  \rightarrow \RR^{n\times d},
$$
are  locally Lipschitz continuous, that is,
for any $N>0$  there exists a positive   constant $C_N$ such that, for any $x, y\in \RR ^n  $ with  $|x |\vee |y|\leq N$ and any $i\in \mathbb{S}$,
  $$|f(x, i)-f(y, i)|\vee |g(x, i)-g(y, i)|\leq C_N |x-y|.$$

Let
  $\mathcal{C}^{2} (\RR^n \times \mathbb{S};  {{\RR}}_+)$ denote the family of all nonnegative functions
 $V(x, i)$ on $\RR^n  \times \mathbb{S}$ which are continuously twice
 differentiable in $x$.
 For each $V(x, i)\in \mathcal{C}^{2} (\RR^n  \times \mathbb{S};   \RR_+)$, define an
operator ${\cal{L}}V$ from
 $  \RR^n  \times \mathbb{S}$ to $  \RR$ by
\begin{align}\la{1e1.2}
{\cal{L}}V(x,  i ) =& V_x(x, i)f(x, i)
 +\frac{1}{2} \hbox{trace}\left(g^T(x, i)V_{xx}(x, i)g(x, i)\right)+\sum_{j=1}^{m}\gamma_{ij}V(x,  j),
\end{align}
where
$$ V_x(x,  i)=\left(\frac{\partial V(x,  i)}{\partial x_1},
 \dots,\frac{\partial V(x,  i)}{\partial x_n}\right),~~~~ V_{xx}(x,  i)=\left(\frac{\partial^2
V(x,  i)}{\partial x_j  \partial x_l}\right)_{n\times n}.$$

In order for the existence of exact regular solutions  we impose the following assumption.
  \begin{assp}\label{a1}
For some  $\bar{p}>0$ and each $ i \in \mathbb{S}$, there exists a  symmetric positive-definite matrix  $Q_i\in \RR^{n\times n}$  and a  constant  $\alpha_i$ such that
\begin{align}\label{e2}
\limsup_{|x|\rightarrow \infty } \displaystyle\frac{(1+x^TQ_ix)\psi(x,i)-(2-\bar{p})|x^TQ_ig(x, i)|^2}{(x^TQ_ix)^2}\leq  \alpha_i,~~~~\forall x  \in\RR^{n},
\end{align}
 where
$
\psi(x,i):=2x^T Q_{i}f(x, i)+\trace{\left[g^T(x, i)Q_{i}g(x, i)\right]}.
$
  \end{assp}

Now we give the  regularity of  exact solutions as well as the  estimation on their $p$th moment.
  \begin{theorem}\la{T:1}
Under Assumption   \ref{a1}, SDS (\ref{e1}) with any initial value $(x_0, \ell)\in \mathbb{R}^n\times \mathbb{S}$ has a unique regular solution $(X(t),r(t))$ satisfying
  \be\la{e0p}\sup_{0\leq t\leq T}
   \E|X(t)|^{p  } \leq C
   \ee
for any $p\in(0, \bar{p}]$, $T> 0$, where the positive constant $\bar{p}$ is given by  Assumption   \ref{a1}.
  \end{theorem}
\begin{proof}
For any $i\in \mathbb{S}$ and  $ 0<\kappa_i< 1$,  Assumption  \ref{a1} together with the continuity of $f$ and $g$ implies that there exists a positive constant  $C_i$ such that
\begin{align}\label{eq02}
 \displaystyle (1+x^TQ_ix) \psi(x,i)-(2-\bar{p})|x^TQ_ig(x, i)|^2
   \leq  \left(\alpha_i+\frac{\kappa_i}{2}\right)(1+x^TQ_ix)^2+C_i
\end{align}
for any $ x \in \RR^n $. For  any fixed  $0<p\leq\bar{p}$,  define $V(x, i)= (1+x^TQ_ix)^{\frac{p}{2}}.$
Direct calculation, together with (\ref{eq02}), leads to
\begin{align}\la{YH_00}
 {\cal{L}}V(x, i)
=&\frac{p}{2} (1+x^TQ_ix)^{\frac{p}{2}-2}\Big\{
(1+x^TQ_ix) \left[2x^TQ_i f(x, i)+ \trace\left(g^T(x,i)Q_ig(x,i)\right) \right]\nn\\
&~~~~~~~~~~~~~~~~~~~~~~~ -(2-p)|x^TQ_ig(x, i)|^2 \Big\}
 +\sum_{j=1}^{m}\gamma_{ij}(1+x^TQ_jx)^{\frac{p}{2}}\nn\\
\leq&
 \frac{p}{2}\left(\alpha_i+  \frac{\kappa_i}{2}\right) (1+x^TQ_ix)^{\frac{p}{2}}
 + \sum_{j=1}^{m}\gamma_{ij}(1+x^TQ_jx)^{\frac{p}{2}}
 +\frac{p C_i}{2}(1+x^TQ_ix)^{\frac{p}{2}-2}.
\end{align}
For any $(x, i)\in \RR^n \times\mathbb{S}$ and any $p>0$, using the Young inequality, we obtain
\begin{align}\la{YH_2}
\frac{p C_i}{2}(1+x^TQ_ix)^{\frac{p}{2}-2}
\leq&
\frac{p\kappa_i }{4}(1+x^TQ_ix)^{\frac{p}{2}}+ C_i.
\end{align}
Therefore it follows from (\ref{YH_00}) and (\ref{YH_2}) that
\begin{align}\la{YH_02}
{\cal{L}}V(x,i)
\leq&  \bigg[\frac{p}{2}\big(\alpha_i+ \kappa_i \big)
+ \gamma_{ii}+  \sum_{j\neq i}^{m}\gamma_{ij}\big(\check{\lambda}\hat{\lambda}^{-1}\big)^{\frac{p}{2}}\bigg](1+x^TQ_ix)^{\frac{p}{2}}+C_i,
\end{align}
where  $\check{\lambda}:=\max_{i\in \mathbb{S}}\{\lambda_{\max}(Q_i)\}, \hat{\lambda}:=\min_{i\in \mathbb{S}}\{\lambda_{\min}(Q_i)\}.$
%
The next proof is standard (see e.g., Mao and Yuan \cite[Theorem 3.19, p.95]{Mao06}) and hence is omitted.
 \end{proof}

Define  $\tau_N=\inf\left\{ t\geq 0:|X(t)|\geq N\right\}$. By the virtue of Theorem \ref{T:1}, we have
\begin{align} \la{eq0-1}
\mathbb{P}\{\tau_{N}\leq T\}\leq \frac{C}{N^{p}}
\end{align}
for all  $p\in(0, \bar{p}]$, where $C$  is  a constant independent of $N$.


\section{Moment estimate and strong convergence}\label{s-c}
In this section, we aim  to construct an easily implementable explicit scheme and show the strong convergence
 under Assumption \ref{a1}. Given a stepsize $\Delta >0$ and let $t_k=k\t$, $r_k =r(t_k)$ for  $k\geq0$, and one-step transition probability matrix
  $
P(\Delta)=\left( P_{ij}(\Delta)\right)_{m\times m}=\exp(\Delta \Gamma).
$
 The discrete Markov chain $\{r_k,\ k=0, 1, \dots\}$ can be simulated by the technique in  \cite[p.112]{Mao06}.
To define  appropriate numerical solutions,  we firstly choose several strictly increasing continuous functions $\f_i: [1, \infty)\rightarrow \RR_+$ such that $\f_i(u)\rightarrow \infty$
as $u\rightarrow \infty$ and
\be\la{e21}
\sup_{|x|\leq u} \dis\left(\frac{ |f (x,i)|}{1+|x|}\vee\frac{| g(x, i)|^2 }{(1+|x|)^2}\right)\leq \f_i(u),~~~~~\forall~i\in \mathbb{S},~ u\geq 1.
\ee
We note that $\f_i$ is well defined since $f(\cdot, i)$ and $g(\cdot, i)$ are locally Lipschitz continuous. Denote by $\f_i^{-1} $ the inverse function of $\f_i$, obviously $\f_i^{-1}: [\f_i(1),\infty)\rightarrow \RR_+ $ is  a strictly increasing continuous function.
We also choose   a  strictly decreasing $h:(0, 1]\rightarrow \big[\check{\f}(|x_0|\vee 1), \infty\big)$ such that
\be\la{e22}
\lim_{\t\rightarrow 0} h(\t)= \infty ~~\hbox{and} ~~\t^{1/2}h(\t)\leq K,~~~~~~\forall~ \t \in (0, 1],
\ee
where $ {K}$ is a positive constant  independent of the iteration order  $k$ and the stepsize $\t$.
For   any given $\t\in (0, 1]$,  each $i\in \Se$, define a truncation mapping $\pi_{\t}^i:\RR^n\ra \RR^{n}$ by
 \begin{align*}
\pi_\t^i(x)= \Big(|x|\wedge \f^{-1}_i(h(\t))\Big) \frac{x}{|x|},
\end{align*}
where we use the convention  $\frac{x}{|x|}=\mathbf{0}$ when $x=\mathbf{0}\in \mathbb{R}^n$.
Obviously, for any $(x, i)\in \RR^n \times\mathbb{S},$
 \begin{equation}\la{e23}
 |f(\pi_\t^i(x), i)|
 \leq      h(\t) \left(1+ |\pi_\t^i(x)| \right),~~~~
 |g(\pi_\t^i(x), i)|^2
\leq     h(\t) \left(1+ |\pi_\t^i(x)| \right)^2.
 \end{equation}
 \begin{rem}
If there exists a  state  $i \in  \mathbb{S}$ such that
$ |f(x, i)|\leq C(1+|x|) ,~ |g(x, i)|^2\leq C(1+|x|)^2, ~\forall x\in \mathbb{R}^n, $
   we  choose $\f_i(u)\equiv C$ for any   $u\geq 1$, and  let $\f^{-1}_i(u)\equiv +\infty$ for any $u\in [C, +\infty)$. Then
$\pi_\t^i(x)=x$, and (\ref{e23})
hold always. 
\end{rem}

  Next we propose an explicit scheme  to approximate the exact solution of SDS (\ref{e1}). To avoid the arbitrarily large excursions from Brownian pathes,
 we truncate the grid point value according to the growth rate of the drift and diffusion functions. 
 Define
 \begin{align}\la{Y_0}
\left\{
\begin{array}{ll}
\tilde{Y}_0=x_0,~~r_0=\ell,&\\
Y_{k}=\pi_\t^{r_{k}}(\tilde{Y}_{k}),&\\
\tilde{Y}_{k+1}=Y_k+f(Y_k, r_k)\t +g(Y_k, r_k)\t B_k,& \\
\end{array}
\right.
\end{align}
for any integer $k\geq 0$, where $\t B_k=B(t_{k+1})-B(t_{k})$. We call this iteration   the {\it truncated EM scheme}. This method prevents the diffusion term from producing extra-ordinary large value. Additionally, the drift term and diffusion term have the linear property
\begin{align}\la{Y_01}
|f (Y_k, r_k) |  \leq    h(\t) (1+| Y_k|), ~~~|g(Y_k, r_k)|^2 \leq    h(\t) (1+| Y_k|)^2.
\end{align}
To obtain the continuous-time approximations,   define $\tilde{Y}(t)$, $Y(t)$ and $\bar{r}(t)$ by
 \begin{align*}
\tilde{Y}(t):=\tilde{Y}_{k},~~~~Y(t):=Y_k,~~~~\bar{r}(t):=r_k,~~~~\forall t\in[t_k,  t_{k+1}).
\end{align*}
  For convenience, denote by
 $\mathcal{G}_{t_k}$ the $\sigma$-algebra  generated by   $\{\mathcal{F}_{t_k}, r_{k+1}\}$. Obviously,
 $\mathcal{F}_{t_k} \subseteq \mathcal{G}_{t_k}$.
The following  lemma will play an important role  in the  proof  of the moment boundedness of the numerical solutions.
  \begin{lemma}\la{Z:1}
For any measurable functions    $ \phi:\mathbb{R}^n\times \mathbb{S}\times \mathbb{S}\rightarrow \mathbb{R}$, $ \tilde{\phi}:\mathbb{R}^n\times \mathbb{S}\times \mathbb{S}\rightarrow \mathbb{R}^{1 \times d}$ and   $ \bar{\phi}:\mathbb{R}^n\times \mathbb{S}\times \mathbb{S}\rightarrow \mathbb{R}^{d\times d}$, we have
 \begin{align}\la{lyhf1}
&\mathbb{E}\Big[\phi(Y_{k}, r_k,r_{k+1})\big|\mathcal{F}_{t_k}\Big] =\phi(Y_{k}, r_k, r_k)
+\sum_{j\in \mathbb{S}} \phi(Y_{k},r_k, j) [\gamma_{r_k j}\t+o(\t)],\\
& \mathbb{E}\Big[\tilde{\phi}(Y_{k}, r_k,r_{k+1})\t B_k\big|\mathcal{F}_{t_k}\Big] = 0,&\la{lyhf1+}
\end{align}
and
\begin{align}\la{lyhf2}
 \E\Big[\t B_k^T \bar{\phi}(Y_{k}, r_k,r_{k+1})\t B_k\big|\mathcal{F}_{t_k}\Big]
 =& \trace{\left(\bar{\phi}(Y_{k}, r_k,r_{k})\right)}\t\nn\\
&+ \sum_{j\in \mathbb{S}} \trace{\left(\bar{\phi}(Y_{k}, r_k,j)\right)} [\gamma_{r_k j}\t^2+o(\t^2)].
\end{align}
  \end{lemma}
\begin{proof}
First of all, note  that    $Y_k$ and $r_k$ are $\F_{t_k}$-measurable.  
 Then, by the Markov property, we derive that
 \begin{align*}
& \mathbb{E}\Big[\phi(Y_{k}, r_k,r_{k+1})\big|\mathcal{F}_{t_k}\Big]\nn\\
  =&\E \Big[\phi(Y_{k}, r_k,r_{k+1}) I_{\{r_{k+1}=r_k\}}\big|\mathcal{F}_{t_k}\Big]
   +\E \Big[\phi(Y_{k}, r_k,r_{k+1}) I_{\{r_{k+1}\neq r_k\}}\big|\mathcal{F}_{t_k}\Big]\nn\\
=&\sum_{i\in \mathbb{S}}I_{\{r_k=i\}} \phi(Y_{k}, i, i) \mathbb{P}\{r_{k+1}=i|r_k=i\} +\sum_{i,j\in \mathbb{S}, j\neq i} I_{\{r_k=i\}}\phi(Y_{k}, i, j) \mathbb{P}\{r_{k+1}=j|r_k=i\}\nn\\
= &\sum_{i\in \mathbb{S}}I_{\{r_k=i\}} \phi(Y_{k}, i, i)\Big(1 +\gamma_{i i}\t+o(\t)\Big)
  +\sum_{i\in \mathbb{S}}I_{\{r_k=i\}}\sum_{j\in \mathbb{S}-\{i\}}\phi(Y_{k}, i, j) \Big(\gamma_{i j}\t+o(\t)\Big)\nn\\
=&\phi(Y_{k}, r_k, r_k)+\sum_{j\in \mathbb{S}} \phi(Y_{k},r_k, j) [\gamma_{r_k j}\t+o(\t)].
\end{align*}
Thus, the required assertion \eqref{lyhf1} follows.
 Note also that    $Y_k$, $r_{k}$ and $r_{k+1}$ are $\mathcal{G}_{t_k}$-measurable.
Using  the properties
\begin{align}\la{L_2+}  \mathbb{E}\big[\triangle B_k|\mathcal{G}_{t_k}\big]=\mathbb{E}(\triangle B_k)=0,~~~~ \mathbb{E}\big[\t B_k\t B_k^T|\mathcal{G}_{t_k}\big]=\mathbb{I}_d \t,
\end{align}
where $\mathbb{I}_d $ denotes the $d\times d$ identity matrix.  Thus, we obtain that
\begin{align*}
 \E\Big[\tilde{\phi}(Y_{k}, r_k,r_{k+1})\t B_k\big|\mathcal{F}_{t_k}\Big]
= &\E\Big[ \mathbb{E}  \big(\tilde{\phi}(Y_{k}, r_k,r_{k+1})\t B_k\big|\mathcal{G}_{t_k}\big)\big|\mathcal{F}_{t_k}\Big]\nn\\
=& \E\Big[\tilde{\phi}(Y_{k}, r_k,r_{k+1}) \mathbb{E}  \big(\t B_k\big|\mathcal{G}_{t_k}\big)\big|\mathcal{F}_{t_k}\Big]=0,
\end{align*} which implies that  \eqref{lyhf1+} holds. Moreover,
\begin{align} \la{lyhf4+}
 \E\Big[\t B_k^T\bar{\phi}(Y_{k}, r_k,r_{k+1})\t B_k\big|\mathcal{F}_{t_k}\Big]
 =&\E\Big[\mathbb{E}\Big(\trace{\big(\bar{\phi}(Y_{k}, r_k,r_{k+1})\t B_k\t B_k^T\big)}\big|\mathcal{G}_{t_k}\Big)\big|\mathcal{F}_{t_k}\Big]\nn\\
=&\trace{\Big\{\E\Big[\bar{\phi}(Y_{k}, r_k,r_{k+1})\mathbb{E}\big(\t B_k\t B_k^T\big|\mathcal{G}_{t_k}\big)\big|\mathcal{F}_{t_k}\Big]}\Big\}\nn\\
=&\E\Big[\trace{\big(\bar{\phi}(Y_{k}, r_k,r_{k+1})  \big)}\big|\mathcal{F}_{t_k}\Big]\t.
\end{align}
 Making use of \eqref{lyhf1} we arrive at \eqref{lyhf2}. The proof is complete.
 \end{proof}

In order to estimate   the $p$th moment of the numerical solution $Y(t)$,  we prepare  an elementary inequality.
 \begin{lemma}\la{L-Tay}
For any given $2k<p\leq 2(k+1)$ ($k$ is a nonnegative integer) the following inequality
 $$
 (1+u)^{\frac{p}{2} } \leq   1+ \frac{p}{2} u + \frac{p(p-2)}{8} u^2+   u^3 a_k(u)
$$
holds for any $u>-1$, where $a_k(u)$ represents a $k$th-order polynomial of $u$ which coefficients    depend  only on $p$.
   \end{lemma}
\begin{proof}
Applying the Taylor formula, for $k=0$, namely, $0<p\leq 2$,     we have
\be\label{Y2}
(1+u)^{\frac{p}{2} } \leq
\displaystyle 1+ \frac{p}{2} u + \frac{p(p-2)}{8} u^2+ \frac{p(p-2)(p-4)}{48} u^3,~~~~\forall u>-1.
\ee
Then using the above inequality, for $k=1$, namely, $2<p\leq 4$, we have
\begin{align*}
(1+u)^{\frac{p}{2} } & = (1+u)(1+u)^{\frac{p}{2}-1 } \\
&\leq (1+u)\displaystyle \Big(1+ \frac{p-2}{2} u + \frac{(p-2)(p-4)}{8} u^2+ \frac{(p-2)(p-4)(p-6)}{48} u^3\Big)\\
&=  1+ \frac{p}{2} u + \frac{p(p-2)}{8} u^2+ \frac{p(p-2)(p-4)}{48} u^3 + \frac{ (p-2)(p-4)(p-6)}{48} u^4
\end{align*}
for any $u>-1.$ Generally, for  $2k<p\leq 2(k+1)$ and  any $u>-1$, we have
\begin{align*}
(1+u)^{\frac{p}{2} }  & = (1+u)^k(1+u)^{\frac{p}{2}-k} \\
&\leq (1+u)^k\displaystyle \Big(1+ \frac{p-2k}{2} u + \frac{(p-2k)(p-2k-2)}{8} u^2\nn\\
&~~~~~~~~~~~~~~~~~~~~+ \frac{(p-2k)(p-2k-2)(p-2k-4)}{48} u^3\Big)\\
&=  1+ \frac{p}{2} u + \frac{p(p-2)}{8} u^2+   u^3 a_k(u).
\end{align*}
Therefore the desired result follows.
\end{proof}

\subsection{Moment estimate}\label{s3.1}
  Let us begin to establish the criterion on
  the $p$th moment boundedness  of the numerical solutions of Scheme \eqref{Y_0}.
   \begin{theorem}\la{T:C_1}
Under Assumption   \ref{a1},  Scheme \eqref{Y_0} has the property that
  \be\la{Y_1}
  \sup_{0<\t\leq 1}\left(\sup_{0\leq k\t\leq T}
   \E|Y_k|^{p }\right) \leq C
   \ee
   for any $p\in(0, \bar{p}]$ and any $T>0$, where $C$ is a constant independent of   the iteration order $k$ and the stepsize $\t$.
  \end{theorem}
\begin{proof}
For any integer $k\geq 0$, we have
 \begin{align}\la{E:H2}
\!\!\!(1+  \tilde{Y}^T_{k+1}Q_{r_{k+1}}\tilde{Y}_{k+1})^{\frac{p}{2}}
   &=\Big[1+ \big(Y_k+f(Y_k, r_k)\t+ g(Y_k, r_k)\t B_k\big)^TQ_{r_{k+1}}\nn\\
   &\times\big(Y_k+f(Y_k, r_k)\t+ g(Y_k, r_k)\t B_k\big)\Big]^{\frac{p}{2}}
   \!=\! \left(1+ Y^T_k Q_{r_{k}}Y_k \right)^{\frac{p}{2}}( 1+\zeta_k)^{\frac{p}{2}},
\end{align}
where
\begin{align*}
 \zeta_k =&(1+ Y^T_k Q_{r_{k}}Y_k )^{-1}\Big( Y^T_kQ_{r_{k+1}}Y_k -Y^T_kQ_{r_{k}}Y_k+ 2Y_k^TQ_{r_{k+1}} f(Y_k, r_k)\t\\
&~~~~+\t B_k^Tg^T(Y_k, r_k)Q_{r_{k+1}} g(Y_k, r_k)\t B_k+2Y_k^TQ_{r_{k+1}}g(Y_k, r_k)\t B_k\\
&~~~~+f^T(Y_k, r_k)Q_{r_{k+1}}f(Y_k, r_k)\t^2+2f^T(Y_k, r_k)Q_{r_{k+1}}g(Y_k, r_k)\t B_k \t\Big),
\end{align*}
and we can see that $\zeta_k>-1$. By the virtue of Lemma \ref{L-Tay}, without loss the generality we   prove \eqref{Y_1} only for $0<p\leq 2$. It follows from \eqref{Y2} and \eqref{E:H2} that
\begin{align} \la{Y_3}
  &  \E  \left[ \left(1+\tilde{Y}^T_{k+1}Q_{r_{k+1}}\tilde{Y}_{k+1}\right)^{\frac{p}{2}}\big|\mathcal{F}_{t_k}\right] \nn\\
   \leq & (1 + Y^T_kQ_{r_k}Y_k)^{\frac{p}{2}}\bigg[1 + \frac{p}{2} \E\big(\zeta_k\big|\mathcal{F}_{t_k}\big)
     + \frac{p(p-2)}{8} \E\big(\zeta_k^2\big|\mathcal{F}_{t_k}\big)
  +  \frac{p(p-2)(p-4)}{48} \E\big(\zeta_k^3\big|\mathcal{F}_{t_k}\big)\bigg].
\end{align}
  It follows from Lemma \ref{Z:1} and (\ref{Y_01})  that
\begin{align*}
\E\big[Y^T_kQ_{r_{k+1}}Y_k\big|\mathcal{F}_{t_k}\big]
=  Y^T_kQ_{r_{k}}Y_k+\sum_{j\in \mathbb{S}} Y_k^T Q_{j}Y_k [\gamma_{r_k j}\t+o(\t)]
\leq   Y^T_kQ_{r_{k}}Y_k+C|Y_k|^2\t,
\end{align*}
and
\begin{align*}
&\E\Big[2Y_k^TQ_{r_{k+1}} f(Y_k, r_k)\t+\t B_k^Tg^T(Y_k, r_k)Q_{r_{k+1}} g(Y_k, r_k)\t B_k\big|\mathcal{F}_{t_k}\Big]\nn\\
=&\psi(Y_k, r_k)\t
 +\sum_{j\in \mathbb{S}}\Big(2Y_k^T Q_{j}f(Y_k, r_k) +\trace{\left(g^T(Y_k, r_k)Q_{j}g(Y_k, r_k)\right)}\Big)\big(\gamma_{r_k j}\t^2 +o(\t^2)\big) \nn\\
\leq&\psi(Y_k, r_k)\t+ C     (1+ |Y_k|    )^2  h(\t ) \t^2
\leq \psi(Y_k, r_k)\t+   C (1+|Y_k|)^2 \t^{\frac{3}{2}}.
\end{align*}
Thus both of the above inequalities implies
\begin{align}\la{Y_4}
 \E\big[\zeta_k\big|\mathcal{F}_{t_k}\big]
\leq& \dis {(1+Y^T_kQ_{r_k}Y_k)^{-1 }} \Big(\psi(Y_k, r_k)\t+C|Y_k|^2\t
 + C|f (Y_k, r_k)|^2 \t^2 + C (1+|Y_k|)^2\t^{\frac{3}{2}}  \Big)\nn\\
  \leq & \dis {(1+Y^T_kQ_{r_k}Y_k)^{-1 }}  \psi(Y_k, r_k)\t+ C\t.
\end{align}
Using Lemma \ref{Z:1} and (\ref{Y_01}) again yields
\begin{align}\la{yhfY_5}
 \E\big[\zeta_k^2\big|\mathcal{F}_{t_k}\big]
 \geq &{(1+Y^T_kQ_{r_k}Y_k)^{-2 }}\bigg\{\E \Big[\E \Big[|2Y_k^TQ_{r_{k+1}}g(Y_k, r_k)\t B_k|^2\Big|\mathcal{G}_{t_k}\Big]\Big|\mathcal{F}_{t_k}\Big] \nn\\
 &+4\E \Big[\E \Big[\big(Y_k^TQ_{r_{k+1}}g(Y_k, r_k)\t B_k\big)^T
 \big(Y^T_kQ_{r_{k+1}}Y_k  -Y^T_kQ_{r_{k}}Y_k\nn\\
 & ~~~~~+ \t B_k^Tg^T(Y_k, r_k)Q_{r_{k+1}} g(Y_k, r_k)\t B_k +f^T(Y_k, r_k) Q_{r_{k+1}}f(Y_k, r_k)\t^2\nn\\
 & ~~~~~+2Y_k^TQ_{r_{k+1}} f(Y_k, r_k)\t+2f^T(Y_k, r_k)Q_{r_{k+1}}g(Y_k, r_k)\t B_k \t  \Big)\Big|\mathcal{G}_{t_k}\Big]\Big|\mathcal{F}_{t_k}\Big]\bigg\}\nn\\
\geq&(1+Y^T_kQ_{r_k}Y_k)^{-2 }\Big[4\t \E \Big(|Y_k^TQ_{r_{k+1}}g(Y_k, r_k)|^2\big|\mathcal{F}_{t_k}\Big)
  - C(1+|Y_k|)^4h^2(\t)   \t^2\Big]\nn\\
\geq &{4(1+Y^T_kQ_{r_k}Y_k)^{-2 }}  |Y_k^TQ_{r_{k}}g(Y_k, r_k)|^2\t -C \t .
\end{align}
Using  the properties
\begin{align}\la{L_2}
 \mathbb{E } \big(|\triangle B_k|^{2j} \big|\mathcal{G}_{t_k}\big)=C\t^j,~~ ~  \mathbb{E } \big((A\triangle B_k)^{2j-1}\big|\mathcal{G}_{t_k} \big)=0,~~  j=1, 2, \dots
\end{align}
for any $A\in \mathbb{R}^{1\times d}$ as well as Lemma \ref{Z:1}, we  deduce that
\begin{align*}
 \E\big[\zeta^3_k\big|\mathcal{F}_{t_k}\big]
=&{(1+Y^T_kQ_{r_k}Y_k)^{-3 }}\E\Big[\E\Big[\big(Y^T_kQ_{r_{k+1}}Y_k -Y^T_kQ_{r_{k}}Y_k+2Y_k^TQ_{r_{k+1}} f(Y_k, r_k)\t \nonumber\\
&+ \t B^T_k g^T(Y_k, r_k) Q_{r_{k+1}}g(Y_k, r_k)\t B_k+2Y_k^TQ_{r_{k+1}}g(Y_k, r_k)\t B_k\nn\\
&+ f^T(Y_k, r_k) Q_{r_{k+1}}f(Y_k, r_k) \t^2 +2f^T(Y_k, r_k)Q_{r_{k+1}}g(Y_k, r_k)\t B_k \t\big)^3\Big|\mathcal{G}_{t_k} \Big]\Big|\mathcal{F}_{t_k} \Big]\nonumber\\
 \leq & C{ (1+Y^T_kQ_{r_k}Y_k)^{-3 }}\bigg\{\E  \Big[|Y^T_kQ_{r_{k+1}}Y_k -Y^T_kQ_{r_{k}}Y_k|^3\big|\mathcal{F}_{t_k}\Big]
\nn\\
 &+  |Y_k|^3|f(Y_k, r_k)|^3\t^3+  | g(Y_k, r_k) |^6\t^3+ |f(Y_k, r_k)|^6\t^6 \nn\\
 &+|g(Y_k, r_k)|^2\t \Big( |Y_k|^2  +|f(Y_k, r_k)|^2\t^2\Big)
\bigg[\E \Big(\big|Y^T_kQ_{r_{k+1}}Y_k -Y^T_kQ_{r_{k}}Y_k\big|\Big|\mathcal{F}_{t_k}\Big) \nn\\
&~~~~~~~~~~~~~~~ +  \t  \Big(|Y_k|| f(Y_k, r_k)|+ |g(Y_k, r_k)|^2
 +|f(Y_k, r_k)|^2 \t\Big)\bigg] \bigg\}.
\end{align*}
By virtue of Lemma  \ref{Z:1}, for any integer $l\geq 1$,
\begin{align*}
  \E\Big[|Y^T_kQ_{r_{k+1}}Y_k -Y^T_kQ_{r_{k}}Y_k|^{l}\big|\mathcal{F}_{t_k}\Big]
= &\sum_{j\in \Se} |Y^T_kQ_{j}Y_k -Y^T_kQ_{r_{k}}Y_k|^l I_{\{ r_k \neq j\}}\Big(\gamma_{r_k j}\t+o(\t)\Big)\nn\\
\leq& C(1+| Y_k|)^{2l} \t.
\end{align*}
Making use of  the above inequality and (\ref{Y_01})  yields
\begin{align}\la{Y_6}
 \E\big[\zeta^3_k\big|\mathcal{F}_{t_k}\big]
\leq & C{ (1+Y^T_kQ_{r_k}Y_k)^{-3 }}\bigg\{(1+| Y_k|)^6 \Big[\t+ h^3(\t)\t^3\big(2+ h^3(\t)\t^3\big) \nn\\
 &+ h(\t)\t^2\big(1+h^2(\t)\t^2\big)\big(1+2h(\t)+h^2(\t)\t\big)
 \Big]\bigg\}
 \leq    C\t.
\end{align}
Similarly, we  can also prove that for any integer $l >3$, $ \E\big[|\zeta_k|^{l}|\F_{t_k}\big]\leq C\t.$
Combining (\ref{Y_3})-(\ref{yhfY_5}) and (\ref{Y_6}),  using   (\ref{eq02}) and \eqref{YH_2},  we obtain that
\begin{align*}
&\E\Big[ (1+\tilde{Y}^T_{k+1}Q_{r_{k+1}}\tilde{Y}_{k+1})^{\frac{p}{2}}\big|\mathcal{F}_{t_k}\Big]\nn\\
\leq  & \left(1+Y^T_kQ_{r_k}Y_k\right)^{\frac{p}{2}}  \bigg\{1+C\t
 +  \frac{p\t}{2}\left[\frac{(1+Y^T_kQ_{r_k}Y_k)\psi(Y_k, r_k)+(p-2)|Y_k^TQ_{r_{k}}g(Y_k, r_k)|^2}{(1+Y^T_kQ_{r_k}Y_k)^{2}}\right]  \bigg\}\nn\\
 \leq & \left(1+Y^T_kQ_{r_k}Y_k\right)^{\frac{p}{2}}  \left(1
 +C\t\right)+C\t
\end{align*}
for any integer $k\geq 0$.
The truncation property of Scheme \eqref{Y_0}
\begin{align*}
 Y^T_{k}Q_{r_{k}}Y_{k}  =\big(\pi_{\t}^{r_k}(\tilde{Y}_k)\big)^TQ_{r_{k}}\pi_{\t}^{r_k}(\tilde{Y}_k)
  = \left(\frac{ |\tilde{Y}_{k }|\wedge \f^{-1}_{r_k}(h(\t))  }{|\tilde{Y}_{k }| }  \right)^2\tilde{Y}^T_{k}Q_{r_{k}}\tilde{Y}_{k}\leq  \tilde{Y}^T_{k}Q_{r_{k}}\tilde{Y}_{k},
\end{align*}
  implies that
\begin{align*}
\E\Big[ (1+Y^T_{k+1}Q_{r_{k+1}}Y_{k+1})^{\frac{p}{2}}\big|\mathcal{F}_{t_k}\Big]
   \leq  &  (1+Y^T_kQ_{r_k}Y_k)^{\frac{p}{2}} \left(1+C\t \right)+C\t.
\end{align*}
Repeating this procedure we obtain
\begin{align*}
 \E\lf(  (1+Y^T_{k}Q_{r_{k}}Y_{k})^{\frac{p}{2}}|\F_{0} \rt)
 \leq& \left(1 +C\t  \right)^{k}\left(1+x^T_{0}Q_{\ell}x_{0}\right)^{\frac{p}{2}}
 +
   C\t \sum_{i=0}^{k-1}\left(1 +C\t  \right)^i.
\end{align*}
Taking expectations on both sides, for any integer $k$ satisfying $0\leq k\t\leq T$, then we have
\begin{align*} 
 \E\lf( (1+Y^T_{k}Q_{r_{k}}Y_{k})^{\frac{p}{2}}\rt)
 \leq C \left(1 +C\t  \right)^{k}\leq  C\exp\left(  Ck\t\right) \leq  C\exp\left(  CT\right) .
\end{align*}
 Therefore  the desired result follows.
\end{proof}

In order to establish  the strong convergence theory of  Scheme \eqref{Y_0}, we give the following lemma.
\begin{lemma}\la{L:C_1}
  Under Assumption  \ref{a1}, define
\begin{align}\la{3.18}
\rho_{\t}=:\inf\{t\geq 0:|\tilde{Y}(t)|\geq  \varphi^{-1}_{\bar{r}(t)}(h(\t))\},
\end{align}
then for any   $T>0$,
\begin{align}\la{3.19}
 \PP {\{\rho_{\t} \leq T\}}\leq \frac{C }{ \big( \hat{\f}^{-1}(h(\t ))\big)^{\bar{p}} }
\end{align}
for all $p\in(0, \bar{p}]$, where $C$ is  a constant independent of      $k$ and  $\t$.
  \end{lemma}
\begin{proof}
 Define~$\beta =:\inf\left\{k\geq 0:|\tilde{Y}_k|\geq \varphi^{-1}_{r_k}(h(\t))\right\},$ then
$
\rho_{\t}=\t \beta .
$
Obviously,    $\rho_{\t}$ and $\beta$ are $\F_{t}$,   $\F_{t_k}$ stopping time, respectively.
For $\omega\in \{\beta\geq k+1\}$, we have  $Y_k=\tilde{Y}_{k}$ and
$$
\tilde{Y}_{(k+1)\wedge \beta}=\tilde{Y}_{k+1},~~~ \tilde{Y}^T_{(k+1)\wedge \beta}Q_{r_{(k+1)\wedge \beta}}\tilde{Y}_{(k+1)\wedge \beta}= \tilde{Y}^T_{k+1}Q_{r_{k+1}}\tilde{Y}_{k+1}.
$$
On the other hand, for $\omega\in \{\beta< k+1\}$, we have $\beta\leq k$ and hence
$$
\tilde{Y}_{(k+1)\wedge \beta}=\tilde{Y}_{\beta}=\tilde{Y}_{k\wedge\beta}, ~~~~~~~ \tilde{Y}^T_{(k+1)\wedge \beta}Q_{r_{(k+1)\wedge \beta}}\tilde{Y}_{(k+1)\wedge \beta} = \tilde{Y}^T_{k\wedge\beta}Q_{r_{k\wedge\beta}}\tilde{Y}_{k\wedge\beta}.
$$
Therefore, we derive from (\ref{Y_0}) that for any integer $k\geq 0$,
\begin{align*}
 \tilde{Y}_{(k+1)\wedge \beta}
  =\tilde{Y}_{k\wedge \beta}+\left[f(\tilde{Y}_{k}, r_k)\t+g(\tilde{Y}_{k}, r_k)\t B_k\right]I_{[[0, \beta]]}(k+1).
\end{align*}
Then
\begin{align}\la{FF1}
 \left(1+ \tilde{Y}^T_{(k+1)\wedge \beta}Q_{r_{(k+1)\wedge \beta}}\tilde{Y}_{(k+1)\wedge \beta} \right)^{\frac{p}{2}}
  = \left(1+ \tilde{Y}^T_{k\wedge \beta}Q_{r_{k\wedge \beta}}\tilde{Y}_{k\wedge \beta}\right)^{\frac{p}{2}}\left(1+\tilde{\zeta}_kI_{[[0, \beta]]}(k+1)\right)^{\frac{p}{2}},
\end{align}
where
\begin{align*}
\tilde{\zeta}_k& =\Big(1+\tilde{Y}^T_kQ_{r_k}\tilde{Y}_k\Big)^{-1}\Big( \tilde{Y}^T_kQ_{r_{k+1}}\tilde{Y}_k -\tilde{Y}^T_kQ_{r_{k}}\tilde{Y}_k+2\tilde{Y}_k^TQ_{r_{k+1}} f(\tilde{Y}_k, r_k)\t\\
  & ~~~~+2\tilde{Y}_k^TQ_{r_{k+1}}g(\tilde{Y}_k, r_k)\t B_k+\t B^T_kg^T(\tilde{Y}_k, r_k) Q_{r_{k+1}}g(\tilde{Y}_k, r_k)\t B_k\nn\\
  & ~~~~+ f^T(\tilde{Y}_k, r_k)Q_{r_{k+1}}f(\tilde{Y}_k, r_k) \t^2+2f^T(\tilde{Y}_k, r_k)Q_{r_{k+1}}g(\tilde{Y}_k, r_k)\t B_k \t \Big).
  \end{align*}
By the virtue of Lemma \ref{L-Tay}, without loss the generality we prove the required result only for $0<p\leq 2$. It follows from (\ref{Y2}) and (\ref{FF1}) that
\begin{align} \la{FF_3}
  &  \E \Big[ \big(1+\tilde{Y}^T_{(k+1)\wedge \beta}Q_{r_{(k+1)\wedge \beta}}\tilde{Y}_{(k+1)\wedge \beta}\big)^{\frac{p}{2}} \big|\mathcal{F}_{t_{k\wedge \beta}}\Big] \nonumber\\
 \leq & \big(1+\tilde{Y}^T_{k\wedge \beta}Q_{r_{k\wedge\beta}}\tilde{Y}_{k\wedge \beta}\big)^{\frac{p}{2}}\bigg\{1+ \frac{p}{2}
 \E \big[\tilde{\zeta}_kI_{[[0, \beta]]}(k+1)\big|\mathcal{F}_{t_{k\wedge \beta}}\big]
 + \frac{p(p-2)}{8}\E \big[\tilde{\zeta}_k^2I_{[[0, \beta]]}(k+1)\big|\mathcal{F}_{t_{k\wedge \beta}}\big]\nn\\
  &
~~~~~~~~~~~~~~~~~~~~~~~~~~~~~~~~~~+ \frac{p(p-2)(p-4)}{48} \E \big[\tilde{\zeta}_k^3I_{[[0, \beta]]}(k+1)\big|\mathcal{F}_{t_{k\wedge \beta}}\big]\bigg\}.
\end{align}
By Lemma \ref{Z:1} one oberserves that
\begin{align}\la{Leq1}
&\E \Big[\big(\tilde{Y}^T_kQ_{r_{k+1}}\tilde{Y}_k\big) I_{[[0, \beta]]}(k+1)\big|\mathcal{F}_{t_{k\wedge \beta}} \Big]\nn\\
=&\E \Big[\Big( \tilde{Y}^T_kQ_{r_{k}}\tilde{Y}_k+\sum_{j\in \mathbb{S}} \tilde{Y}_k^T Q_{j}\tilde{Y}_k [\gamma_{r_k j}\t+o(\t)]\Big)I_{[[0, \beta]]}(k+1)\big|\mathcal{F}_{t_{k\wedge \beta}} \Big]\nn\\
\leq& \E\Big[\Big(\tilde{Y}^T_kQ_{r_{k}}\tilde{Y}_k +C|\tilde{Y}_k|^2\t\Big)I_{[[0, \beta]]}(k+1)\big|\mathcal{F}_{t_{k\wedge \beta}} \Big],
\end{align}
Note that
$
\t B_kI_{[[0, \beta]]}(k+1)=B(t_{(k+1)\wedge \beta})-B(t_{k\wedge \beta}).
$
Since $B(t)$ is a continuous local martingale, by the virtue of the Doob martingale stopping time theorem, 
  we know that $
\E\big[\t B_kI_{[[0, \beta]]}(k+1)\big|\mathcal{G}_{t_{k\wedge \beta}}\big]=0
$ 
and
 $
\E\big[\t B_k\t B_k^T I_{[[0, \beta]]}(k+1)\big|\mathcal{G}_{t_{k\wedge \beta}}\big]=\t \mathbb{I}_d \E\big[I_{[[0, \beta]]}(k+1)\big|\mathcal{G}_{t_{k\wedge \beta}}\big].
$
Hence
\begin{align*}
&\E\Big[\big(\t B_k^Tg^T(\tilde{Y}_k, r_k)Q_{r_{k+1}} g(\tilde{Y}_k, r_k)\t B_k\big)I_{[[0, \beta]]}(k+1)\big|\mathcal{F}_{t_{k\wedge \beta}} \Big]\nn\\
=&\E\Big\{\trace\Big[g^T(\tilde{Y}_{k\wedge \beta}, r_{k\wedge \beta})Q_{r_{(k+1)\wedge \beta}} g(\tilde{Y}_{k\wedge \beta}, r_{k\wedge \beta}) \E\Big(\t B_k\t B_k^T I_{[[0, \beta]]}(k+1)\big|\mathcal{G}_{t_{k\wedge \beta}}\Big)\Big]\Big|\mathcal{F}_{t_{k\wedge \beta}}\Big\}\nn\\
=&\t \E \Big[\trace\big( g^T(\tilde{Y}_k, r_k)Q_{r_{k+1}} g(\tilde{Y}_k, r_k)\big)  I_{[[0, \beta]]}(k+1)\big|\mathcal{F}_{t_{k\wedge \beta}} \Big].
\end{align*}
The above equality together with  Lemma \ref{Z:1} implies
\begin{align*}
&\E \Big[\big(2\tilde{Y}_k^TQ_{r_{k+1}} f(\tilde{Y}_k, r_k)\t+\t B_k^Tg^T(\tilde{Y}_k, r_k)Q_{r_{k+1}} g(\tilde{Y}_k, r_k)\t B_k\big)I_{[[0, \beta]]}(k+1)\big|\mathcal{F}_{t_{k\wedge \beta}} \Big]\nn\\
=&\t \E\Big[\Big(2\tilde{Y}_k^TQ_{r_{k+1}} f(\tilde{Y}_k, r_k) +\trace\big( g^T(\tilde{Y}_k, r_k)Q_{r_{k+1}} g(\tilde{Y}_k, r_k)\big)\Big)I_{[[0, \beta]]}(k+1)\big|\mathcal{F}_{t_{k\wedge \beta}}  \Big]\nn\\
\leq&\t\E\Big[\Big(\psi(\tilde{Y}_k, r_k) +   C (1+|\tilde{Y}_k|)^2h(\t )\t\Big)I_{[[0, \beta]]}(k+1)\big|\mathcal{F}_{t_{k\wedge \beta}} \Big].
\end{align*}
Thus  the above inequality and \eqref{Leq1} imply
\begin{align}\la{FF_4}
&\E \Big[\tilde{\zeta}_kI_{[[0, \beta]]}(k+1)\big|\mathcal{F}_{t_{k\wedge \beta}} \Big]=\E\Big[\E\Big(\tilde{\zeta}_kI_{[[0, \beta]]}(k+1)\big|\mathcal{G}_{t_{k\wedge \beta}}\Big)\Big|\mathcal{F}_{t_{k\wedge \beta}} \Big]\nn\\
\leq& \dis \E\Big[{(1+\tilde{Y}^T_kQ_{r_k}\tilde{Y}_k)^{-1 }}\Big( \psi(\tilde{Y}_k, r_k)\t+C|\tilde{Y}_k|^2\t+C|f(\tilde{Y}_k, r_k)|^2 \t^2\nn\\
&~~~~~~~~~~~~~~~~~~~~~~~~~~~~~~~~+C(1+|\tilde{Y}_k|)^2h(\t)\t^2  \Big)I_{[[0, \beta]]}(k+1)\big|\mathcal{F}_{t_{k\wedge \beta}} \Big]\nn\\
\leq& \dis\E\Big[\Big( {(1+\tilde{Y}^T_kQ_{r_k}\tilde{Y}_k)^{-1 }}\psi(\tilde{Y}_k, r_k)\t+C \t  \Big)I_{[[0, \beta]]}(k+1)\big|\mathcal{F}_{t_{k\wedge \beta}} \Big].
\end{align}
Using the  techniques in the proof of  Theorem \ref{T:C_1}, we  show that
\begin{align}\la{FF_5}
&\E \Big[\tilde{\zeta}^2_kI_{[[0, \beta]]}(k+1)\big|\mathcal{F}_{t_{k\wedge \beta}} \Big]=\E\Big[\E\Big(\tilde{\zeta}^2_kI_{[[0, \beta]]}(k+1)\big|\mathcal{G}_{t_{k\wedge \beta}}\Big)\Big|\mathcal{F}_{t_{k\wedge \beta}} \Big]\nn\\
\geq &\E\Big[\Big(4(1+\tilde{Y}^T_kQ_{r_k}\tilde{Y}_k)^{-2 } |\tilde{Y}_k^TQ_{r_{k}}g(\tilde{Y}_k, r_k)|^2\t -C\t \Big)I_{[[0, \beta]]}(k+1)\big|\mathcal{F}_{t_{k\wedge \beta}} \Big],
\end{align}
and
 \begin{align}\la{FF_6}
  \E\Big[\tilde{\zeta}^3_kI_{[[0, \beta]]}(k+1)\big|\mathcal{F}_{t_{k\wedge \beta}} \Big]
  =&\E\Big[\E\Big(\tilde{\zeta}^3_kI_{[[0, \beta]]}(k+1)\big|\mathcal{G}_{t_{k\wedge \beta}}\Big)\Big|\mathcal{F}_{t_{k\wedge \beta}} \Big]\nn\\
  \leq &   C\t \E\Big[I_{[[0, \beta]]}(k+1)\big|\mathcal{F}_{t_{k\wedge \beta}} \Big].
\end{align}
Similarly, we  can also show that for any integer $l>3$,
 \begin{align*}
 \E\Big[ |\tilde{ \zeta}_k|^{l}I_{[[0, \beta]]}(k+1)\big|\mathcal{F}_{t_{k\wedge \beta}} \Big]\leq C\t \E\Big[I_{[[0, \beta]]}(k+1)\big|\mathcal{F}_{t_{k\wedge \beta}} \Big].
 \end{align*}
 Combining (\ref{FF_3}) and \eqref{FF_4}-(\ref{FF_6}),  using   (\ref{eq02}) and \eqref{YH_2}, for any integer $k\geq 0$,
\begin{align*}
& \E \Big[\big(1+ \tilde{Y}^T_{(k+1)\wedge \beta}Q_{r_{(k+1)\wedge \beta}} \tilde{Y}_{(k+1)\wedge \beta}\big)^{\frac{p}{2}} \big|\mathcal{F}_{t_{k\wedge \beta}} \Big] \nonumber\\
   \leq  &  \bigg\{1
+  \E\Big[\Big(\frac{p\t}{2} \frac{(1+\tilde{Y}^T_kQ_{r_k}\tilde{Y}_k)\psi(\tilde{Y}_k, r_k)+(p-2)|\tilde{Y}_k^TQ_{r_{k}}g(\tilde{Y}_k, r_k)|^2}{(1+\tilde{Y}^T_kQ_{r_k}\tilde{Y}_k)^{2}} \\
  & ~~~~~~~~~~~~~ +C\t  \Big)I_{[[0, \beta]]}(k+1)\Big|\mathcal{F}_{t_{k\wedge \beta}} \Big]\bigg\}\Big(1+ \tilde{Y}^T_{k\wedge \beta}Q_{r_{k\wedge \beta}}\tilde{Y}_{k\wedge \beta}\Big)^{\frac{p}{2}} \nn\\
 \leq  & \Big(1+ \tilde{Y}^T_{k\wedge \beta}Q_{r_{k\wedge \beta}}\tilde{Y}_{k\wedge \beta}\Big)^{\frac{p}{2}}  \bigg\{1
+  \E\Big[\Big(\frac{p\t}{2} \big(\alpha_{r_{k}}+\frac{\kappa_{r_{k}}}{2}\big)
   +C\t  \Big)I_{[[0, \beta]]}(k+1)\Big|\mathcal{F}_{t_{k\wedge \beta}} \Big]\bigg\}\nn\\
   &+\frac{p\t}{2}\big(1+ \tilde{Y}^T_{k\wedge \beta}Q_{r_{k\wedge \beta}}\tilde{Y}_{k\wedge \beta}\big)^{\frac{p}{2}-2}\E\Big[C_{r_{k}}I_{[[0, \beta]]}(k+1)\Big|\mathcal{F}_{t_{k\wedge \beta}} \Big]\nn\\
 \leq  & \big(1+ \tilde{Y}^T_{k\wedge \beta}Q_{r_{k\wedge \beta}}\tilde{Y}_{k\wedge \beta}\big)^{\frac{p}{2}}  \big(1
+  C\t  \big)
 +C\t.
\end{align*}
Repeating this procedure we obtain
$$
  \E\Big[\big(1+ \tilde{Y}^T_{k\wedge \beta}Q_{r_{k\wedge \beta}} \tilde{Y}_{k\wedge \beta}\big)^{\frac{p}{2}} \Big]
      \leq      (1+ C\t )^{k} \left(1+ x^T_{0}Q_{\ell}x_{0}\right)^{\frac{p}{2}} +C\t\sum_{i=0}^{k-1}(1+ C\t )^{i}
      \leq C(1+ C\t )^{k}
$$
for any integer $k$ satisfying $0\leq k\t\leq T$. 
Therefore the required assertion follows from that
$$
    \big( \hat{\f}^{-1}(h(\t ))\big)^{\bar{p}}  \PP {\{\rho_{\t} \leq T\}}\leq \E\big[|\tilde{Y}(T\wedge \rho_{\t}) |^p\big]=\E\big[ |\tilde{Y}_{ \lfloor\frac{T}{\t}\rfloor \wedge \beta}  |^p\big]\leq 
  C\exp\left(CT\right),
$$
where $\lfloor\frac{T}{\t}\rfloor$ represents the integer part of $T/\t$.  The proof is complete.
\end{proof}

\subsection{Strong convergence}\label{s3.2}
 In this subsection, we give   the  convergence result of Scheme \eqref{Y_0}.
\begin{theorem}\la{T:C_2}
If   Assumption   \ref{a1} holds with $\bar{p}> 0$, then for any $q\in (0, \bar{p})$,
\be \la{F_0}
\lim_{\t\rightarrow 0} \E |Y(T)-X(T)|^q=0,~~~~\forall ~T> 0.
\ee
\end{theorem}
\begin{proof}
Let $\tau_N$ and $\r_{\t}$ be defined as before.
Define
$$
\theta_{N, \t}=\tau_N \wedge \r_{\t},~~ e_{\t }(T)=X(T)- {Y}(T).
$$
For any $l>0$, using the Young inequality  we obtain that
\begin{align}\la{F_1}
 \E|e_{\t }(T)|^q
 =& \E\lf(|e_{\t }(T)|^q I_{\{\theta_{N, \t} \geq T\}}\rt)
+ \E\lf(|e_{\t }(T)|^q I_{\{\theta_{N, \t} \leq T\}}\rt) \nonumber \\
 \leq & \E\lf(|e_{\t }(T)|^q I_{\{\theta_{N, \t} \geq T\}}\rt)
+ \frac{q l}{\bar{p}}\E\lf(|e_{\t }(T)|^{\bar{p}} \rt)+ \frac{\bar{p}-q}{\bar{p} l^{q/(\bar{p}-q)}} \PP {\{\theta_{N, \t} \leq T\}}.
\end{align}
 It follows from    Theorem  \ref{T:1} and  Theorem \ref{T:C_1} that
$$
 \E |e_{\t }(T)|^{\bar{p}}  \leq  2^{{\bar{p}}} \E |X(T)|^{\bar{p}}
+  2^{{\bar{p}}}\E |Y(T)|^{\bar{p}} \leq    C.
$$
Now let $\e>0$ be arbitrary.  Choose $l>0 $ small sufficiently such that $ {Cq l}/{\bar{p}} \leq \e/3$, then we have
\be\la{F_2}
\frac{q l}{\bar{p}}\E\lf(|e_{\t }(T)|^{\bar{p}}  \rt)\leq \frac{\e}{3}.
\ee
 Then choose $N>1$ large sufficiently such that
 $\frac{C({\bar{p}}-q)}{N^{\bar{p}} {\bar{p}} l^{q/({\bar{p}}-q)}}\leq   \frac{\e}{6}.$
Choose $\t^{*}\in (0, 1]$ small sufficiently such that
$
 \hat{\f}^{-1}(h(\t^*))\geq N.
$
Then for any $\t\in(0, \t^{*}]$,  it follows from  Lemma   \ref{L:C_1}   that
$  \PP {\{\rho_{\t} \leq T\}}\leq   \frac{C }{N^{\bar{p}} }.$
This together with \eqref{eq0-1} implies
\be\la{F_3}
\frac{ {\bar{p}}-q }{   {\bar{p}} l^{q/({\bar{p}}-q)}} \PP {\{\theta_{N, \t} \leq T\}}\leq   \frac{ {\bar{p}}-q }{   {\bar{p}} l^{q/({\bar{p}}-q)}} \Big(   \PP {\{\tau_{N } \leq T\}}+ \PP {\{\r_{  \t} \leq T\}}\Big)\leq \frac{2 C({\bar{p}}-q)}{ N^{\bar{p}} {\bar{p}} l^{q/({\bar{p}}-q)}} \leq \frac{\e}{3}  .
\ee
Combining (\ref{F_1}), (\ref{F_2}) and (\ref{F_3}), we know that for the chosen $N$ and all $\t\in (0, \t^{*}]$,
$$\E|e_{\t }(T)|^q \leq \E\lf(| {e}_{\t }(T)|^q I_{\{\theta_{N, \t} \geq T\}}\rt)+\frac{2\e}{3}. $$
If we can show that
\be\la{F_4}\lim_{\t\rightarrow 0}\E\lf(| {e}_{\t }(T)|^q I_{\{\theta_{N, \t} \geq T\}}\rt)=0,\ee
the required assertion  follows.
For this purpose  we define the truncated functions
$$
f_N(x, i)=f\left(\lf(|x|\wedge N \rt) \frac{x}{|x|}, i\right),~~\hbox{and}~~g_N(x, i)=g\left(\lf(|x|\wedge N \rt) \frac{x}{|x|}, i\right)
 $$
for any $(x, i)\in \RR^n\times\mathbb{S}$. Consider the truncated SDS
\be\la{F_5}
dz(t) =f_N(z(t), r(t))dt +g_N(z(t), r(t))dB(t)
\ee
with  the initial value $z(0)=x_0$ and $r(0)=\ell$. For the chosen $N$,   Assumption \ref{a1} implies that $f_N(\cdot, \cdot)$ and $g_N(\cdot, \cdot)$ are globally Lipschitz continuous with the Lipschitz constant $C_N$. Therefore, SDS (\ref{F_5}) has a unique regular solution $z(t)$ satisfying
\be\la{F_6}
X(t\wedge \tau_N)=z(t\wedge \tau_N)~~\hbox{a.s.}~~~~\forall~ t\geq 0.
\ee
On the other hand, for each $\t\in (0, \t^{*}]$, we apply the EM method to  (\ref{F_5}) and  denote by $u(t)$ the piecewise constant EM solution  (see \cite{Mao06}) which has the property
\be\la{F_7}
\E\lf(\sup_{0\leq t\leq T}|z(t)-u(t)|^q \rt)\leq C_N\t^{q/2}, ~~~~\forall ~T> 0.
\ee
Due to $ \f_i^{-1}(h(\t )) \geq N  $, we have
$Y(t\wedge \theta_{N, \t})=u(t\wedge \theta_{N, \t})$ a.s.
 This together with (\ref{F_6}) implies
\begin{align*}
\E\lf(| {e}_{\t }(T)|^q I_{\{\theta_{N, \t} \geq T\}}\rt)
 =\E\lf(| {e}_{\t }(T\wedge \theta_{N, \t})|^q I_{\{\theta_{N, \t} \geq T\}}\rt)
 \leq   \E\Big(\sup_{0\leq t\leq T\wedge \theta_{N, \t}}| z(t)- u(t )|^q   \Big).
\end{align*}
Thus the desired assertion (\ref{F_4}) follows from (\ref{F_7}). The proof is  complete.
\end{proof}

\subsection{Convergence Rate}\la{cov-rates}
In this subsection our aim is to establish the rate of convergence. 
The rate is optimal similar to the standard results of the  explicit EM scheme for SDSs with globally Lipschitz  coefficients, see \cite[p.115]{Mao06}. To estimate the rate, we need somewhat stronger conditions compared with the convergence alone, which are stated as follows.

\begin{assp}\label{a6}
For some  $\tilde{p}>2 $, there exist  positive constants $L_i$ and $l$ such that
   \begin{align}
   &  2(x-y)^T \big(f(x, i)-f(y, i)\big)     +( \tilde{p}-1 )|g(x, i)-g(y, i)|^2
        \leq  L_i| x-y |^2, \label{cond-1}\\
        &  | f(x, i)-f(y, i)|\leq L_i\big(1+|x|^l+|y|^l\big)|x-y|,   ~~~~ ~\forall x, y \in \RR^d,~~i \in \mathbb{S}.  \label{cond-6}
   \end{align}
  \end{assp}

\begin{rem}  {\rm
One observes that if Assumption \ref{a6} holds, then
 \be\label{cond-5}
 |g(x, i)-g(y, i)|^2\leq  C_i(1+|x|^l+|y|^l)|x-y|^2.
\ee
One also knows that
\begin{align}\label{cond-3}
|f(x, i)|\leq | f(x, i)-f(\mathbf{0}, i)|+|f(\mathbf{0}, i)|\leq L_i(1+|x|^l )|x | +|f(\mathbf{0}, i)|\leq C_i\big(1+|x|^{l+1}\big),
\end{align}
and by Young's inequality, \begin{align}\label{cond-4}
|g(x, i)|\leq   C_i\big(| x  |^2 +|x|^{l+2}\big)^{1/2}+|g(\mathbf{0}, i)|\leq C_i\big(1+|x|^{l/2+1}\big).
\end{align}}
\end{rem}

\begin{rem}  {
 Under Assumption \ref{a6}, we may define $\varphi_i(u)=C_i(1+u^{l})$  for any $u>0$ 
 in \eqref{e21}, then $\varphi^{-1}_i(u)=\big(u/C_i-1\big)^{1/l}$ for all $u> C_i$. In order to obtain the rate, we specify $h(\t)=K\t^{-\tau}$ for any $\t\in (0, \t^{*}]$, where $\tau\in (0, 1/2]$ will be specified in Lemma \ref{lemma+4}. Thus, $\pi_{\t}^i (x)=\big(|x|\wedge (K\t^{-\tau}/C_i-1)^{1/l}\big)x/|x|$ for any $x\in \mathbb{R}^n, i\in \SS$.
 }
\end{rem}

Making use of Scheme \eqref{Y_0}, we define an auxiliary approximation process by
\begin{align}\label{cond-7}
\bar{Y}(t)
&=Y_k+ f(Y_k, r_k)( t-t_k )+   g(Y_k, r_k)( B(t)-B(t_k) ) , ~~~~\forall t\in [t_k, t_{k+1}).
\end{align}
Note that 
$  \bar{Y}(t_k) =Y(t_k)=Y_k$, that is $\bar{Y}(t)$
and $Y(t)$ coincide with the discrete solution at the grid points.

  \begin{lemma}\la{lemma+1}
   If Assumptions \ref{a1} and \ref{a6} hold, 
   for any $p\in (0,   \bar{p}/(l+1)] $,  the process defined by (\ref{cond-7}) has the property
  \be\la{cond-8}
   \sup_{0\leq t\leq T}  \E\big(|\bar{Y}(t)-Y(t)|^{p} \big)\leq C \t^{\frac{p}{2} },~~~~  \forall T> 0,
  \ee  where $C$ is a positive constant independent of $\t$.
  \end{lemma}
\begin{proof}  For any $t\in [0, T]$, there is a nonnegative  integer $k$ such that $t\in [t_k, t_{k+1})$. Then,
 \begin{align*}
 \E\big(|\bar{Y}(t)-Y(t)|^{p} \big)&=\E\big(|\bar{Y}(t)-Y(t_k)|^{p} \big)
 \leq  2^{p} \E\big( |  f(Y_k, r_k) |^{p} \big) \t^{p} + 2^{p}\E\big( |  g(Y_k, r_k)|^{p}  |B(t)-B(t_k)|^{p}\big) \\
 &\leq  C\lf( \E |  f(Y_k, r_k) |^{p}\t^{p} +  \E |  g(Y_k, r_k)|^{p} \t^{\frac{p}{2}} \rt).
  \end{align*}
 Due to   (\ref{cond-3}), (\ref{cond-4}) and  Theorem \ref{T:C_1},
  \begin{align*}
 \E\big(|\bar{Y}(t)-Y(t)|^{{p} } \big)
 &\leq C \E \big( 1+ |Y_k|^{l+1}\big)^{p}  \t^{p} + C \E \big( 1+ |Y_k|^{l/2+1}\big)^{p} \t^{\frac{{p}}{2}}\nonumber\\
 &\leq C \t^{\frac{{p}}{2}}+ C \big(\E    |Y_k|^{\bar{p}}\big)^{\frac{(l+1)p}{{\bar{p}}}} \t^{{p}}+ C \big(\E    |Y_k|^{\bar{p}}\big)^{\frac{(l+2)p}{2{\bar{p}}}} \t^{\frac{{p}}{2}}
\leq  C   \t^{\frac{{p}}{2} } .
 \end{align*}
 The required assertion follows.  \end{proof}

Using the techniques in the proof  of Theorem \ref{T:C_1} and Lemma \ref{L:C_1} yields  the following lemmas.
\begin{lemma}\la{lemma+2}
Under Assumption  \ref{a1},  the    process defined by (\ref{cond-7}) has the property that
  \be\la{cond-9}
     \sup_{0<\t \leq 1}\sup_{0\leq t\leq T}\E|\bar{Y}(t)|^{p}\leq C,~~~~\forall T>0,
  \ee
 for all $p\in (0, \bar{p}]$,  where $C$ is a positive constant independent of $\t$.
  \end{lemma}

\begin{lemma}\la{lemma+3}
Let Assumption  \ref{a1}  hold.  Define
 \be\la{cond-10}
\bar{ \rho}_{\t} := \inf \{ t\geq 0: | {\bar{Y}}(t)|\geq  \varphi^{-1}_{\bar{r}(t)}(h(\t))\}.
 \ee
Then for any $T>0$,
  \be\la{cond-11}
   \PP {\{\bar{\rho}_{\t} \leq T\}}\leq \frac{C }{ \big( \hat{\f}^{-1}(h(\t ))\big)^{\bar{p}} }
  \ee
for all $p\in (0, \bar{p}]$,  where $C$ is a positive constant independent of $\t$.
  \end{lemma}

\begin{lemma}\la{lemma+4}
 If Assumptions  \ref{a1} and  \ref{a6} hold with  $4(l+1)\leq \bar{p}$,   the process defined by (\ref{cond-7})  with $\tau\in [ l/(\bar{p}-2), 1/2]$  has the property that
\begin{align*}
\E \big| \bar{Y}(T)-X(T)\big|^2\leq C\t,~~~~\forall ~T>0.
\end{align*}
\end{lemma}
\begin{proof}
Define $\bar{\theta}_{ \t}=\tau_{ \hat{\varphi}^{-1}(h(\t ))} \wedge \rho_{\t}\wedge \bar{\rho}_{\t}$, $  \Omega_1:= \{\omega:~\bar{\theta}_{  \t} > T\}$, $\bar{e}(t)= \bar{Y}(t)-X(t),$ for  any $t\in[0, T]$,
where  $\tau_N$, $\rho_{\t}$ and $\bar{\rho}_{\t}$ are defined by (\ref{eq0-1}), (\ref{3.18}) and (\ref{cond-10}),  respectively.
Using the Young inequality,   we have
\begin{align}\la{cond-13}
\E| {\bar{e}} (T)|^2 =& \E\lf(|\bar{e} (T)|^2 I_{\Omega_1}\rt)
+ \E\lf(|\bar{e} (T)|^2 I_{\Omega_1^c}\rt) \nonumber \\
 \leq &\E\lf(| {\bar{e}} (T)|^2 I_{\Omega_1}\rt)
+ \frac{2\t}{\bar{p}}\E\lf(|\bar{e} (T)|^{\bar{p}}  \rt)+ \frac{\bar{p}-2}{\bar{p}\t^{2/(\bar{p}-2)}} \PP(\Omega_1^c).
\end{align}
It follows from the results of  Theorem  \ref{T:1} and Lemma \ref{lemma+2} that
\be\label{cond-19}
 \frac{2\t}{\bar{p}}\E\lf(|\bar{e} (T)|^{\bar{p}}  \rt) \leq  C\t.
\ee
 It follows from   \eqref{eq0-1}, \eqref{3.19} and   \eqref{cond-11} that
\begin{align}\la{cond-20}
 \frac{\bar{p}-2}{\bar{p}\t^{2/(\bar{p}-2)}} \PP(\Omega_1^c)  \leq &  \frac{\bar{p}-q}{\bar{p}\t^{2/(\bar{p}-2)}}  \Big( \PP {\{\tau_{\f^{-1}(h(\t)) } \leq T\}} +\PP {\{\rho_{  \t} \leq T\}} + \PP {\{\bar{\rho}_{  \t} \leq T\}} \Big)\nonumber\\
 \leq & \frac{3 (\bar{p}-2)}{\bar{p}\t^{2/(\bar{p}-2)}}   \frac{C}{( \hat{\varphi}^{-1}(h(\t )))^{\bar{p}}}
=  \frac{3 (\bar{p}-2)}{\bar{p}\t^{2/(\bar{p}-2)}}   \frac{C}{(K\t^{-\tau}/C-1)^{\bar{p}/l }}  \leq  C\t.
\end{align}
 On the other hand, note that for any $t\in (0, T \wedge \bar{\theta}_{ \t}] $,
 $$
  {\bar{e}} (t)  = \int_0^t \Big(f(X(s), r(s))-f(Y(s), \bar{r}(s))\Big)\mathrm{d}s+\int_0^t \Big(g(X(s), r(s))-g(Y(s), \bar{r}(s))\Big)\mathrm{d}B(s).
$$
Using the
generalised It\^{o} formula (see,  e.g., \cite[Lemma 1.9, p.49]{Mao06}) yields that
\begin{align*}
&|{\bar{e}} (T \wedge \bar{\theta}_{ \t})|^2\\
 &  =\int_0^{T \wedge \bar{\theta}_{ \t}}  2 \bar{e}^T (s)\big( f(X(s), r(s))-f(Y(s), \bar{r}(s)) \big)+|g(X(s), r(s))-g(Y(s), \bar{r}(s))|^2  \mathrm{d}s + M(T \wedge \bar{\theta}_{ \t}),
\end{align*}
where $
M(t)=2\int_0^t \bar{e}^T (s) \big(g(X(s), r(s))-g(Y(s), \bar{r}(s))\big)\mathrm{d}B(s)
$
 is a local martingale (see, e.g. \cite{Mao06}). Choose a small constant $\iota>0$ such that $1+\iota\leq  \tilde{p} -1$, then the application of Young's  inequality implies
\begin{align}\la{cond-23}
 &\E(|{\bar{e}} (T \wedge \bar{\theta}_{ \t})|^{2} )\nn\\
 &\leq  \E\!\int_0^{T \wedge \bar{\theta}_{ \t}}\! \Big[2 \bar{e}^T (s)\big(f(X(s), r(s))\!-f(\bar{Y}(s), r(s)) \big) + 2 \bar{e}^T (s)\big(f(\bar{Y}(s), r(s))\!-f(Y(s), \bar{r}(s))\big)\nn\\
 &~~~~+(1\!+\iota)|g(X(s), r(s))\!-g(\bar{Y}(s), r(s))|^2
  + \Big(1+\frac{1}{\iota}\Big)|g(\bar{Y}(s), r(s))\!-g(Y(s), \bar{r}(s))|^2\Big]  \mathrm{d}s\nn\\
 &\leq  \E\!\int_0^{T \wedge \bar{\theta}_{ \t}}\! \Big[2 \bar{e}^T (s)\big(f(X(s), r(s))\!-f(\bar{Y}(s), r(s)) \big) +(1\!+\iota)|g(X(s), r(s))\!-g(\bar{Y}(s), r(s))|^2 \nn\\
 &+  |\bar{e} (s)|^2+ |f(\bar{Y}(s), r(s))\!-f(Y(s), \bar{r}(s))|^2
  + \Big(1+\frac{1}{\iota}\Big)|g(\bar{Y}(s), r(s))\!-g(Y(s), \bar{r}(s))|^2\Big]  \mathrm{d}s.
    \end{align}
It follows from   Assumption \ref{a6} and the elementary inequality that
 \begin{align}\la{eqxx}
&\E(|{\bar{e}} (T \wedge \bar{\theta}_{ \t})|^{2} ) \nn\\
  \leq & C \E  \int_0^{T \wedge \bar{\theta}_{ \t} } \Big( |\bar{e}  (s ) |^2  +|f(\bar{Y}(s), r(s)) -f(Y(s), \bar{r}(s))|^2 + |g(\bar{Y}(s), r(s)) -g(Y(s), \bar{r}(s))|^2 \Big) \mathrm{d}s\nn\\
  \leq & C \E\bigg[  \int_0^{T}\!  |\bar{e}  (s \wedge \bar{\theta}_{ \t} ) |^2 \mathrm{d}s +  \int_0^{T }|f(\bar{Y}(s), r(s))\! -f(Y(s), \bar{r}(s))|^2\! + |g(\bar{Y}(s), r(s))\! -g(Y(s), \bar{r}(s))|^2 \mathrm{d}s\bigg]
  \nn\\
  \leq & C \E \bigg[ \int_0^{T}\!  |\bar{e}  (s \wedge \bar{\theta}_{ \t} ) |^2\mathrm{d}s  +  \int_0^{T }|f(\bar{Y}(s), r(s)) -f(Y(s),  {r}(s))|^2 + |g(\bar{Y}(s), r(s)) -g(Y(s),  {r}(s))|^2  \mathrm{d}s\nn\\
  &~~+ \int_0^{T } f( {Y}(s), r(s)) -f(Y(s), \bar{r}(s))|^2 + |g( {Y}(s), r(s)) -g(Y(s), \bar{r}(s))|^2   \mathrm{d}s\bigg]
  \nn\\
  =: & C \E  \int_0^{T }  |\bar{e}(s \wedge \bar{\theta}_{ \t} ) |^2\mathrm{d}s  +J_1+J_2.
\end{align}
Due to (\ref{cond-6}) and (\ref{cond-5}), by Jensen's inequality and Lemma \ref{lemma+1}, one observes
\begin{align}\la{eqxx1}
 J_1&\leq C\E \int_0^{T }\Big(|f(\bar{Y}(s), r(s)) -f(Y(s),  {r}(s))|^2 + |g(\bar{Y}(s), r(s)) -g(Y(s),  {r}(s))|^2 \Big) \mathrm{d}s\nn\\
  &\leq  C\E \int_0^{T }\Big((1+|\bar{Y}(s)|^l+|Y(s)|^l)^2|  {\bar{Y}}(s)   - Y(s) |^2  + (1+|\bar{Y}(s)|^l+|Y(s)|^l) | {\bar{Y}}(s)   - Y(s)|^2 \Big) \mathrm{d}s  \nn\\
   &\leq  C\E \int_0^{T }\Big((1+|\bar{Y}(s)|^{2l}+|Y(s)|^{2l}) |  {\bar{Y}}(s)   - Y(s) |^2   \Big) \mathrm{d}s  \nn\\
   &\leq  C  \int_0^{T }\Big(\E(1+|\bar{Y}(s)|^{2l}+|Y(s)|^{2l})^2\Big)^{\frac{1}{2}}\Big(\E | {\bar{Y}}(s)   - Y(s) |^4   \Big)^{\frac{1}{2}} \mathrm{d}s  \nn\\
     &\leq  C  \t.
\end{align}
Let $j=\lfloor T/\t\rfloor$, by  (\ref{cond-3}), (\ref{cond-4}), using Young's inequality yields
\begin{align*}
  J_2 &\leq
C\sum_{k=0}^{j}\int_{t_k}^{t_{k+1}}\E \Big( |f(Y_k, r(s)) -f(Y_k, r_k)|^2+  |g(Y_k, r(s)) -g(Y_k, r_k)|^2\Big) \mathrm{d}s\\
&=
C\sum_{k=0}^{j}\int_{t_k}^{t_{k+1}} \Big[\E \Big(|f(Y_k, r(s)) -f(Y_k, r_k)|^2+  |g(Y_k, r(s)) -f(Y_k, r_k)|^2\Big)I_{\{r(s)\neq r_k\}}\Big] \mathrm{d}s\nn\\
&\leq
C\sum_{k=0}^{j}\int_{t_k}^{t_{k+1}} \E \Big[ \big(1 +|Y_k|^{l/2+1} +|Y_k|^{l+1}\big)^{2} \E\big(I_{\{r(s)\neq r_k \}}\big|\F_{t_k}\big)\Big] \mathrm{d}s\\
&\leq
C\sum_{k=0}^{j}\int_{t_k}^{t_{k+1}} \E \Big[ \Big(1  +|Y_k|^{2l+2}\big)  \PP{\{r(s)\neq r_k }\big| r_k \}\Big] \mathrm{d}s.
\end{align*}
 It follows from Theorem \ref{T:C_1} and the property of Markov chain $r(\cdot)$ that
\begin{align}\la{eqx}
J_2  \leq
C\t \sum_{k=0}^{j}\int_{t_k}^{t_{k+1}}  \E  \Big(1  +|Y_k|^{2l+2}\big)   \mathrm{d}s\leq C\t.
\end{align}
  Substituting  \eqref{eqxx1}  and \eqref{eqx} into \eqref{eqxx}, applying the Gronwall inequality, we yield that
\begin{align} \label{cond-16}
\E\lf(| {\bar{e}} (T)|^2 I_{\Omega_1}\rt)\leq   \E (|{\bar{e}} (T\wedge \bar{\theta}_{ \t})|^{2} )
\leq   C  \int_0^T \E(|{\bar{e}} (s\wedge \bar{\theta}_{ \t})|^{2}) \mathrm{d}s +C \t \leq C \t.
\end{align}
Inserting (\ref{cond-19}), (\ref{cond-20}) and (\ref{cond-16}) into (\ref{cond-13}) yields the desired assertion.
\end{proof}

By the virtues of Lemmas \ref{lemma+1} and \ref{lemma+4}, we yield the optimal rate of strong convergence.
\begin{theorem}\la{th+1}
 If Assumptions  \ref{a1} and  \ref{a6} hold with  $4(l+1)\leq \bar{p}$,   the numerical solution of Scheme (\ref{Y_0})  with $\tau\in [ l/(\bar{p}-2), 1/2]$   has the property that
\begin{eqnarray*}
\E |  Y(T)-x(T)|^2\leq C\t,~~~~\forall ~T>0.
\end{eqnarray*}
\end{theorem}

\begin{rem}
 If we can find a uniform function  $\varphi(u)$ such that \eqref{e21} holds for each $i\in\SS$, then we can find  the unform truncation function $\pi_\t(x)$. The second equation in \eqref{Y_0} degenerates to
  $Y_{k}=\pi_\t(\tilde{Y}_{k })$ independent of $r_{k }$.   Note that  the results  of Theorem \ref{T:C_2} and Theorem \ref{th+1} still hold for  this special case. 
  One observes that sometimes it is hard to find a uniformly  continuous function  $\varphi(u)\geq \max_{i\in \SS}\varphi_i(u)$ for $ u\in [1, \infty)$. However, for the given  $\t$ and $h(\t)$, it is easy to find $\hat{\varphi}^{-1}(h(\t))$,  the minimum of all $\varphi_i^{-1}(h(\t))$. Let $\pi_\t(x)=  \Big(|x|\wedge \hat{\f}^{-1} (h(\t))\Big)  {x}/{|x|}. $
  Clearly, the property \eqref{Y_01} still holds for Scheme \eqref{Y_0} with $\pi^i_\t(x) \equiv \pi_\t(x)$, which implies Theorem \ref{T:C_2} and Theorem \ref{th+1} still hold.
  On the other hand,   in view of computation,   the   cost of using the uniform scheme is the same as that of using the non-unform one because the value of $r_{k }$ is fixed before each iteration. More precisely, we illustrate it in Example \ref{exp3.1}.
 \end{rem}

\subsection{Numerical examples}\la{exam1}
In order to illustrate the efficiency of Scheme \eqref{Y_0}  we recall the introductory example and present some simulations.
\begin{expl}\la{exp3.1}{\rm
 Consider the  stochastic volatility model   with random switching  between (\ref{sex1}) and (\ref{sex2}) modulated by a Markov chain $r(t)$  with generator $$\Gamma=\left(
  \begin{array}{ccc}
    -4 & 4\\
    0.2 & -0.2\\
  \end{array}
\right),$$
where initial value $x_0=(1, 1)^T$, $\ell=2$ and $B(t)$ is a two-dimensional Brownian motion.
Obviously, its coefficients
\begin{align*}
 f(x,1)=&2.5x\Big(1-|x|\Big), ~~~~~~~~~~~~~~f(x,2)=(1, 2)^T- x,
 \\g(x,1)=&\left(
  \begin{array}{ccc}
    -1 & \sqrt{2}\\
     \sqrt{2}  & 1\\
  \end{array}
\right)|x|^{3/2},~~~~~~ g(x,2)=\left(
  \begin{array}{ccc}
    0.2 & -0.5\\
    1 &  0.4\\
  \end{array}
\right)|x|,
\end{align*}
are locally Lipschitz continuous for any $x\in \mathbb{R}^2$ and $|g(x,1)|^2=6|x|^{3}$, $|g(x,2)|^2=1.45|x|^2$.

In order to represent the simulations by  Scheme \eqref{Y_0},  we divide it into five steps.

\vspace*{4pt}\noindent{\bf Step 1.}  Examine the hypothesis.
Since that
\begin{align*}
&\limsup_{|x|\rightarrow \infty } \displaystyle\frac{(1+|x|^2)\Big[2x^T  f(x, 1)+\trace{\left(g^T(x,1) g(x, 1)\right)}\Big]-(2-5/3)|x^T g(x, 1)|^2}{ |x|^4}
\leq 5,
\end{align*}
and
\begin{align*}
  &\limsup_{|x|\rightarrow \infty } \displaystyle\frac{(1+|x|^2)\Big[2x^T  f(x, 2)+\trace{\left(g^T(x, 2) g(x, 2)\right)}\Big]-(2-5/3)|x^T g(x, 2)|^2}{|x|^4}
\leq -0.64,
\end{align*}
  Assumption \ref{a1} holds for any $0<\bar{p}\leq5/3$. By Theorem \ref{T:1}, the unique regular solution $X(t)$ exists.

\vspace*{4pt}\noindent{\bf Step 2.} Choose $\varphi_i(\cdot)$ and $h(\cdot)$. For any $u\geq 1$, compute
\begin{align*}
 \sup_{|x|\leq u} \dis\left(\frac{ |f(x, 1)|}{1+|x|}\vee\frac{|g(x, 1)|^2 }{(1+|x|)^2}\right)
\leq 6 u, ~~
 \sup_{|x|\leq u} \dis\left(\frac{ |f(x, 2)|}{1+|x|}\vee\frac{|g(x, 2)|^2 }{(1+|x|)^2}\right)
\leq  \sqrt{10}.
\end{align*}
Then choose $\varphi_1(u)=6 u$, $\varphi_2(u)=\sqrt{10}$. Obviously, $\varphi_1^{-1}(u)=u/6$, $\varphi_2^{-1}(u)\equiv +\infty$.
We set $h(\t)=18\t^{-1/2}$ and then it satisfies  (\ref{e22}) for any $\t\in (0,1]$. We can therefore conclude by Theorem   \ref{T:C_2} that the numerical solution $Y(t)$ of Scheme \eqref{Y_0} 
satisfies
$
\lim_{\t\rightarrow 0} \E |X(T)-Y(T)|^q=0
$
for any $T\in [0,+\infty)$, $0<q<5/3$.

\vspace*{4pt}\noindent{\bf Step 3.} {M{\scriptsize ATLAB}} code.
Next we specify the {M{\scriptsize ATLAB}} code for simulating   $Y(T)$:
\vspace{-1.2em}
\begin{lstlisting}
%MATLAB code for simulating the truncated EM approximation Y(T)
clear all;
Y=[1; 1]; r=2; T=10; dt=2^(-17); Gam=[-4 4;0.2 -0.2]; c=expm(Gam*dt);
dB=sqrt(dt)*randn(2,T/dt); v=3*dt^(-1/2); %Obviously, v>norm(Y);
for n=1:T/dt
    if r==1
        Y=Y+2.5*Y*(1-norm(Y))*dt+...
            [-1  sqrt(2);sqrt(2) 1]*norm(Y)^(3/2)*dB(:,n);
    else
        Y=Y+([1; 2]-Y)*dt+[0.2 -0.5;1 0.4]*norm(Y)*dB(:,n);
    end
    if rand<=c(r,1)
        r=1;
        if norm(Y)>v
            Y=v*Y/norm(Y);
        end
    else
        r=2;
    end
end
\end{lstlisting}

\vspace*{4pt}\noindent{\bf Step 4.} Approximating the error  $\mathbb{E}|X(T)-Y(T)|^q$. Due to no closed-form of the solution,
using Scheme \eqref{Y_0}, we regard  the better approximation with $\t=2^{-19}$  as  the exact solution $X(t)$ and compare it with the numerical solution $Y(t)$ with $\t=2^{-8}, 2^{-9}, \dots, 2^{-17}$. To compute the approximation error, we run $M$ independent trajectories where $X^{(j)}(t)$ and $Y^{(j)}(t)$ represent the $j$th trajectories of  the exact solution $X(t)$ and  the numerical solution $Y(t)$ respectively. Thus
$$
 \mathbb{E}|X(T)-Y(T)|^q = \frac{1}{M}\sum_{j=1}^{M}
| X^{(j)}(T)-Y^{(j)}(T)|^q.
$$
\noindent{\bf Step 5.}  The log-log error   plot with $M=1000$.  The simulation procedure is carried out by steps 3 and 4.  The   blue solid line depicts  log-log error   while the red dashed  is  a reference line of slope $1/2$ in Figure \ref{errorfig1}. Figure \ref{errorfig1} depicts the approximation error $\mathbb{E}|X(10)-Y(10)|$ of the exact solution  and the numerical solution of Scheme \eqref{Y_0}  as the function of stepsize $\t\in \{2^{-8},2^{-9},\dots,2^{-17}\}$.  
One observes that the schemes  proposed in \cite{Mao20152,Sabanis13,Hutzenthaler12}  are quite sensitive to the high nonlinearity of the diffusion coefficient, which don't work for the above equation. However, the performance of Scheme \eqref{Y_0}  is very nice for this case.
\begin{figure}[!ht]
  \centering
\includegraphics[width=12cm,height=6.2cm]{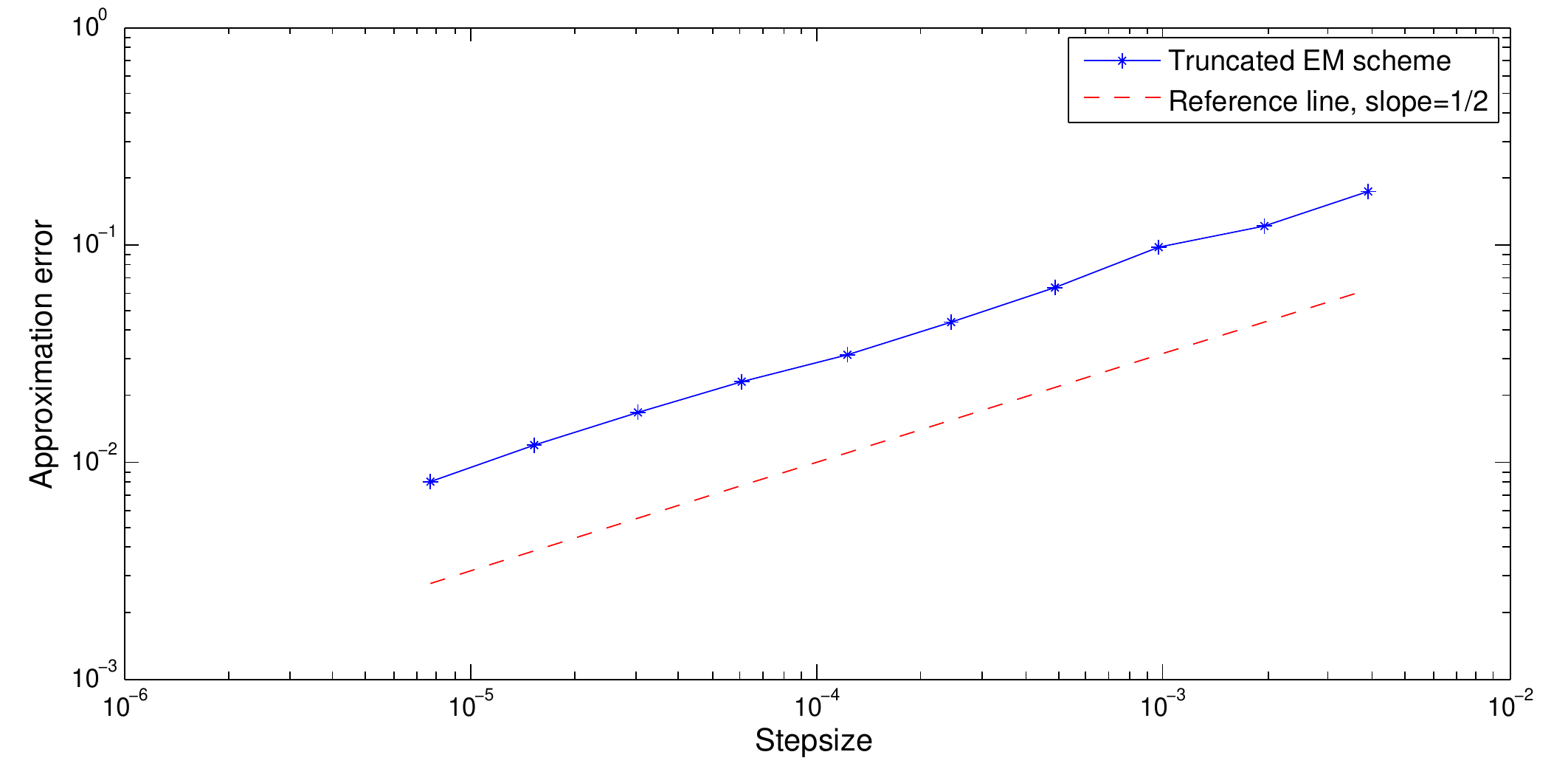}
  \caption{The approximation error $\mathbb{E}|X(10)-Y(10)|$ of the exact solution and the numerical solution by Scheme \eqref{Y_0} as the function of stepsize $\t\in \{2^{-8},2^{-9},\dots,2^{-17}\}$.}
  \label{errorfig1}
\end{figure}
\begin{figure}[!ht]
  \centering
\includegraphics[width=12cm,height=6.2cm]{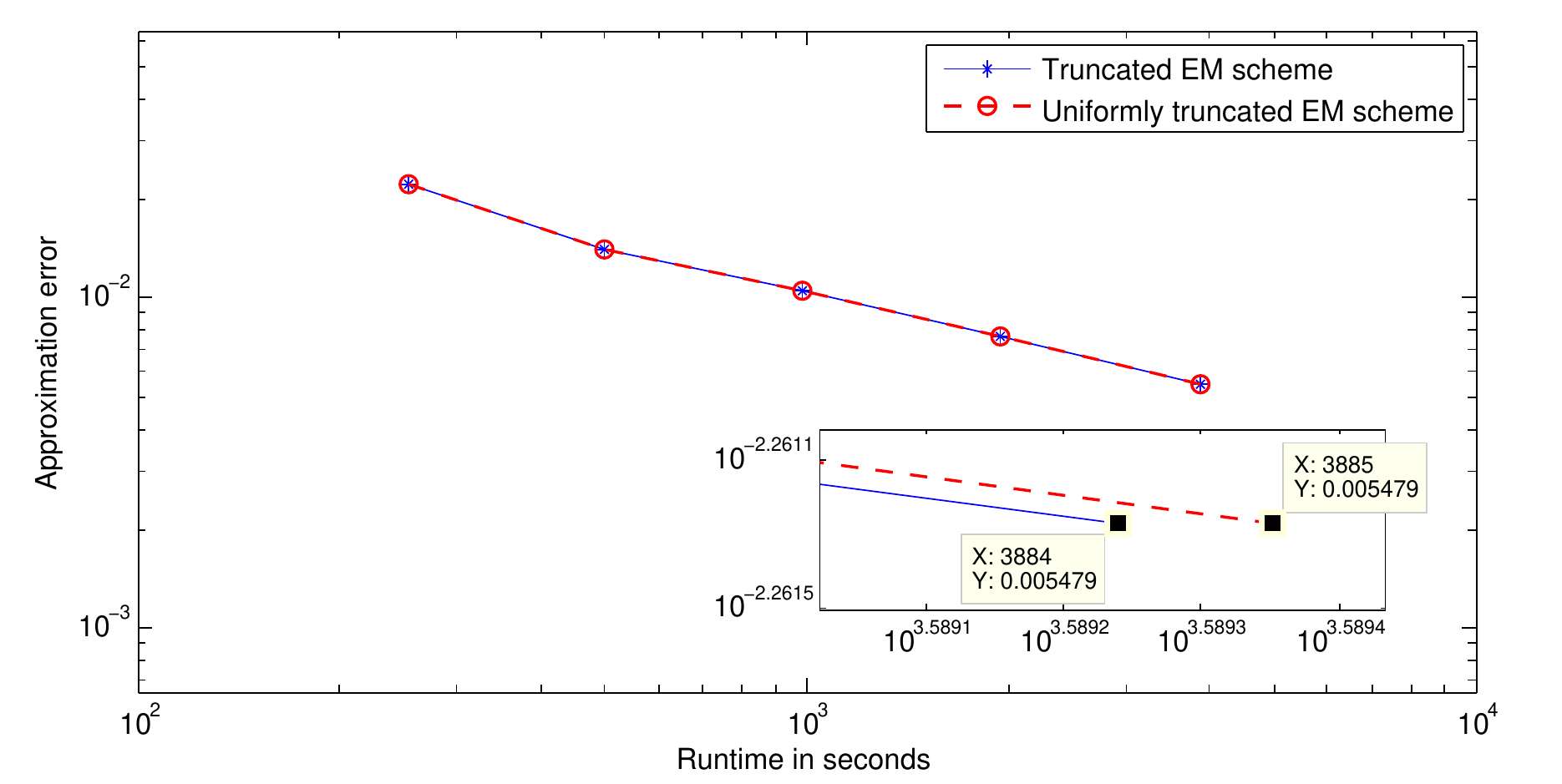}
  \caption{The approximation error $\mathbb{E}|X(1)-Y(1)|$ of the exact solution, and the numerical solutions by the truncated EM scheme dependent on states  and by the uniformly truncated EM scheme, respectively, as functions of the runtime with $\t\in \{2^{-13},2^{-14},\dots,2^{-17}\}$.}
  \label{errorfig1_2}
\end{figure}

  On the other hand, as the claim in Remark 3.14, we can also use the uniform truncation mapping $\pi_\t(x)=  \Big(|x|\wedge 3  \t^{-1/2}\Big)  {x}/{|x|}. $  Now let us  compare the simulation time of   Scheme \eqref{Y_0} with the truncation mapping dependent on the states and with the uniform one $ \pi^i_\t(x)\equiv\pi_\t(x)$.
  Figure \ref{errorfig1_2} depicts the approximation error $\mathbb{E}|X(1)-Y(1)|$ of the exact solution,  and the numerical solutions by the truncated  scheme and  by the uniformly truncated scheme, respectively, as functions of the runtime with $\t\in \{2^{-13},2^{-14},\dots,2^{-17}\}$ and $M=1000$.  When $\t=2^{-17}$,   the runtime of   the truncated  scheme  achieving the accuracy 0.005479 on the  computer with Intel Core 2 duo CPU 2.20GHz, is about 3884 seconds while the runtime  of the uniformly truncated scheme achieving the accuracy 0.005479 is about 3885 seconds on the same computer (see the enlargement in Figure \ref{errorfig1_2}). Thus, the computational cost   of the non-uniform  scheme  is the same as that of the uniform one.
}\end{expl}

Let us discuss another example  to compare our scheme with the implicit EM scheme.
\begin{expl}\la{exp2}{\rm
 Consider the scalar hybrid cubic SDE (i.e. the stochastic Ginzburg-Laudau equation (4.52) in \cite[p.125]{Kloeden})
\begin{align}\la{exyhf4.1}
\!\! \mathrm{d}X(t)=\big(a(r(t))X(t)\!+b(r(t))X^3(t)\big)\mathrm{d}t+\sigma(r(t)) X(t)\mathrm{d}B(t),~~~~ t \geq 0,
\end{align}
 where
   $r(\cdot)$ is the Markov chain taking values in $ \mathbb{S}=\{1, 2 \}$ with generator matrix
   $$  \Gamma=\left(
  \begin{array}{ccc}
    -\gamma_{12} & \gamma_{12}\\
    \gamma_{21} & -\gamma_{21}\\
  \end{array}
\right).
$$ and $b(i)\leq 0$ for any $i\in \mathbb{S}.$
{ Obviously, for any $i\in \mathbb{S}$, $f(x,i)=a(i)x+b(i)x^3$, $g(x,i)=\sigma(i) x$ are locally  Lipschitz continuous.  Assumption \ref{a1} holds for any $\bar{p}>0$.}
Note that there exists a unique regular solution $X(t)$ to SDS \eqref{exyhf4.1} for any initial data $x_0>0$, $\ell\in \mathbb{S}$.
In the same way as in \cite{Kloeden}, we get its closed-form
\begin{align}\la{exact}
X(t) =  \frac{x_0\exp{\bigg\{\dis\int_{0}^{t} \Big[a(r(s))-\frac{1}{2}\sigma^2(r(s))\Big]\mathrm{d}s+\int_{0}^{t}\sigma(r(s))\mathrm{d}B(s)\bigg\}}}
 {\sqrt{1 - 2x^2_0\dis\int_{0}^{t} b(r(s))\exp{\bigg\{\int_{0}^{s}\! \Big[2a(r(u)) -  \sigma^2(r(u))\Big]\mathrm{d}u + \int_{0}^{t} 2\sigma(r(u))\mathrm{d}B(u)\bigg\}}\mathrm{d}s}}.
\end{align}
Clearly, for any $i\in \mathbb{S}$,
$
 \sup_{|x|\leq u} \dis\left(\frac{  |f(x, i)|}{1+|x|} \vee\frac{|g(x, i)|^2 }{(1+|x|)^2}\right)
   \leq  c (u^2+1)
$
for any $u\geq 1$, where $c=(|\breve{b}|\vee |\breve{\sigma}|^2\vee |\breve{a}|), |\breve{b}|=\max_{i\in\mathbb{S}}|b(i)|$.  Thus, we define
$\varphi(u) =c (u^2+1),~\forall u\geq 1,$
  and
$
  h(\t) = \varphi(|x_0|)\t^{-0.2},~\forall \t\in (0, 1].
$
Compute $\varphi^{-1}(u)=\big(u/c-1\big)^{1/2}$ for any $u>c$.  
Note that Assumption \ref{a5} holds with any $\tilde{p}>2$ and $l=2$. By the virtue of Theorem \ref{th+1} the numerical solution of Scheme \eqref{Y_0}  converges to the exact solution in the root mean square with error estimate $\t^{1/2}$.

We compare the simulations by the implicit EM scheme and by Scheme \eqref{Y_0} for   SDS \eqref{exyhf4.1} with $\gamma_{12}=1$, $\gamma_{21}=4;$ $a(1)=1$,  $b(1)=-1$, $\sigma(1)=2$;
  $a(2)=0.5$, $b(2)=-1$, $\sigma(2)=1$; $x_0=20$, $\ell=1$.
We specify  the truncated EM scheme
\begin{align} \la{yhfTE1}
\left\{
\begin{array}{ll}
Y_0=x_0,~~r_0=\ell,&\\
\tilde{Y}_{k+1}=Y_k+[a(r_k)Y_{k}+b(r_k)(Y_{k})^3]\t +\sigma(r_k)Y_k\t B_k,~~~~~~~ \\
Y_{k+1}=\dis\Big(|\tilde{Y}_{k+1}|\wedge   \sqrt{(x_0^2+1)\t^{-0.2}-1} \Big) \tilde{Y}_{k+1}/|\tilde{Y}_{k+1}|,
\end{array}
\right.
\end{align} and the implicit EM scheme
\begin{align}\la{yhfBE1}
\left\{
\begin{array}{ll}
\bar{Y}_0=x_0,~~r_0=\ell,&\\
\bar{Y}_{k+1}=\bar{Y}_k+ [a(r_k)\bar{Y}_{k+1}+b(r_k)(\bar{Y}_{k+1})^3] \t + \sigma(r_k)\bar{Y}_k\t B_k,
\end{array}
\right.
\end{align}
for any $k=0, 1,\dots, \bar{N}-1$, where $T=2, \bar{N}=\lfloor T/\t\rfloor\geq 7$.
Thanks to Cardano's method, roots of one-dimensional polynomials of degree three are known explicitly. Thus  the implicit  scheme \eqref{yhfBE1}  becomes
\begin{align}\la{yhfBE1+1}
\left\{
\begin{array}{ll}
\bar{Y}_0=x_0,~~r_0=\ell,&\\
D_k=\bar{Y}_k(1+ \sigma(r_k)\t B_k)/(-2b(r_k)\t),&\\
H_k=\sqrt{D_k^2-(1/\t- a(r_k))^3/(3b(r_k))^3},&\\
\bar{Y}_{k+1}=(H_k+D_k)^{1/3}-(H_k-D_k)^{1/3}.
\end{array}
\right.
\end{align}
  Figure \ref{errorfig2} depicts the root mean square  approximation error  with $1000$ sample points  by different schemes as the  functions of runtime, which reveals that  the runtime of the implicit EM scheme \eqref{yhfBE1} achieving the accuracy $2/1000$ is  $1.43$ time than that of the truncated EM scheme \eqref{yhfTE1} with the same stepsize.
\begin{figure}[!ht]
  \centering
\includegraphics[width=15cm,height=6.4cm]{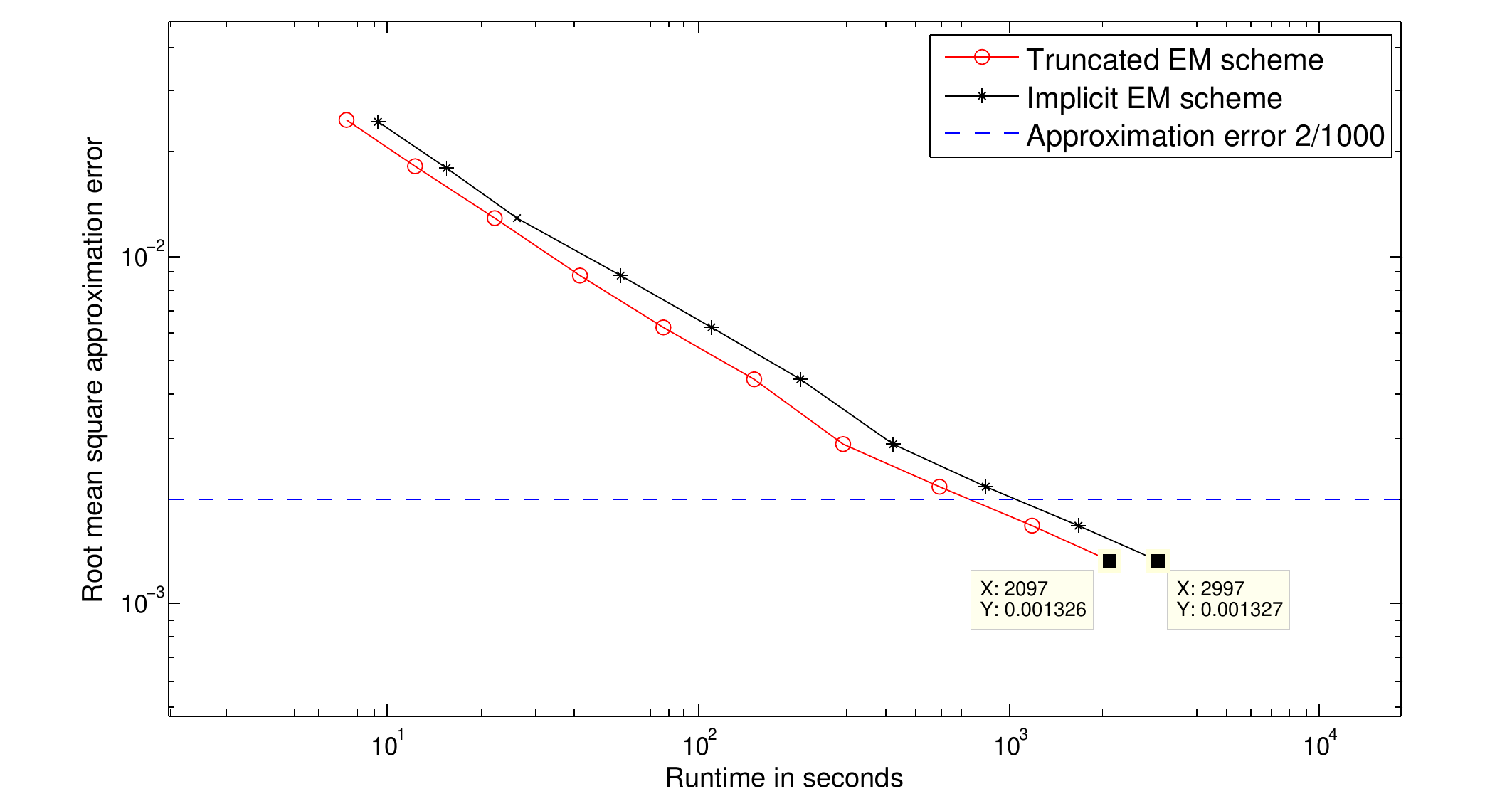}
  \caption{  The black trajectory depicts the root mean square  approximation error $(\mathbb{E}|X(2)-\bar{Y}_{\bar{N}}|^{2})^{\frac{1}{2}}$ of the
exact solution of SDS \eqref{exact} and the numerical solution by the implicit EM scheme (\ref{yhfBE1+1}) while the red trajectory depicts  the  root mean square approximation error $(\mathbb{E}|X(2)-Y_{\bar{N}}|^{2})^{\frac{1}{2}}$ of the
exact solution of SDS \eqref{exact} and the numerical solution by  the  truncated EM scheme (\ref{yhfTE1})  as the functions of the runtime with $\t\in \{2^{-10}, 2^{-11}, \dots, 2^{-19}\}$.}
  \label{errorfig2}
  \vspace*{-8pt}
\end{figure}
}
\end{expl}

\section{Approximation of invariant measure}\la{Inv}
In this section, we are concerned with the existence and uniqueness of   invariant measures  of the exact  and numerical solutions for SDS \eqref{e1}. We reconstruct an explicit scheme and show the existence and uniqueness of   the numerical invariant measure  converging  to  the underlying exact one  in  the Wasserstein metric.
  Our results  cover  a large kind of nonlinear SDSs  with only locally Lipschitz continuous   coefficients.
 For the convenience of invariant measure study we introduce some notations as well as an assumption.
 We write $(X_t^{x_0, \ell},r_t^{\ell})$ in lieu of
$(X(t),r(t))$ to highlight the initial data $(X(0),r(0))=(x_0,\ell)$. For each $N>0$, let $B_{N}(\mathbf{0})=\{x\in \RR^n:|x|\leq N\}$.
 Let $\mathscr{B}(\RR^n)$ denote  the family of all Borel sets in $\RR^n$ and   $\mathcal P(\RR^n\times \SS)$ denote the family of all probability measures on $\RR^n\times \SS$.
  For any $p\in(0,1]$, define a metric on $\RR^n\times \SS$ as follows
$$
d_{p}((x,i),(y,j)):=|x-y|^{p}+I_{\{i\neq j\}},~~~(x,i),(y,j)\in \RR^n\times \SS,
$$
and the corresponding  Wasserstein distance between $\mu,\bar{\mu}\in \mathcal P(\RR^n\times \SS)$ by
\begin{align}
W_{p}(\mu, \bar{\mu}):=\inf_{\pi\in \mathbb{C}(\mu, \bar{\mu})} \int_{\RR^{n}\times \mathbb{S}} \int_{\RR^{n}\times \mathbb{S}} d_{p}(x,y) \pi(dx,dy),
\end{align}
where $ \mathbb{C}(\mu, \bar{\mu})$ denotes the collection of all probability measures on $(\RR^{n}\times \mathbb{S})\times (\RR^{n}\times \mathbb{S})$ with marginal measures $\!\mu$ and $\bar{\mu}$. The Wasserstein distance $W_{p}$ admits a duality representation (see, e.g, \cite{Shao2015}), i.e.
\begin{align*}
W_{p}({\mu},\bar{\mu})=
\sup_{\bar{\psi}:\textrm{Lip}(\bar{\psi})\leq 1}\Big\{\int_{\RR^{n}\times \mathbb{S}}\bar{\psi}(x, i)\mathrm{d}  {\mu} -\int_{\RR^{n}\times \mathbb{S}}\bar{\psi}(y, j)\mathrm{d}\bar{\mu}\Big\},
\end{align*}
where $\bar{\psi}$ is a continuous function on $\RR^{n}\times \mathbb{S}$ and
$$
\textrm{Lip}(\bar{\psi}):=\textrm{sup}\bigg\{\frac{\bar{\psi}(x, i)-\bar{\psi}(y, j)}{d_{p}((x, i),(y, j))}:\forall(x, i)\neq (y, j)\in\RR^{n}\times \mathbb{S}\bigg\}.
$$
 From this  section  we always assume   $\gamma(t)$ is {\it irreducible}, 
 namely, the following linear equation
 \begin{eqnarray}\label{eq:a1.2}
\pi \Gamma=0,\ \ \ \ \ \ \ \  \ \sum_{i=1}^{m}\pi_i=1,
\end{eqnarray}
has a unique solution $\pi=(\pi_1, \dots, \pi_m)\in {\mathbb R}^{1\times m}$  satisfying $\pi_i>0$ for each $i\in \mathbb{S}$. This solution is termed a stationary distribution.
 For any vector $u=(u_1,\dots, u_m)^T$,  any constant $ p>0$,  define
\begin{align}\label{yhfYH_01}
\mathrm{diag}(u):= \mathrm{diag}(u_1,  \dots,   u_m),~\Gamma_{p,u}:= \Gamma+\frac{p}{2}\mathrm{diag}(u),~\eta_{p,u}:= -\!\!\!\max_{\nu\in \mathrm{spec}(\Gamma_{p,u})}\mathrm{Re}(\nu),
\end{align}
  where
$\mathrm{spec}(\Gamma_{p,u})$ and $\mathrm{Re}(\nu)$ denote the spectrum of $\Gamma_{p,u}$ (i.e.  the multiset of its eigenvalues)  and the real part of $\nu$, respectively.
  In order to examine the properties  of the exact solution of SDS (\ref{e1}), we prepare the following lemma.
  \begin{lemma}\la{L:1}
Assume  that $
\pi u:= \sum_{i\in \mathbb{S}}\pi_i u_i<0.
$
Then for any $p\in (0,p^*_u)$,  $\eta_{p,u}>0$, and there is a vector $\xi ^{p,u}=(\xi_1 ^{p,u},\dots,\xi_m ^{p,u})^T\gg \mathbf{0}$ such that
$$
 \sum_{j= 1}^{m}\gamma_{ij}\xi_j^{p,u}+\frac{p u_i}{2}\xi_i^{p,u}=-\eta_{p,u}\xi_i^{p,u}<0,~~~~i=1,\dots, m,
$$
where
\begin{align}\la{yhfEq:1}
 \left\{
\begin{array}{ll}
p^*_u =+\infty ,&\mathrm{if}~~\check{u}  \leq  0,\\
\dis  p^*_u \in \Big(0,\min_{i\in \mathbb{S}, u_i>0}\big\{ -2\gamma_{ii}/u_i\big\}\Big) ,&\mathrm{if}~~\check{u}  >  0.
\end{array}
\right.
\end{align}
  \end{lemma}
\begin{proof} Fix any  $p\in (0,p^*_u)$. Recalling (\ref{yhfYH_01}) and utilizing  \cite[Proposition 4.2]{Ba} yields  $\eta_{p,u}>0$.  Let $\Gamma_{p,u}(t)=\exp(t \Gamma_{p,u})$.
 Then, the spectral radius of $\mathrm{Ria}(\Gamma_{p,u}(t))$ (i.e. $\mathrm{Ria}(\Gamma_{p,u}(t))=\sup_{\bar{\nu}\in \mathrm{spec}(\Gamma_{p,u}(t))}|\bar{\nu}|$)   equals to $\exp(-\eta_{p,u} t)$.
  Since all
 coefficients of $\Gamma_{p,u}(t)$ are positive, by the Perron-Frobenius theorem (see, e.g., \cite[p.6]{Chen2007}),
  $\exp(-\eta_{p,u} t)$ is a simple eigenvalue of $\Gamma_{p, u}(t)$ and
 there exists an eigenvector
 $
 \xi^{p,u}=(\xi_1^{p,u}, \ldots, \xi_m^{p,u})^T\gg \mathbf{0}
 $
  corresponding to $\exp(-\eta_{p, u} t)$,
 where $\mathbf{0}\in \mathbb{R}^m$ is a zero vector, and $\xi^{p,u}\gg \mathbf{0}$ ($\ll \mathbf{0}$) means each component $\xi_i^{p,u}>0$ ($<0$) for any $1\leq i\leq m$.  Note
 that eigenvector of $\Gamma_{p,u}(t)$ corresponding to $\exp(-\eta_{p, u} t)$ is also an eigenvector of $\Gamma_{p, u}$ corresponding to $-\eta_{p, u}$.   Thus, we obtain that
$
\Gamma_{p,u} \xi^{p,u} =-\eta_{p,u}\xi^{p,u}\ll \mathbf{0}.
$
 Therefore the desired result follows.
\end{proof}

\subsection{Boundedness}\label{bound}
Since it is related closely to the tightness  as well as the ergodicity we go further to investigate the moment boundedness of solutions in infinite time interval. Firstly, we establish the criterion on the moment boundedness of the exact solutions  of SDS (\ref{e1}). Then we show that the numerical solutions  of Scheme \eqref{Y_0} keep this property very well.
\begin{theorem}\la{Tyle2.2+}
Suppose that Assumption \ref{a1} with matrix $Q_i\equiv Q$  and $\pi \alpha<0$ hold, for any $p\in (0,p^*_\alpha)\cap (0,  \bar{p}]$, the
 solution $X(t)$ of SDS (\ref{e1})
 satisfies
\begin{align}\la{Typ}
\sup_{t\geq 0}\mathbb{E}|X(t)|^p\leq C,
  \end{align}
where $C$ is independent of    $t$ and $i$, each $\alpha_i$ is given by Assumption \ref{a1}, and $\alpha:=(\alpha_1, \dots, \alpha_m)^T$.
\end{theorem}
\begin{proof}
For the given $\alpha $ and any $p>0$,   $\mathrm{diag}(\alpha), ~\Gamma_{p,\alpha},~ \eta_{p,\alpha},~p^*_\alpha$ are defined as \eqref{yhfYH_01} and \eqref{yhfEq:1}. For any $p\in (0,p^*_\alpha)\cap (0,  \bar{p}]$,  Lemma \ref{L:1} implies that $\eta_{p,\alpha}>0$ and there exists  $\xi^{p,\alpha}\gg \mathbf{0}$ such that
\begin{equation}\la{yhfF_101}
 \sum_{j= 1}^{m}\gamma_{ij}\xi^{p,\alpha}_j+\frac{p \alpha_i}{2}\xi^{p,\alpha}_i=-\eta_{p,\alpha}\xi^{p,\alpha}_i<0,~~~~i=1,\dots, m
\end{equation}
holds. Choose $0<\check{\kappa}< ({\eta_{p,\alpha}}/{p})\wedge 1$.
Using the
generalised It\^{o} formula (see,  e.g., \cite[Lemma 1.9, p.49]{Mao06})  together with \eqref{yhfF_101}, (\ref{eq02})-\eqref{YH_2}, yields
\begin{align*}
 &\mathbb{E}\Big[\mathrm{e}^{\frac{\eta_{p,\alpha}}{2} t}\big(1+X^T(t)QX(t)\big)^{\frac{p}{2}}\xi^{p,\alpha}_{r_t}\Big]\nn\\
=& \mathbb{E}\Big[\left(1+x^T_0Qx_0\right)^{\frac{p}{2}}\xi^{p,\alpha}_{\ell}\Big] +\mathbb{E}\int_{0}^{t}\mathrm{e}^{\frac{\eta_{p,\alpha}}{2}s}\Big[\frac{\eta_{p,\alpha}}{2}\big(1+X^T(s)QX(s)\big)^{\frac{p}{2}}\xi^{p,\alpha}_{r_s}\nn\\
 &~~~~~~~~~~~~~~~~~~~~~~~~~~~~~~~~~~~~~~~~~~~~~~~~~~~+ {\cal{L}}\Big(\big(1+X^T(s)QX(s)\big)^{\frac{p}{2}}\xi^{p,\alpha}_{r_s}\Big) \Big]\mathrm{d}s \nn\\
\leq& \left(1+x^T_0Qx_0\right)^{\frac{p}{2}}\xi^{p,\alpha}_{\ell}
+\mathbb{E}\int_{0}^{t}\frac{p\kappa_{r_s}-\eta_{p,\alpha}}{2}\mathrm{e}^{\frac{\eta_{p,\alpha}}{2}s}\big(1+X^T(s)QX(s)\big)^{\frac{p}{2}} \xi^{p,\alpha}_{r_s}   \mathrm{d}s
  +\mathbb{E}\int_{0}^{t}C_{r_s}\mathrm{e}^{\frac{\eta_{p,\alpha}}{2}s}\mathrm{d}s \nn\\
 \leq& \left(1+x^T_0Qx_0\right)^{\frac{p}{2}}\xi^{p,\alpha}_{\ell} + C\int_{0}^{t}\mathrm{e}^{\frac{\eta_{p,\alpha}}{2} s}\mathrm{d}s,
\end{align*}
where
\begin{align*}
 &{\cal{L}}\Big(\big(1+x^TQx\big)^{\frac{p}{2}}\xi^{p,\alpha}_{i}\Big)\nn\\
 =&\frac{p}{2} (1+x^TQx)^{\frac{p}{2}-2}\xi^{p,\alpha}_{i}\Big[
(1+x^TQx) \psi(x,i)
  -(2-p)|x^TQg(x, i)|^2 \Big]
 +\sum_{j=1}^{m}\gamma_{ij}(1+x^TQx)^{\frac{p}{2}}\xi^{p,\alpha}_{j}\nn\\
\leq&\Big[\frac{p}{2}\big(\alpha_i+\kappa_i\big)\xi^{p,\alpha}_{i}
 +\sum_{j=1}^{m}\gamma_{ij}\xi^{p,\alpha}_{j}\Big](1+x^TQx)^{\frac{p}{2}}+C_i.
\end{align*}
One observes that
 $
\mathbb{E}\Big[\big(1+X^T(t)QX(t)\big)^{\frac{p}{2}}\xi^{p,\alpha}_{r_t}\Big]
 \leq   C\mathrm{e}^{-\frac{\eta_{p,\alpha}}{2} t} + C.
$
Therefore the desired assertion follows from that
 \begin{align*}
 \mathbb{E}|X(t)|^{p}\leq \left( \lambda_{\min}(Q) \right)^{-\frac{p}{2}}\big(\hat{\xi}^{p,\alpha}\big)^{-1} \mathbb{E}\Big[\big(1+X^T(t)QX(t)\big)^{\frac{p}{2}}\xi^{p,\alpha}_{r_t}\Big]
\leq  C.
\end{align*}
 The proof is complete. \end{proof}

  In order to show  that the numerical solutions of Scheme \eqref{Y_0}  preserve this asymptotic boundedness perfectly, we further require    the function  $h(\t)$  satisfying
\be\la{F+}
 \t^{1/2-\theta}h(\t)\leq K,~~~\forall~ \t \in (0, 1]
\ee
 for some $0<\theta <\frac{1}{2}$.  Then the asymptotic boundedness of the numerical solution   is  obtained  as follows.
\begin{theorem}\la{4TT:F_1}
 Under the conditions of Theorem \ref{Tyle2.2+}, for any $p\in (0,p^*_\alpha)\cap (0,  \bar{p}]$, there is a constant $ {\t}^*\in (0, 1]$  such that  the   scheme   \eqref{Y_0} has the   property that
\begin{align}\label{ty3.8}
\sup\limits_{ 0< \t \leq {\t}^*}\sup\limits_{k\geq0}\mathbb E|Y_k|^{p}
\leq C,
\end{align}
where  $C$ is independent of  $k$ and  $\t$.
  \end{theorem}
\begin{proof}
For any $p\in (0,p^*_\alpha)\cap (0,  \bar{p}]$,  by the virtue of \eqref{Y_0},  we have
 \begin{align}\la{FE:H2}
(1+\tilde{Y}^T_{k+1}Q\tilde{Y}_{k+1})^{\frac{p}{2}}=\left(1+ Y^T_kQY_k \right)^{\frac{p}{2}}( 1+\varsigma_k)^{\frac{p}{2}}
\end{align}
for any integer $k\geq 0$, where
\begin{align*}
\varsigma_k =\big(1+\! Y^T_kQY_k\big)^{-1}\Big(& 2Y_k^TQ f(Y_k, r_k)\t+\!  \t B_k^Tg^T(Y_k, r_k)Q g(Y_k, r_k)\t B_k+\! 2Y_k^TQg(Y_k, r_k)\t B_k\\
   &+ f^T(Y_k, r_k)Qf(Y_k, r_k)\t^2 +2f^T(Y_k, r_k)Qg(Y_k, r_k)\t B_k \t\Big).
\end{align*}
One observes $\varsigma_k>-1$. By Lemma \ref{L-Tay}, without loss the generality,  we prove \eqref{ty3.8} only for  $0<p\leq 2$.
It follows from \eqref{Y2} and   \eqref{FE:H2} that
\begin{align} \la{FY_3}
  &  \E \Big[ \left(1+\tilde{Y}^T_{k+1}Q\tilde{Y}_{k+1}\right)^{\frac{p}{2}}\xi^{p,\alpha}_{r_{k+1}} \big|\mathcal{F}_{t_k}\Big]  \nonumber\\
 \leq & (1+ Y^T_kQY_k)^{\frac{p}{2}}\bigg\{\E\big[\xi^{p,\alpha}_{r_{k+1}}\big|\mathcal{F}_{t_k}\big]+ \frac{p}{2} \E\big[\varsigma_k\xi^{p,\alpha}_{r_{k+1}}\big|\mathcal{F}_{t_k}\big] + \frac{p(p-2)}{8} \E\big[\varsigma_k^2\xi^{p,\alpha}_{r_{k+1}}\big|\mathcal{F}_{t_k}\big]\nn\\
 &~~~~~~~~~~~~~~~~~~~~~~ + \frac{p(p-2)(p-4)}{48} \E\big[\varsigma_k^3\xi^{p,\alpha}_{r_{k+1}}\big|\mathcal{F}_{t_k}\big]\bigg\},
\end{align}
   where the vector $\xi^{p,\alpha}=(\xi^{p,\alpha}_1,\dots, \xi^{p,\alpha}_m)^T\gg \mathbf{0}$ is given in the proof of Theorem \ref{Tyle2.2+}.    It follows from Lemma \ref{Z:1} that
$\E\big[\xi^{p,\alpha}_{r_{k+1}}\big|\mathcal{F}_{t_k}\big]=\xi^{p,\alpha}_{r_{k}}+\sum_{j\in \mathbb{S}}\xi^{p,\alpha}_{j}\left(\gamma_{r_k j}\t+o(\t)\right).
$
Then, making use of the techniques in the proof of Theorem \ref{T:C_1} as well as   \eqref{F+} yields
\begin{align}\la{FY_4}
& \E\big[\varsigma_k\xi^{p,\alpha}_{r_{k+1}}\big|\mathcal{F}_{t_k}\big]  =\E\big[\E\big(\varsigma_k\xi^{p,\alpha}_{r_{k+1}}\big|\mathcal{G}_{t_k}\big)\big|\mathcal{F}_{t_k}\big] \nn\\
\leq&  \Big[ \frac{\psi(Y_k,r_k)\t}{1+Y^T_kQY_k}+ C\t^{1+2\theta}\Big] \Big[\xi^{p,\alpha}_{r_{k}}+\sum_{j\in \mathbb{S}}\xi^{p,\alpha}_{j}\left(\gamma_{r_k j}\t+o(\t)\right)\Big]
\leq    \frac{\psi(Y_k,r_k)\xi^{p,\alpha}_{r_{k}}\t}{1+Y^T_kQY_k}+ C\t^{1+2\theta}.
\end{align}
One further observes that
\begin{align}\la{FY_5}
&\E\Big[\varsigma_k^2\xi^{p,\alpha}_{r_{k+1}}|\F_{t_k}\Big]=\E\big[\E\big(\varsigma_k^2\xi^{p,\alpha}_{r_{k+1}}\big|\mathcal{G}_{t_k}\big)\big|\mathcal{F}_{t_k}\big]\nn\\
\geq &4(1+Y^T_kQY_k)^{-2 }\Big( | Y_k^TQg(Y_k, r_k)|^2 \t
    -2|Q|^2(1+|Y_k|)^4 h^2(\t)  \t^2 \Big)\E\big[\xi^{p,\alpha}_{r_{k+1}}\big|\mathcal{F}_{t_k}\big]\nonumber\\
  \geq & \frac{4| Y_k^TQg(Y_k, r_k)|^2\xi^{p,\alpha}_{r_{k}} \t}{(1+Y^T_kQY_k)^2} -C\t^{1+2\theta}.
\end{align}
From  \eqref{Y_01},  \eqref{L_2} and \eqref{F+}, one obtains
\begin{align}\la{FY_6}
&\E\Big[\varsigma_k^3\xi^{p,\alpha}_{r_{k+1}}|\F_{t_k}\Big]=\E\Big[\E\big(\varsigma_k^3\xi^{p,\alpha}_{r_{k+1}}\big|\mathcal{G}_{t_k}\big)\big|\mathcal{F}_{t_k}\Big]\nn\\
\leq&  C{(1+Y_k^TQY_k)^{-3 }}   \bigg[  |Y_k|^3 |f(Y_k, r_k)|^3\t^3+  |g(Y_k, r_k)|^6\t^3
 + |f(Y_k, r_k)|^6 \t^6
  \nonumber\\
+& |g(Y_k, r_k)|^2\Big( |Y_k|^2 + |f(Y_k, r_k)|^2 \t^2\Big) \t^2\Big( |Y_k||f(Y_k, r_k)|  +  |g(Y_k, r_k)|^2
 + |f(Y_k, r_k)|^2\t \Big) \bigg] \E\big[\xi^{p,\alpha}_{r_{k+1}}\big|\mathcal{F}_{t_k}\big]\nonumber\\
\leq&  C{(1+Y_k^TQY_k)^{-3 }}(1+|Y_k|)^6 \Big[ h^3(\t)\t^3
 +  h^6(\t)  \t^6
  \nonumber\\
&~~~~~~~~~~~~~~~~~~~~~~+ h^2(\t)\t^2 \Big(1+h^2(\t)\t^2  \Big)\Big( 1+ h(\t)\t \Big)
 \Big] \E\big[\xi^{p,\alpha}_{r_{k+1}}\big|\mathcal{F}_{t_k}\big]
 \leq  C\t^{1+2\theta}.
\end{align}
Similarly, we  can also prove that for any integer $l>3$, $ \E\big[|\varsigma_k|^l\xi^{p,\alpha}_{r_{k+1}}|\F_{t_k}\big]\leq C \t^{1+2\theta}.$  For any integer $k\geq 0$,
 substituting (\ref{FY_4})-(\ref{FY_6}) into (\ref{FY_3}),  we derive from \eqref{yhfF_101}, (\ref{eq02}) and \eqref{YH_2} that
\begin{align*}
 &\E\Big[ (1+\tilde{Y}^T_{k+1}Q\tilde{Y}_{k+1})^{\frac{p}{2}}\xi^{p,\alpha}_{r_{k+1}}|\F_{t_k}\Big]\nn\\
   \leq & \left(1+Y^T_kQY_k\right)^{\frac{p}{2}} \Big\{\xi^{p,\alpha}_{r_k}+\sum_{j\in \mathbb{S}}\xi^{p,\alpha}_{j}\gamma_{r_k j}\t  +o\left(\t\right)\nn\\
   &~~~~~~~~~~~~~~~~~~~~~~~~~+ \frac{p\t}{2}\frac{(1+Y^T_kQY_k)\psi(Y_k,r_k)+(p-2)| Y_k^TQg(Y_k, r_k)|^2}{(1+Y^T_kQY_k)^2}\xi^{p,\alpha}_{r_k}\Big\}\nn\\
   \leq & \left(1+Y^T_kQY_k\right)^{\frac{p}{2}} \Big[\xi^{p,\alpha}_{r_k} -\frac{\eta_{p, \alpha}}{2}\xi^{p,\alpha}_{r_k}\t  +o\left(\t\right)\Big]+C\t.
\end{align*}
  Choose    $ {\t} ^*\in(0,1] $ sufficiently small
such that
${\t} ^*< 4/\eta_{p,\alpha}$,   $o\left(  {\t} ^*\right)\leq \eta_{p,\alpha}\hat{\xi}^{p,\alpha} {\t} ^*/4.
$ Taking expectations on both sides, for any $\t\in(0, {\t} ^*]$, yields
\begin{align*}
\E\left[\left(1+Y^T_{k+1}QY_{k+1}\right)^{\frac{p}{2}} \xi^{p,\alpha}_{r_{k+1}}\right] &\leq\E\left[(1+\tilde{Y}^T_{k+1}Q\tilde{Y}_{k+1})^{\frac{p}{2}}\xi^{p,\alpha}_{r_{k+1}}\right]\\
 & \leq   \Big(1-  \frac{\eta_{p,\alpha}}{4}  \t\Big)\E\Big[\left(1+Y^T_kQY_k\right)^{\frac{p}{2}} \xi^{p,\alpha}_{r_k} \Big]+C\t.
 \end{align*}
for any integer $k\geq 0$. Repeating this procedure arrives at
\begin{align*} 
\E\left[\left(1+Y^T_{k}QY_{k}\right)^{\frac{p}{2}} \xi^{p,\alpha}_{r_{k}}\right]
 \leq& \mathrm{e}^{- \frac{\eta_{p,\alpha}}{4} k\t}(x^T_{0}Qx_{0})^{\frac{p}{2}}\xi^{p,\alpha}_{\ell}+\frac{4C}{\eta_{p,\alpha}}\left[1-\left(1-  \frac{\eta_{p,\alpha}}{4}  \t\right)^{k}\right]\leq C.
\end{align*}
Therefore,
$
\sup\limits_{k\geq0}\mathbb E|Y_{k}|^{p}
\leq C.
$
The desired assertion follows. \end{proof}

\subsection{Invariant measure}
In this subsection, we discuss the criterion on the existence and uniqueness of the invariant measure  for  SDS \eqref{e1}. It follows from   SDS (\ref{e1}) that
\begin{align}\label{y++}
\mathrm{d}\big(X^{x_0, \ell}_t-X^{\bar{x}_0, \ell}_t\big)&=\big(f(X^{x_0, \ell}_t, r_t^{\ell})-f(X^{\bar{x}_0, \ell}_t, r_t^{\ell})\big) \mathrm{d}t
 +\big(g(X^{x_0, \ell}_t, r_t^{\ell})-g(X^{\bar{x}_0, \ell}_t, r_t^{\ell})\big)\mathrm{d}B(t)\nn\\
&=:F(X^{x_0, \ell}_t, X^{\bar{x}_0, \ell}_t, r_t^{\ell}) \mathrm{d}t+G(X^{x_0, \ell}_t, X^{\bar{x}_0, \ell}_t, r_t^{\ell})\mathrm{d}B(t)
\end{align}
for any two initial values  $(x_0, \ell)$, $(\bar{x}_0,  \ell)\in \mathbb{K}\times\mathbb{S}$, where $\mathbb{K}$ is a compact set in $\mathbb{R}^n$.
For convenience we impose the following hypothesis.
\begin{assp}\label{a5}
 There exists a positive constant  $\rho>0$   such that
\begin{equation*}
|x-y|^{2}\psi(x,y,i) -(2-  \rho)|(x-y)^{T}G(x,y,i)|^{2}\leq \beta_{i}|x-y|^{4},~~~~\forall x,~y\in\RR^n,~ i\in \mathbb{S},
\end{equation*}
where
$  \psi(x,y,i) = 2(x-y)^{T}F(x,y,i)+|G(x,y,i)|^{2},$ and each $\beta_i $ is a constant.
\end{assp}

In order for the unique invariant measure we provide the asymptotic  attractivity  of    SDS (\ref{e1}).
\begin{lemma}\la{Tyle2.3}
Suppose that  Assumption \ref{a1} with $Q_i\equiv Q$ and  Assumption  \ref{a5} hold, $\pi \alpha<0$ and $\pi  \beta<0$. Then
  for any   $p\in (0,p^*_\alpha\wedge p^*_{\beta}\wedge\bar{p})\cap (0, \rho ]$, any compact $\mathbb{K} \subset \RR^n$,   solutions $X^{x_0, \ell}_t$ and $X^{\bar{x}_0, \bar{\ell}}_t$ of    SDS (\ref{e1}) with  initial data $(x_0, \ell)$, $(\bar{x}_0, \bar{\ell})\in \mathbb{K}\times\mathbb{S}$ respectively satisfy
\begin{align}\la{Typ0}
 \lim_{t\rightarrow \infty}\mathbb{E}|X^{x_0, \ell}_t-X^{\bar{x}_0, \bar{\ell}}_t|^p  =0,
\end{align} where each $\beta_i$ is given by  Assumption  \ref{a5} and  $\beta:=(\beta_1,\dots, \beta_m)^T$.
\end{lemma}
\begin{proof}   For the given $\beta $ and any $p>0$,   $\mathrm{diag}(\beta), ~\Gamma_{p,\beta},~ \eta_{p,\beta},~p^*_\beta$ are defined as \eqref{yhfYH_01} and \eqref{yhfEq:1}. Fix $p\in (0,p^*_\beta)\cap (0,  \rho]$.    Lemma \ref{L:1} implies that $\eta_{p,\beta}>0$ and there exists  $\xi^{p,\beta}\gg \mathbf{0}$ such that
\begin{align}\la{TyhfF_1+}
 \sum_{j= 1}^{m}\gamma_{ij}\xi^{p,\beta}_j+\frac{p\beta_i}{2}\xi^{p,\beta}_i=-\eta_{p,\beta}\xi^{p,\beta}_i<0,~~~~~~i=1,\dots,m
\end{align}
holds. 
By generalised It\^{o}'s formula (see  e.g., \cite[Theorem 1.45, p.48]{Mao06}), we obtain
\begin{align} \la{yhfF_2}
&\mathrm{e}^{\eta_{p,\beta} t}\big|X^{x_0, \ell}_t-X^{\bar{x}_0,  \ell}_t\big|^{p}\xi^{p,\beta}_{r(t)}\nn\\
\!\!  =& \left|x_0\!-\bar{x}_0\right|^{p}\xi^{p,\beta}_{\ell}  +\int_{0}^{t}\mathrm{e}^{\eta_{p,\beta} s}\bigg\{\eta_{p,\beta}\big|X^{x_0, \ell}_t\!-X^{\bar{x}_0,  \ell}_t\big|^{p}\xi^{p,\beta}_{r(s)}
\! + {\cal{L}}\Big(\big|X^{x_0, \ell}_t\!-X^{\bar{x}_0,  \ell}_t\big|^{p}\xi^{p,\beta}_{r(s)}\Big)\bigg\}\mathrm{d}s\!+M(t),\!\!
\end{align}
where
\begin{align*}
 {\cal{L}}\left(|x-y|^{p} \xi^{p,\beta}_i \right)
=  |x-y|^{p}\left[ \frac{p\xi^{p,\beta}_i}{2 }
\frac{|x-y|^2\psi(x,y,i)-(2-p)|(x-y)^T G(x,y, i)|^2}{|x-y|^{4}}
 +\sum_{j=1}^{m} \gamma_{ij}\xi^{p,\beta}_j \right],
\end{align*}
and
\begin{align*}
M(t)=& p \int_{0}^{t}\mathrm{e}^{\eta_{p,\beta} s} \xi^{p,\beta}_{r(s)}\big|X^{x_0, \ell}_t-X^{\bar{x}_0,  \ell}_t\big|^{p-2} \big(X^{x_0, \ell}_t-X^{\bar{x}_0,  \ell}_t\big)^TG(X^{x_0, \ell}_t, X^{\bar{x}_0,  \ell}_t,r(s))\mathrm{d}B(s)\nn\\
&+\int_{0}^{t}\int_{\mathbb{R}}\mathrm{e}^{\eta_{p,\beta} s}\big|X^{x_0, \ell}_t-X^{\bar{x}_0,  \ell}_t\big|^{p}\big(\xi^{p,\beta}_{\ell+\hbar(r(s),l)}-\xi^{p,\beta}_{r(s)}\big)\mu(\mathrm{d}s, \mathrm{d}l)
\end{align*}
where $\mu(\mathrm{d}s, \mathrm{d}l)=\nu(\mathrm{d}s, \mathrm{d}l)-\mu(\mathrm{d}l)\mathrm{d}s$ is a martingale measure,
the definition of $\hbar$ can be found in \cite[p.48]{Mao06}, and $M(t)$
is a real-valued local continuous martingale (see \cite{Mao06}) with $M(0)=0$. 
Thus, taking expectations on both sides, using (\ref{TyhfF_1+}) and  Assumption   \ref{a5}  implies
\begin{align}\la{Y-HF1}
 \mathbb{E} |X_t^{x_0, \ell}-X_t^{\bar{x}_0, \ell}|^{p}
 \leq  C\exp{\big(-\eta_{p,\beta} t\big)}
\end{align}
for   any two initial values  $(x_0, \ell)$, $(\bar{x}_0,  \ell)\in \mathbb{K}\times\mathbb{S}$.
 Define a stopping time
 $
 \bar{\tau}=\inf\{t\geq0:r^{\ell}_t=r^{\bar{\ell}}_t\}.
 $
Due to the irreduciblility of $r(\cdot)$, there exists  a constant $\bar{\lambda}>0$ such that
\begin{equation}\label{equ8}
 \mathbb P(\bar{\tau}>t)\leq \mathrm{e}^{-\bar{\lambda} t},~\forall t>0.
 \end{equation}
 For any $p\in (0,p^*_\alpha\wedge p^*_{\beta}\wedge\bar{p})\cap (0, \rho ]$, choose $q>1$ such that $pq\in(0,p^*_\alpha )\cap (0,  \bar{p}]$. Using H\"{o}lder's inequality,  \eqref{Y-HF1} and \eqref{equ8}  yields
\begin{align}\label{eqn4}
&\mathbb E|X^{x_0, \ell}_t-X^{\bar{x}_0, \bar{\ell}}_t|^{p}\nn\\
=&\mathbb E\big(|X^{x_0, \ell}_t-X^{\bar{x}_0, \bar{\ell}}_t|^{p}I_{\{\bar{\tau}>t/2\}}\big)+\mathbb E\big(|X^{x_0, \ell}_t-X^{\bar{x}_0, \bar{\ell}}_t|^{p}I_{\{\bar{\tau}\leq t/2\}}\big)\nn \\
\leq&\big(\mathbb E|X^{x_0, \ell}_t-X^{\bar{x}_0, \bar{\ell}}_t|^{pq}\big)^{\frac{1}{q}}\big[\mathbb P\big(\bar{\tau}>t/2\big)\big]^{1-\frac{1}{q}}+\mathbb E\Big[I_{\{\bar{\tau}\leq t/2\}}\mathbb E\big(|X^{x_0, \ell}_t-X^{\bar{x}_0, \bar{\ell}}_t|^{p}|\mathcal F_{\bar{\tau}}\big)\Big]\nn \\
=&\big(\mathbb E|X^{x_0, \ell}_t-X^{\bar{x}_0, \bar{\ell}}_t|^{pq}\big)^{\frac{1}{q}}\big[\mathbb P\big(\bar{\tau}>t/2\big)\big]^{1-\frac{1}{q}}
 +\mathbb E\Big[I_{\{\bar{\tau}\leq t/2\}}\mathbb E\Big(\big|X^{{X^{x_0, \ell}_{\bar{\tau}},r^{\ell}_{\bar{\tau}}}}_{t-\bar{\tau}}
-X^{{X^{\bar{x}_0, \bar{\ell}}_{\bar{\tau}},r^{\bar{\ell}}_{\bar{\tau}}}}_{t-\bar{\tau}}\big|^{p}\Big)\Big]\nn \\
\leq&\mathrm{e}^{-\frac{q-1}{2q}\bar{\lambda} t}\Big(\mathbb E\big|X^{x_0, \ell}_t-X^{\bar{x}_0, \bar{\ell}}_t\big|^{pq}\Big)^{\frac{1}{q}}
+C\mathbb E\Big(I_{\{\bar{\tau}\leq t/2\}}\big|X_{\bar{\tau}}^{x_0, \ell}
-X^{\bar{x}_0, \bar{\ell}}_{\bar{\tau}}\big|^{p}\mathrm{e}^{-\eta_{p, \beta}(t-\bar{\tau})}\Big).
\end{align}
For any $pq\in(0,p^*_\alpha )\cap (0,  \bar{p}]$,  it follows from  Theorem \ref{Tyle2.2+} that
$
\sup\limits_{t\geq0}\mathbb E\big|X_t^{x_0, \ell}\big|^{pq}\leq C,~
\sup\limits_{t\geq0}\mathbb E\big|X_t^{\bar{x}_0, \bar{\ell}}\big|^{pq}\leq C,
$
 and
$$
 C\mathbb{E}\Big(I_{\{\bar{\tau}\leq t/2\}}\big|X_{\bar{\tau}}^{x_0, \ell}
-X^{\bar{x}_0, \bar{\ell}}_{\bar{\tau}}\big|^{p}\mathrm{e}^{-\eta_{p, \beta}(t-\bar{\tau})}\Big)
\leq   C\mathrm{e}^{-\frac{\eta_{p, \beta}}{2}t}.
$$
Thus the desired assertion (\ref{Typ0}) follows from (\ref{eqn4}). The proof is complete.
\end{proof}

As well as we know  the stability of the trivial solution is one of the major concerns in many applications. As a corollary, the result on the stability follows from the proof of Lemma \ref{Tyle2.3} directly. For clarity we impose the following assumption.
  \begin{assp}\label{a2}
If
\be\la{eq0solu}
f(\mathbf{0},i)= \mathbf{0},~~~~~g(\mathbf{0},i)= \mathbf{0},  ~~  \forall i\in \mathbb{S}
\ee
holds,  and  for  some $\rho>0$,
 there exists a symmetric positive-definite matrix $ Q\in \RR^{n\times n}$  such that
\begin{align}\label{yhfF1}
 (x^TQx) \psi(x,i)- (2-\rho)|x^TQg(x, i)|^2 \leq   \beta_i (x^TQx)^2,~~~~\forall x\in \RR^n ,~i\in \mathbb{S}
\end{align} holds,  where $
\psi(x,i):=2x^T Q f(x, i)+\trace{\left(g^T(x, i)Q g(x, i)\right)},
$ and   $\beta_i $ is a constant.
  \end{assp}
Due to the above assumption, one observes that
$X(t)\equiv \mathbf{0}$ is  the {\it trivial solution} of SDS (\ref{e1}) with initial values $ X(0)=\mathbf{0}$, any  $r(0)\in \SS $.

 \begin{cor}\la{CT:F_0}
Suppose that  Assumption   \ref{a2} holds and $\pi\beta<0$. Then  for any $p\in (0, p^*_\beta) \cap(0, \rho]$,
 the solution $X(t)$ of SDS (\ref{e1})
 has the  property that
  \be\la{yhfF_0}
   \limsup_{t\rightarrow \infty}\frac{\log(\E|X(t)|^{p  })}{t} \leq  -\eta_{p, \beta}<0,\ee
and
  \be\la{yhf0F_0}
   \limsup_{t\rightarrow \infty}\frac{\log(|X(t)|)}{t} \leq  -\frac{\eta_{p, \beta}}{p} ~~~~a.s.\ee
  where each $\beta_i$ is given by Assumption   \ref{a2}, $\beta=:(\beta_1,\dots,\beta_m)^T$ and $\eta_{p, \beta}$ is defined by \eqref{yhfYH_01}.
   \end{cor}

   \begin{proof}  In the same way as  the proof of  Lemma \ref{Tyle2.3}, for any initial data $(x_0,\ell)\in \RR^n\times \SS$, the solution $ X_t^{ {x}_0, \ell}$ satisfies  inequality  \eqref{Y-HF1} with $ X_t^{ \bar{x}_0, \ell}  \equiv\mathbf{0}$. Then (\ref{yhfF_0}) follows directly.  On the other hand,
 due to Assumption \ref{a2} and   (\ref{yhfF_2}) with $ X_t^{ \bar{x}_0, \ell}  \equiv\mathbf{0}$, by  the nonnegative semimartingale convergence theorem (see, e.g., \cite[p.18, Theorem 1.10]{Mao06}), we obtain
$
\limsup_{t\rightarrow \infty} \hat{\xi}^{p,\beta} \mathrm{e}^{\eta_{p,\beta} t} \big|X^{x_0, \ell}_t\big|^{p} <\infty~a.s.
$
Therefore the other required assertion follows. The proof is complete.
\end{proof}

Next we give the existence and uniqueness of the invariant measure for the solution  $(X^{x_0, \ell}_t,r^{\ell}_t)$  of SDS (\ref{e1}).
 Let $\mathbf{P}_{t}(x_0, \ell;\mathrm{d}x\times \{i\})$ be the transition probability kernel of the pair $\big(X_t^{x_0, \ell},r^{\ell}_t\big)$, a time homogeneous Markov process (see, e.g, \cite[Theorem 3.27, pp.104-105]{Mao06}). Recall that $\mu\in\mathcal P(\RR^n\times S)$ is called an invariant measure of $\big(X_t^{x_0, \ell},r_t^{\ell}\big)$ if
\begin{align*}
\mu(\Upsilon\times\{\ell\})=\sum_{i=1}^{m}\int_{\RR^n} \mathbf{P}_{t}(x, i; \Upsilon\times\{\ell\})\mu(\mathrm{d}x\times\{i\}),~~~\forall t\geq0,~\Upsilon\in\mathscr{B}(\RR^n),~\ell\in \SS
\end{align*}
holds.
\begin{theorem}\la{yth3.1}
Under the conditions of Lemma  \ref{Tyle2.3},
 the solution  $(X^{x_0, \ell}_t,r^{\ell}_t)$   admits a unique invariant measure $\mu\in \mathcal{P}(\RR^n\times \SS)$.
\end{theorem}
\begin{proof}   For arbitrary $t>0$, define a probability measure
$$
 \hbar_t(\Upsilon)=\frac{1}{t}\int_{0}^{t}\mathbf{P}_{s}(x, i;\Upsilon)\mathrm{d}s,~~~~~~\Upsilon\in \mathscr{B}(\RR^n\times \mathbb{S}).
$$
Then, for any $\varepsilon>0$, by Theorem \ref{Tyle2.2+} and Chebyshev's inequality, there exists a positive integer  $N>0$ sufficiently large such that
$$
 \hbar_t(B_{N}(\mathbf{0})\times \mathbb{S})=\frac{1}{t}\int_{0}^{t}\mathbf{P}_{s}(x_0, \ell;B_{N}(\mathbf{0})\times \mathbb{S})\mathrm{d}s \geq 1-\frac{\sup_{t\geq 0}\mathbb{E}|X_t^{x_0, \ell}|^p}{N^p}\geq 1-\varepsilon
$$
for any $p\in (0,p^*_\alpha)\cap (0, \bar{p}]$. Hence $\{ \hbar_t\}_{t\geq 0}$ is tight since $B_{N}(\mathbf{0})$ is a compact subset of $\RR^n$. Moreover, one observes that $(X_t^{x_0, \ell},r_t^{\ell})$ enjoys the Feller property (see, e.g., \cite[Theorem 2.18, p.48]{yz09}). Thus, the solution $(X_t^{x_0, \ell},r_t^{\ell})$ of  SDS  (\ref{e1})    has an invariant measure (see, e.g., \cite[Theorem 4.14, p.128]{Chen2004}).
  We have established the existence of the  invariant measure, and  next we go a further step to show its uniqueness. For any $p\in (0,p^*_\alpha\wedge p^*_{\beta}\wedge\bar{p})\cap (0, \rho]$,   (\ref{Typ0}) and the irreducibility of $r(\cdot)$ result  in
\begin{align}\label{eqn5}
\!\!\! W_{p}\big(\delta_{(x_0, \ell)}\mathbf{P}_{t},\delta_{(\bar{x}_0, \bar{\ell})}\mathbf{P}_{t}\big)
 \leq  \mathbb E|X^{x_0, \ell}_t-X^{\bar{x}_0, \bar{\ell}}_t|^{p}+\mathbb P\big(r^{\ell}_t\neq r^{\bar{\ell}}_t\big)\rightarrow 0,~~t\rightarrow\infty,
\end{align}
where $\delta_{(x,i)}$ stands for the Dirac's measure at the point $(x,i)$.
 Assume both $\mu$ and $\bar{\mu}$ are invariant measures, then we have
\begin{align*}
W_{p}(\bar{\mu},\mu) &=W_{p}(\bar{\mu} \mathbf{P}_{t},\mu \mathbf{P}_{t})
 =
\sup_{\bar{\psi}:\textrm{Lip}(\bar{\psi})\leq 1}\Big\{\int_{\RR^{n}\times \mathbb{S}}\bar{\psi}(x, i)\mathrm{d}(\bar{\mu} \mathbf{P}_{t})-\int_{\RR^{n}\times \mathbb{S}}\bar{\psi}(y, j)\mathrm{d}(\mu  \mathbf{P}_{t})\Big\}\\
\notag
 &\leq\int_{\RR^{n}\times \mathbb{S}}\int_{\RR^{n}\times \mathbb{S}}\bar{\mu}\big(\mathrm{d}x\times {\{i\}}\big)\mu\big(\mathrm{d}y\times {\{j\}}\big)W_{p}\big(\delta_{(x,i)}\mathbf{P}_{t},\delta_{(y,j)}\mathbf{P}_{t}\big).
\end{align*}
One then observes that, due to  (\ref{eqn5}),
$W_{p}(\bar{\mu},\mu)\rightarrow 0 $ as $t\rightarrow\infty.$
The desired assertion follows. \end{proof}

\subsection{The invariant measure of numerical solution}
\label{S-dist} In order to approximate  the invariant measure $\mu$ of  SDS (\ref{e1}) we need to construct the appropriate  scheme such that   the numerical solutions are attractive in $\r$th moment and admit a unique numerical invariant measure. However, the   proposed  mapping  $\pi_\t^i(x)$  is not suitable for the avoidance of   extra wide divergence  in distance between two different numerical solutions. Thus we construct a new    truncation mapping $\bar{\pi}_\t^i(x)$ according to the local Lipschitz continuity  of the drift and diffusion coefficients. Then making use of the appropriate truncation mapping we give an  explicit scheme. Finally we show that,  it produces  a unique numerical invariant measure $\mu^\t$ which tends to the invariant measure $\mu$ of SDS (\ref{e1}) as $\t \rightarrow0$ in the Wasserstein metric.

For each $i\in \Se$ choose a strictly increasing continuous functions $\bar{\f}_i: \RR_+\rightarrow \RR_+$ such that $\bar{\f}_i(u)\rightarrow \infty$
as $u\rightarrow \infty$ and
\begin{align}\la{6Fe21}
 \sup_{|x|\vee|y|\leq u, x\neq y}  \Big(\frac{ |f(x, i)-f(y,i)|}{|x-y|}\vee \frac{|g(x, i)-g(y, i)|^2}{|x-y|^2}\Big)\leq \bar{\f}_i(u),~~\forall~ u\geq 1.
\end{align}
Due to the local Lipschitz continuity of $f$ and $g$  the function $ \bar{\f}_i$ can be well defined as well as its inverse function  $\bar{\f}_i^{-1}: [\bar{\f}_i(1),\infty)\rightarrow \RR_+ $.
We choose  a  strictly decreasing $\bar{h}:(0, 1]\rightarrow [\check{\bar{\f}}(|x_0|\vee 1), \infty)$ such that
\begin{align}\la{6Fe22}
 \bar{h}(1)\geq   \max_{i\in \mathbb{S}}\{|f(0,i)|\vee |g(0,i)|^2\}, ~~
 \lim_{\t\rightarrow 0} \bar{h}(\t)= \infty,~~\hbox{and} ~~\t^{1/2- \bar{\theta}} \bar{h}(\t)\leq  {K},~~~\forall \t \in (0, 1]
\end{align}
for  some $0<\bar{\theta}< {1}/{2}$ holds, where $ {K}$ is a positive constant  independent of   $k$ and  $\t$.
For a given stepsize $\t\in (0, 1]$, let us define the truncated mapping $\bar{\pi}^{i}_{\t}:\RR^n\ra \RR^n$ by
\begin{align}\la{6Fe23}
\bar{\pi}^i_\t(x)= \Big(|x|\wedge  \bar{\f}_i^{-1}(\bar{h}(\t))\Big) \frac{x}{|x|},
\end{align}
where we let $\frac{x}{|x|}=\mathbf{0}$ when $x=\mathbf{0}\in \mathbb{R}^n$. Obviously, for any $i\in \mathbb{S}$ and $x,~y\in\RR^n$,
\begin{align} \la{Fe24}
 &|f(\bar{\pi}^i_\t(x),i)-f(\bar{\pi}^i_\t(y), i)|
\leq    \bar{h}(\t) |\bar{\pi}^i_\t(x)-\bar{\pi}^i_\t(y)|,\nn\\
&|g(\bar{\pi}^i_\t(x), i)-g(\bar{\pi}^i_\t(y), i)|^2
\leq    \bar{h}(\t) |\bar{\pi}^i_\t(x)-\bar{\pi}^i_\t(y)|^2,
\end{align}
and
\begin{align} \la{Fe25}
|f(\bar{\pi}^i_\t(x),i)| \leq   \bar{h}(\t)\big(1+ |\bar{\pi}^i_\t(x)|\big),~~~
|g(\bar{\pi}^i_\t(x), i)|^2  \leq   \bar{h}(\t)\big(1+ |\bar{\pi}^i_\t(x)|\big)^2.
\end{align}

\begin{rem}
If there exists the  state  $i \in  \mathbb{S}$ such that
\begin{align}\la{trun3}
|f(x, i)-f(y, i)| \leq C|x-y|,~~~~ |g(x, i)-g(y, i)|^2\leq C|x-y|^2,~~~\forall x, y\in \mathbb{R}^n,
\end{align}
   we  choose $\bar{\f}_i(u)\equiv C$ for any   $u\geq 1$, and  let $\bar{\f}^{-1}_i(u)\equiv +\infty$ for any $u\in [C, +\infty)$. Then
$\bar{\pi}_\t^i(x)=x$, (\ref{Fe24}) and (\ref{Fe25}) hold always.
\end{rem}

For any given stepsize $\t\in (0, 1]$, define a new  truncated EM method scheme by
 \begin{align}\la{6FY_0}
\left\{
\begin{array}{ll}
\tilde{Z}_0=x_0,~~r_0=\ell,&\\
Z_{k}=\bar{\pi}^{r_{k}}_\t(\tilde{Z}_{k}),\\
\tilde{Z}_{k+1}=Z_{k}+f(Z_{k}, r_k)\t +g(Z_{k}, r_k)\t B_k,
\end{array}
\right.
\end{align}
for any integer  $k\geq 0$. To obtain the continuous-time approximation,  define
$ Z(t):=Z_k$ for any $t\in[t_k,  t_{k+1})$.  We write $(Z_k^{x_0, \ell},r_k^{\ell})$ in lieu of
$(Z_k,r_k)$ to highlight the initial data $(Z_0,r_0)=(x_0,\ell)$.
 Then we have
 \begin{align}\la{I-01}
\begin{array}{lll}
&|f(Z^{\bar{x}_0,\ell}_{k},r^{\ell}_k)-f(Z^{\bar{x}_0,\ell}_{k}, r^{\ell}_k)|
\leq   \bar{h}(\t) |Z^{x_0,\ell}_{k}-Z^{\bar{x}_0,\ell}_{k}|,\\
&|g(Z^{x_0,\ell}_{k}, r^{\ell}_k)-g(Z^{\bar{x}_0,\ell}_{k}, r^{\ell}_k)|^2
\leq    \bar{h}(\t) |Z^{x_0,\ell}_{k}-Z^{\bar{x}_0,\ell}_{k}|^2,
\end{array}
\end{align}
and linear property
\begin{align}\la{I-02}
\!\!\!|f(Z^{x_0,\ell}_{k},r^{\ell}_k)|  \leq   \bar{h}(\t)\big(1+ |Z^{x_0,\ell}_{k}|\big),~
|g(Z^{x_0,\ell}_{k}, r^{\ell}_k)|^2  \leq   \bar{h}(\t)\big(1+ |Z^{x_0,\ell}_{k}|\big)^2.
\end{align}
 Thus,  the above properties support  that  the conclusion of Theorem \ref{T:C_2} holds for Scheme \eqref{6FY_0}.
In order for the uniqueness of the numerical invariant measure we prepare  the   attractive property of the numerical solutions.
\begin{lemma}\la{yle3.5}
Under the conditions of Lemma \ref{Tyle2.3}, there is a constant ${\t}^{**} \in(0,\t^*]$  such that  for any $p\in (0,p^*_\alpha\wedge p^*_{\beta}\wedge\bar{p})\cap (0,\rho]$, Scheme \eqref{6FY_0} has the   property that
\begin{align}\label{Tty3.8}
\lim\limits_{k \rightarrow \infty}\mathbb E|Z^{x_0,\ell}_{k}-Z^{\bar{x}_0,\bar{\ell}}_{k}|^{p}
=0
\end{align}
for any $\t \in (0, \t^{**}]$,   $(x_0,\ell), (\bar{x}_0,\bar{\ell})\in \mathbb{K} \times\mathbb{S}$, where   $\t^*$ is given in  Theorem \ref{4TT:F_1}.
\end{lemma}
\begin{proof}
Recall the definitions of  $F$ and $G$ in (\ref{y++}), we know that
 \begin{align}\la{E_y}
 \tilde{Z}^{x_0,\ell}_{k+1}-\tilde{Z}^{\bar{x}_0,\ell}_{k+1}
   =  Z^{x_0,\ell}_{k}-Z^{\bar{x}_0,\ell}_{k}+ F(Z^{x_0,\ell}_{k},Z^{\bar{x}_0,\ell}_{k},r_{k}^{\ell})\t+G(Z^{x_0,\ell}_{k},Z^{\bar{x}_0,\ell}_{k},r_{k}^{\ell})\t B_{k}.
\end{align}
Without loss the generality,  we prove \eqref{Tty3.8} only for  $0<p\leq 2$. 
For any  $p\in (0, p^*_\beta) \cap(0, \rho\wedge 2]$, any $\varepsilon>0$, we derive from \eqref{Y2} and   \eqref{E_y} that
\begin{align} \la{FY_3+}
  &  \E \Big[ \left(\varepsilon+\big|\tilde{Z}^{x_0,\ell}_{k+1}-\tilde{Z}^{\bar{x}_0,\ell}_{k+1}\big|^2\right)^{\frac{p}{2}}\xi^{p,\beta}_{r_{k+1}} \big|\mathcal{F}_{t_k}\Big]  \nonumber\\
 \leq & \left(\varepsilon+ \big|Z^{x_0,\ell}_{k}-Z^{\bar{x}_0,\ell}_{k}\big|^2\right)^{\frac{p}{2}}\bigg\{\E\big[\xi^{p,\beta}_{r_{k+1}}\big|\mathcal{F}_{t_k}\big]+ \frac{p}{2} \E\big[\tilde{\varsigma}_k\xi^{p,\beta}_{r_{k+1}}\big|\mathcal{F}_{t_k}\big] + \frac{p(p-2)}{8} \E\big[\tilde{\varsigma}_k^2\xi^{p,\beta}_{r_{k+1}}\big|\mathcal{F}_{t_k}\big]\nn\\
 &~~~~~~~~~~~~~~~~~~~~~~~~~~~~~~~~~~ + \frac{p(p-2)(p-4)}{48} \E\big[\tilde{\varsigma}_k^3\xi^{p,\beta}_{r_{k+1}}\big|\mathcal{F}_{t_k}\big]\bigg\},
\end{align}
   where $\xi^{p,\beta}$ satisfying  \eqref{TyhfF_1+} is given in the proof of Lemma \ref{Tyle2.3}, and
\begin{align*}
\tilde{\varsigma}_k \!= \Big(\varepsilon \!+ \big|Z^{x_0,\ell}_{k}\!-Z^{\bar{x}_0,\ell}_{k}\big|^2\Big)^{-1}&\Big[ 2\big(Z^{x_0,\ell}_{k}\!-Z^{\bar{x}_0,\ell}_{k}\big)^T F(Z^{x_0,\ell}_{k},Z^{\bar{x}_0,\ell}_{k},r_{k}^{\ell})\t+  \big|G(Z^{x_0,\ell}_{k},Z^{\bar{x}_0,\ell}_{k},r_{k}^{\ell})\t B_k\big|^2\\
+&  2\big(Z^{x_0,\ell}_{k}-Z^{\bar{x}_0,\ell}_{k}\big)^T G(Z^{x_0,\ell}_{k},Z^{\bar{x}_0,\ell}_{k},r_{k}^{\ell})\t B_k+ \big|F(Z^{x_0,\ell}_{k},Z^{\bar{x}_0,\ell}_{k},r_{k}^{\ell})\big|^2\t^2\nn\\
 +&2F^T(Z^{x_0,\ell}_{k},Z^{\bar{x}_0,\ell}_{k},r_{k}^{\ell})G(Z^{x_0,\ell}_{k},Z^{\bar{x}_0,\ell}_{k},r_{k}^{\ell})\t B_k \t\Big].
\end{align*}
Using the techniques in the proof of Theorem \ref{4TT:F_1},  by \eqref{I-01},  we deduce that
\begin{align}
 &\E\big[\tilde{\varsigma}_k\xi^{p,\beta}_{r_{k+1}}\big|\mathcal{F}_{t_k}\big]  = \E\big[\E\big(\tilde{\varsigma}_k\xi^{p,\beta}_{r_{k+1}}\big|\mathcal{G}_{t_k}\big)\big|\mathcal{F}_{t_k}\big]
\leq    \frac{\psi(Z^{x_0,\ell}_{k},Z^{\bar{x}_0,\ell}_{k},r_{k}^{\ell})\xi^{p,\beta}_{r_{k}}\t}{\varepsilon+\big|Z^{x_0,\ell}_{k}-Z^{\bar{x}_0,\ell}_{k}\big|^2}+ C\t^{1+2\bar{\theta}}, \la{FY_4+}\\
& \E\Big[\tilde{\varsigma}_k^2\xi^{p,\beta}_{r_{k+1}}|\F_{t_k}\Big]
  \geq   \frac{4| (Z^{x_0,\ell}_{k}-Z^{\bar{x}_0,\ell}_{k} )^TG(Z^{x_0,\ell}_{k},Z^{\bar{x}_0,\ell}_{k},r_{k}^{\ell})|^2\xi^{p,\beta}_{r_{k}} \t}{\big(\varepsilon+ |Z^{x_0,\ell}_{k}-Z^{\bar{x}_0,\ell}_{k}|^2\big)^2} -C\t^{1+2\bar{\theta}},\la{FY_5+}
\end{align}
and
\begin{align}\la{FY_6+}
 \E\Big[\tilde{\varsigma}_k^3\xi^{p,\beta}_{r_{k+1}}|\F_{t_k}\Big]
 =\E\Big[\E\big(\tilde{\varsigma}_k^3\xi^{p,\beta}_{r_{k+1}}\big|\mathcal{G}_{t_k}\big)\big|\mathcal{F}_{t_k}\Big]
 \leq  C\t^{1+2\bar{\theta}}.
\end{align}
Similarly, we  can also prove that for any integer $l>3$, $ \E\big[|\tilde{\varsigma}_k|^l\xi^{p,\beta}_{r_{k+1}}|\F_{t_k}\big]\leq C \t^{1+2\bar{\theta}}.$   For any integer $k\geq 0$,
 substituting (\ref{FY_4+})-(\ref{FY_6+}) into (\ref{FY_3+}), we deduce from (\ref{TyhfF_1+}) that
\begin{align*}
 &\E\Big[ \big(\varepsilon+\big|\tilde{Z}^{x_0,\ell}_{k+1}-\tilde{Z}^{\bar{x}_0,\ell}_{k+1}\big|^2\big)^{\frac{p}{2}}\xi^{p,\beta}_{r_{k+1}}|\F_{t_k}\Big]\nn\\
   \leq & \left(\varepsilon+ \big|Z^{x_0,\ell}_{k}-Z^{\bar{x}_0,\ell}_{k}\big|^2\right)^{\frac{p}{2}} \bigg\{\xi^{p,\beta}_{r_k}-\frac{p\beta_{r_k}}{2}\xi^{p,\beta}_{r_k}\t-\eta_{p,\beta}\xi^{p,\beta}_{r_k}\t  +o\left(\t\right)\nn\\
   &+ \frac{p\t}{2}\frac{(\varepsilon\!+ |Z^{x_0,\ell}_{k}\!-Z^{\bar{x}_0,\ell}_{k}|^2)\psi(Z^{x_0,\ell}_{k},Z^{\bar{x}_0,\ell}_{k},r_{k}^{\ell})
   \!+(p-2)\big| (Z^{x_0,\ell}_{k}\!-Z^{\bar{x}_0,\ell}_{k})^T G(Z^{x_0,\ell}_{k},Z^{\bar{x}_0,\ell}_{k},r_{k}^{\ell})\big|^2}{\big(\varepsilon+|Z^{x_0,\ell}_{k}-Z^{\bar{x}_0,\ell}_{k}|^2\big)^2}\xi^{p,\beta}_{r_k}\bigg\}.
\end{align*}
For any  given $\varrho\in(0, \eta_{p,\beta})$,   choose    $ \t^{**}\in(0, \t^{*}] $ sufficiently small
such that
$
\t^{**}< 1/(\eta_{p,\beta}-\varrho),$  $o\big( \t^{**}\big)\leq\varrho\hat{\xi}^{p,\beta} \t^{**}.
$
Taking expectations on both sides and  letting $\varepsilon\downarrow 0$, then using the theorem on monotone convergence 
and   Assumption \ref{a5},  for any $\t\in(0, \t^{**}]$, yields
\begin{align*}
 \E\left[\big|\tilde{Z}^{x_0,\ell}_{k+1}-\tilde{Z}^{\bar{x}_0,\ell}_{k+1}\big|^p{\xi}^{p, \beta}_{r^{\ell}_{k+1}}\right]
\leq \left(1+\varrho\t-\eta_{p,\beta}\t \right)\E \Big[\big|Z^{x_0,\ell}_{k}-Z^{\bar{x}_0,\ell}_{k}\big|^p {\xi}^{p, \beta}_{r^{\ell}_k}\Big]
\end{align*}
for any $p\in (0,p^*_{\beta})\cap (0,  \rho]$ and any $\t\in (0, \t^{**}]$.
 Notice that
$
 \big|Z^{x_0,\ell}_{k+1} -Z^{\bar{x}_0,\ell}_{k+1}\big|
 \leq\big|\tilde{Z}^{x_0,\ell}_{k+1}-\tilde{Z}^{\bar{x}_0,\ell}_{k+1}\big|,
$
which implies that
\begin{align*}
\E\left[|Z^{x_0,\ell}_{k+1}\!-Z^{\bar{x}_0,\ell}_{k+1}|^{p}\xi^{p, \beta}_{r^{\ell}_{k+1}}\right]
\leq  \big(1-(\eta_{p,\beta} -\varrho)\t \big)\E \Big[|Z^{x_0,\ell}_{k} \!-Z^{\bar{x}_0,\ell}_{k}|^{p}\xi^{p, \beta}_{r^{\ell}_k}\Big]
\end{align*}
for any integer $k\geq 0$. Repeating this procedure yields
\begin{align}\la{Y-HF}
 \E\left[|Z^{x_0,\ell}_{k}-Z^{\bar{x}_0,\ell}_{k}|^{p}\right]
 \leq C\mathrm{e}^{-(\eta_{p,\beta}-\varrho)k\t}.
\end{align}
Define
$  \tilde{\tau}  =\inf\{k\geq 0:r_{k}^{\ell}=r_{k}^{\bar{\ell}}\}.$
For any $p\in (0,p^*_\alpha\wedge p^*_{\beta}\wedge\bar{p})\cap (0,  \rho]$, choose $q>1$ such that $pq\in(0,p^*_\alpha )\cap (0,  \bar{p}]$. Thus,  using  H\"{o}lder's inequality,  \eqref{equ8} and \eqref{Y-HF} yields
\begin{align}\la{h1}
 &\mathbb E|Z^{x_0,\ell}_{k}-Z^{\bar{x}_0,\bar{\ell}}_{k}|^{p} \nn\\
 =& \mathbb E\Big(|Z^{x_0,\ell}_{k}-Z^{\bar{x}_0,\bar{\ell}}_{k}|^{p}I_{\{\tilde{\tau} >  \lfloor k/2\rfloor+1 \}}\Big)+\mathbb E\Big(|Z^{x_0,\ell}_{k}-Z^{\bar{x}_0,\bar{\ell}}_{k}|^{p}I_{\{\tilde{\tau} \leq \lfloor k/2\rfloor+1\}}\Big) \nn\\
 \leq&\Big(\mathbb E|Z^{x_0,\ell}_{k}-Z^{\bar{x}_0,\bar{\ell}}_{k}|^{pq}\Big)^{\frac{1}{q}}\Big(\mathbb P(\tilde{\tau} > \lfloor k/2\rfloor+1)\Big)^{1-\frac{1}{q}}
  +\mathbb E\Big[I_{\{\tilde{\tau} \leq \lfloor k/2\rfloor+1\}}\mathbb E\Big(|Z^{x_0,\ell}_{k}-Z^{\bar{x}_0,\bar{\ell}}_{k}|^{p}\big|\mathcal F_{\tilde{\tau} \triangle}\Big)\Big]\nn \\
 \leq&\Big(\mathbb E|Z^{x_0,\ell}_{k}-Z^{\bar{x}_0,\bar{\ell}}_{k}|^{pq}\Big)^{\frac{1}{q}}\Big(\mathbb P(\tilde{\tau} > k/2 )\Big)^{1-\frac{1}{q}}
 +\mathbb E\Big[I_{\{\tilde{\tau} \leq \lfloor k/2\rfloor+1\}}\mathbb E\Big(|Z_{k-\tilde{\tau} }^{Z_{\tilde{\tau} }^{x_0,\ell},r_{\tilde{\tau} }^{\ell}}\!-Z_{k-\tilde{\tau} }^{Z_{\tilde{\tau} }^{\bar{x}_0,\bar{\ell}},r_{\tilde{\tau} }^{\bar{\ell}}}|^{p}\Big)\Big] \nn\\
\leq&\mathrm{e}^{-\frac{q-1}{2q}\bar{\lambda} k\triangle}\Big(\mathbb E|Z^{x_0,\ell}_{k}-Z^{\bar{x}_0,\bar{\ell}}_{k}|^{pq}\Big)^{\frac{1}{q}} +C \mathbb E\Big[I_{\{\tilde{\tau} \leq \lfloor k/2\rfloor+1\}} \mathrm{e}^{-(\eta_{p,\beta}-\varrho)(k-\tilde{\tau})\t}\mathbb E\Big(|Z_{\tilde{\tau} }^{x_0,\ell}-Z_{\tilde{\tau} }^{\bar{x}_0,\bar{\ell}}|^{p}\Big)\Big]\nn \\
 \leq&\mathrm{e}^{-\frac{q-1}{2q}\bar{\lambda} k\triangle}\Big(\mathbb E|Z^{x_0,\ell}_{k}-Z^{\bar{x}_0,\bar{\ell}}_{k}|^{pq}\Big)^{\frac{1}{q}}
 +C\mathrm{e}^{-\frac{(\eta_{p,\beta}-\varrho)(k-2)}{2}\t}\mathbb E\Big(\big|Z_{\tilde{\tau} \wedge( \lfloor k/2\rfloor+1)}^{x_0,\ell}-Z_{\tilde{\tau} \wedge(\lfloor k/2\rfloor+1)}^{\bar{x}_0,\bar{\ell}}\big|^{p}\Big).
\end{align}
Due to (\ref{I-02}), one observes that  the conclusion of Theorem \ref{4TT:F_1} holds for the new   scheme \eqref{6FY_0}.
 Then 
$
\sup\limits_{k\geq0}\mathbb E\big|Z_k^{x_0,\ell}\big|^{pq}\leq C,~
\sup\limits_{k\geq0}\mathbb E\big|Z_k^{\bar{x}_0,\bar{\ell}}\big|^{pq}\leq C.
$
One further observes that
\begin{align*}
\mathbb E\big|Z_{\tilde{\tau} \wedge(\lfloor k/2\rfloor+1)}^{x_0,\ell}\big|^{p} \leq \sum_{l=0}^{ \lfloor k/2\rfloor+1 }\mathbb E\Big(\big|Z^{x_0,\ell}_{l}\big|^{p}I_{\{\tilde{\tau} \wedge(\lfloor k/2\rfloor+1)=l\}}(\omega)\Big)
 \leq \sum_{l=0}^{ \lfloor k/2\rfloor+1 }\mathbb E\big|Z^{x_0,\ell}_{l}\big|^{p}
 \leq C(2+k/2),
\end{align*}
which implies
$$
C\mathrm{e}^{-\frac{(\eta_{p,\beta}-\epsilon)(k-2)}{2}\t}\mathbb E\Big(\big|Z_{\tilde{\tau} \wedge  (\lfloor k/2\rfloor+1)}^{x_0,\ell}-Z_{\tilde{\tau} \wedge(\lfloor k/2\rfloor+1)}^{\bar{x}_0,\bar{\ell}}\big|^{p}\Big)\rightarrow 0~ \mathrm{as}~ k\rightarrow\infty.
$$
Inserting this into (\ref{h1}) yields the desired  assertion (\ref{Tty3.8}).
 \end{proof}


Furthermore we  can also use Scheme \eqref{6FY_0} to approximate the stability as follows.
  \begin{cor}\la{CTT:F_1}
Under the conditions of Corollary \ref{CT:F_0}, for any  $p\in (0, p^*_\beta) \cap(0, \rho]$,  any $\varrho\in(0, \eta_{p,\beta})$, there is a constant $ {\t}^{**}\in(0,1]$  such that $\forall\t\in(0,  {\t}^{**}]$, Scheme \eqref{6FY_0} satisfies
  \be\la{FyhfF_4}
 \limsup_{k\rightarrow \infty}\frac{\log\big(\E|Z^{x_0,\ell}_{k}|^{p  }\big)}{k\t} \leq  -(\eta_{p,\beta}-\varrho)<0,
   \ee
and
  \be\la{FyhfF_5}
  \limsup_{k\rightarrow \infty}\frac{\log\big(|Z^{x_0,\ell}_{k}|\big) }{k\t}
   \leq-\frac{\eta_{p,\beta}-\varrho}{p}~~~~~~~a.s.
   \ee
  \end{cor}
\begin{proof}Clearly,  (\ref{FyhfF_4}) follows from \eqref{Y-HF} directly.
Moreover,  an application of  the Borel-Cantelli lemma (see, e.g., \cite{Mao06}) results in (\ref{FyhfF_5}),
please refer to \cite[p.600]{Higham2007}.
\end{proof}

In order for the   ergodicity, we  show the Markov property of  $\big\{(Z^{x_0,\ell}_{k}, r^{\ell}_k)\big\}_{k\geq 0}$.
\begin{lemma}\la{L-4.2}
 $\{(Z_{k},r_{k})\}_{k\geq0}$ is a time homogeneous Markov chain.
\end{lemma}
\begin{proof}   One observes 
\begin{align*}
Z_{k+1}=\bar{\pi}^{r_{k+1}}_{\t}\big(x+f(x,i)\triangle+g(x,i)\triangle B_{k}\big),~~
\mathrm{and}~~
Z_{1}=\bar{\pi}^{r_{1}}_{\t}\big(x+f(x,i)\triangle+g(x,i)\triangle B_{0}\big).
\end{align*}
Since~$\triangle B_{k}$ and~$\triangle B_{0}$ are identical in probability law, comparing the two fomulas above, we know that $(Z_{k+1},r_{k+1})$ and~$(Z_{1}, r_1)$ are identical in probability law under $Z_{k}=x$,  $r_{k}=i$ and $x_{0}=x$,  $r_0=i$ respectively.  Thus
$$
\mathbb{P}\Big((Z_{k+1},r_{k+1})\in \mathbb{D}\times \{j\}\big|(Z_{k},r_{k})=(x,i)\Big) =\mathbb{P}\Big((Z_{1},r_{1})\in \mathbb{D}\times \{j\}\big|(x_{0}, r_0)=(x,i)\Big)
$$
for any $\mathbb{D}\in\mathscr{B}(\RR^n),~j\in \mathbb{S}$, which is the desired homogenous property.
For any $k\geq0$,~$\t\in (0,1]$, $x\in\RR^{n}$,~$i\in \mathbb{S}$, define
$$
 \lambda^{i}_{k+1}:=i+r_{k+1}-r_{k},~~~~
\theta^{x,i}_{k+1, j}:=\bar{\pi}^{ j}_{\t}\big(x+f(x,i)\triangle+g(x,i)\triangle B_{k}\big).
$$
By \eqref{6FY_0} we know that $r_{k+1}=\lambda^{r_k}_{k+1}$ and $Z_{k+1}= \theta^{Z_k,r_k}_{k+1,r_{k+1}}$. Note that $\lambda^{i}_{k+1}$ and $\theta^{x,i}_{k+1, j}$ are bounded measurable random functions
 independent of $\F_{t_k}$. 
 Hence, for any $\mathbb{D}\times \{j\}\in \mathscr{B}(\RR^n)\times \SS$,  using \cite[Lemma 3.2, p.104]{Mao06} with
 $ \bar{h}\big((x,i), \omega\big)=I_{\mathbb{D}\times \{j\}}\big(\theta^{x,i}_{k+1,\lambda^{i}_{k+1}}, \lambda^{i}_{k+1}\big)$ yields
\begin{align*}
\PP\Big((Z_{k+1},r_{k+1})\in \mathbb{D}\times \{j\}\Big|\F_{t_k}\Big) &= \E\Big(I_ {\mathbb{D}\times \{j\}}\big(Z_{k+1},r_{k+1}\big)\Big|\F_{t_k}\Big)\\
 &=\E\Big(I_{\mathbb{D}\times \{j\} }\big(\theta^{Z_k,r_k}_{k+1,j}, \lambda^{r_k}_{k+1}\big)  \Big|\F_{t_k}\Big)\\
&=\E\Big[I_{\mathbb{D}\times \{j\} }\big(\theta^{x,i}_{k+1,j}, \lambda^{i}_{k+1}\big)\Big]\Big|_{x=Z_k, i=r_k}\\
& = \PP\Big(\big(Z_{k+1},r_{k+1}\big)\in \mathbb{D}\times \{j\}  \Big|(Z_k,r_k)\Big),
\end{align*}
which is the desired Markov property.
\end{proof}

Let $\mathbf{P}_{k\triangle}^{\Delta}(x_0,\ell;\mathrm{d} x\times\{i\})$ be the transition probability  kernel of the pair $\big(Z_{k}^{x_0,\ell},r_{k}^{\ell}\big)$, a time homogeneous Markov chain.
 If $\mu^{\triangle}\in\mathcal P(\RR^n\times \mathbb{S})$ satisfies
\begin{align*}
\mu^{\triangle}(\Upsilon\times\{\ell\})
=\sum_{i=1}^{m}\int_{\RR^n}\mathbf{P}_{k\triangle}^{\triangle}(x,i;\Upsilon\times\{\ell\})\mu^{\triangle}(\mathrm{d}x\times\{i\}),
~\Upsilon\in\mathscr{B}(\RR^n),~\ell\in \SS
\end{align*}
for any $k\geq0$, then $\mu^{\triangle}$ is called an invariant measure of $\big(Z_{k}^{x_0, \ell},r_{k}^{\ell}\big)$. Moreover, such an invariant measure $\mu^{\triangle}\in\mathcal P(\RR^n\times \mathbb{S})$ is also called a numerical invariant measure. 
Next we give the existence and uniqueness of the numerical invariant measure
for  SDS (\ref{e1}) using Scheme \eqref{6FY_0}.
\begin{theorem}\la{yth3.2}
Under the conditions of Theorem \ref{yth3.1},
  for any $\t \in(0, \t^{**}]$,
 the numerical solutions of Scheme \eqref{6FY_0} admit a unique invariant measure $\mu^{\Delta}\in \mathcal{P}(\RR^n \times \SS)$.
\end{theorem}
\begin{proof}    For arbitrary integer $l>0$,
define a measure sequence $\{\hbar_{l}\}$ that
\begin{align*}
\hbar_{l}\big(\Upsilon\times \{i\}\big):=\frac{1}{l+1}\sum_{k=0}^{l} \mathbb{P}\Big((Z^{x_0,\ell}_{k},r^{\ell}_{k})\in \Upsilon\times \{i\}\Big)
\end{align*}
for any $\Upsilon\in\mathscr{B}(\RR^n)$ and $i\in \SS$.
For any $p\in (0,p^*_\alpha)\cap (0,  \bar{p}]$ and $\t\in(0, \t^{**}]$, by Theorem \ref{4TT:F_1} and Chebyshev's inequality, we derive that $\{\hbar_{l}\}$ is tight, then one can extract a subsequence which converges weakly to an invariant measure.
  Thus, the numerical solution  $(Z_k^{x_0,\ell},r_k^{\ell})$   has an invariant measure  $\mu^{\triangle}$.
For any $\t\in(0, \t^{**}]$, we have established the existence of  the numerical   invariant measure, and now we further show its uniqueness. For any $p\in (0,p^*_\alpha\wedge p^*_{\beta}\wedge\bar{p})\cap (0,\rho]$, it follows from  \eqref{equ8} and (\ref{Tty3.8}) that
\begin{align}\la{y2.15}
W_{p}(\delta_{(x_0,\ell)}\mathbf{P}_{k\triangle}^{\triangle},\delta_{(\bar{x}_0,\bar{\ell})}\mathbf{P}_{k\triangle}^{\triangle})\leq \mathbb{E}\big|Z_{k}^{x_0,\ell}-Z_{k}^{\bar{x}_0,\bar{\ell}}\big|^{p}+\mathbb{P}(r_{k}^{\ell}\neq r_{k}^{\bar{\ell}})\rightarrow 0,~~k\rightarrow\infty.
\end{align}
 Assume both $\mu^{\triangle}$ and  $\bar{\mu}^{\triangle}\in \mathcal P(\RR^n\times \mathbb{S})$ are invariant measures, then we have
\begin{align*}
W_{p}(\mu^{\triangle},\bar{\mu}^{\triangle})
\leq W_{p}(\delta_{(x,i)}\mathbf{P}_{k\triangle}^{\triangle},\mu^{\triangle}) +W_{p}(\delta_{(x,i)}\mathbf{P}_{k\triangle}^{\triangle},\delta_{(y,j)}\mathbf{P}_{k\triangle}^{\triangle})
+W_{p}(\delta_{(y,j)}\mathbf{P}_{k\triangle}^{\triangle},\bar{\mu}^{\triangle}).
%
\end{align*}
Due to  (\ref{y2.15}), $
W_{p}(\mu^{\triangle},\bar{\mu}^{\triangle})\rightarrow 0$ holds as $k\rightarrow\infty.
$
The desired assertion follows. \end{proof}

The following theorem reveals that numerical invariant measure $\mu^{\triangle}$ converges in the Wassertein distance to the underlying one $\mu$.
\begin{theorem}\la{yth3.3}
Under the conditions of  Theorem \ref{yth3.1},
$
\lim_{\triangle\to 0}W_{p}(\mu,\mu^{\triangle})=0.
$
\end{theorem}
\begin{proof} By the virtues of Theorems  \ref{yth3.1} and \ref{yth3.2}, we have
\begin{align*}
W_{p}(\delta_{(x,i)}\mathbf{P}_{k\triangle},\mu)\leq \int_{\RR^{n}\times \SS}\mu(\mathrm{d}y\times {\{j\}})W_{p}(\delta_{(x,i)}\mathbf{P}_{k\triangle},\delta_{(y,j)}\mathbf{P}_{k\triangle}),
\end{align*}
and
\begin{align*}
W_{p}(\delta_{(x,i)}\mathbf{P}_{k\triangle}^{\triangle},\mu^{\triangle})\leq \int_{\RR^{n}\times \SS}\mu^{\triangle}(\mathrm{d}y\times {\{j\}})W_{p}(\delta_{(x,i)}\mathbf{P}_{k\triangle}^{\triangle},
\delta_{(y,j)}\mathbf{P}_{k\triangle}^{\triangle})
\end{align*}
for any $\t\in(0, \t^{**}]$. Hence
$$
\lim_{k\to \infty}W_{p}(\delta_{(x,i)}\mathbf{P}_{k\triangle},\mu)=\lim_{k\to \infty}W_{p}(\delta_{(x,i)}\mathbf{P}_{k\triangle}^{\triangle},\mu^{\triangle})=0.
$$
Thus for any $\epsilon>0$, there exists a positive integer $\bar{k}>0$ sufficiently large such that
\begin{align}\la{f1}
W_{p}(\delta_{(x,i)}\mathbf{P}_{\bar{k}\triangle},\mu)
+W_{p}(\delta_{(x,i)}\mathbf{P}_{\bar{k}\triangle}^{\triangle},\mu^{\triangle})\leq \frac{\epsilon}{2},~~~~\forall ~\triangle\in(0, \t^{**}].
\end{align}
Moreover,  by Theorem \ref{T:C_2},  there is a $\t_1\in (0, \t^{**}]$  such that
$
 W_{p}(\delta_{(x,i)}\mathbf{P}_{\bar{k}\triangle},\delta_{(x,i)}\mathbf{P}_{\bar{k}\triangle}^{\triangle})< {\epsilon}/{2}
$
for any $\t\in(0, \t_1]$. Therefore, for any $\t\in(0, \t_1]$,
\begin{align*}
W_{p}(\mu,\mu^{\triangle}) \leq  W_{p}(\delta_{(x,i)}\mathbf{P}_{\bar{k}\triangle},\mu)
+W_{p}(\delta_{(x,i)}\mathbf{P}_{\bar{k}\triangle}^{\triangle},\mu^{\triangle})
+W_{p}(\delta_{(x,i)}\mathbf{P}_{\bar{k}\triangle},\delta_{(x,i)}\mathbf{P}_{\bar{k}\triangle}^{\triangle})<\epsilon.
\end{align*}
The proof is complete.
\end{proof}

\begin{rem}
  For the given stepsize $\t$ and $\bar{h}(\t)$, it is easy to find $\hat{\bar{\varphi}}^{-1}(\bar{h}(\t))$,  the minimum of all $\bar{\varphi}_i^{-1}(\bar{h}(\t))$. Define the uniform truncation mapping $\bar{\pi}_\t(x)=  \Big(|x|\wedge \hat{\bar{\varphi}}^{-1}(\bar{h}(\t))\Big)  {x}/{|x|}. $
  Clearly, the properties  \eqref{I-01} and \eqref{I-02} still hold  for Scheme \eqref{6FY_0} with each $\bar{\pi}^i_\t(x)=\bar{\pi}_\t(x)$. This implies that all results on numerical solutions in this section still hold for this uniformly truncated scheme.
\end{rem}
\subsection{Numerical examples}\la{exp}
Before closing this section we carry out some  simulations   to illustrate the efficiency of the  scheme (\ref{6FY_0}) in the approximation of invariant measures.

Recall Example \ref{exp3.1}.  Assumption \ref{a1} holds with $\bar{p}=5/3$ and $\alpha_1=5$, $\alpha_2=-0.64$. Due to
more computations 
 Assumption \ref{a5} holds with $\rho=0.8$ and  $\beta_1=5$, $\beta_2=-0.898$.
 Solving the linear equation (\ref{eq:a1.2}) results in $
\pi=\left(\pi_1, \pi_2\right)=\left( {1}/{21},  {20}/{21}\right).
$ Clearly, $$
\pi \alpha=\pi_1 \alpha_1+\pi_2 \alpha_2= -7.8/21<0, ~\pi \beta=\pi_1 \beta_1+\pi_2 \beta_2=- 12.96/21<0.
$$
It follows  from Theorem \ref{yth3.1} that   exact solutions of stochastic volatility model with random switching  between (\ref{sex1}) and (\ref{sex2}) admit  a unique invariant measure $\mu\in \mathcal{P}(\RR^n\times \SS)$.
By the virtues of Theorems \ref{yth3.2} and \ref{yth3.3} Scheme \eqref{6FY_0} produces a unique numerical invariant measure $\mu^{\t}\in \mathcal{P}(\RR^n\times \SS)$    converging to  $\mu$ in the Wasserstein metric. To the best of our knowledge almost all numerical methods  in the literatures such as \cite{Yuan3,Bao2016,Liu} cannot treat this case.

Secondly we carry out some  numerical experiments.
  Basing on the model structure,  we begin to specify the explicit truncated EM scheme \eqref{6FY_0}.

\vspace*{4pt}\noindent{\bf Step 1.} Choose $\bar{\varphi}_i(\cdot)$ and $\bar{h}(\cdot)$.  Compute
\begin{align*}
 \sup_{|x|\vee|y|\leq u, x\neq y}  \Big(\frac{ |f(x, 1)-f(y,1)|}{|x-y|}\vee \frac{  |g(x, 1)-g(y, 1)|^2}{|x-y|^2}\Big)
\leq   18u,~~~~\forall u\geq 1,
\end{align*}
  and
\begin{align*}
 \sup_{|x|\vee|y|\leq u, x\neq y}  \Big(\frac{ |f(x, 2)-f(y,2)|}{|x-y|}\vee \frac{  |g(x,2)-g(y,2)|^2}{|x-y|^2}\Big)
\leq  1.45,~~~~\forall u\geq 1.
\end{align*}
 Then choose
$
\bar{\varphi}_1(u)=18u, ~ \bar{\varphi}_2(u)=1.45, ~\forall u \geq 1,
$  which implies $$\bar{\varphi}_{1}^{-1}(u)=u/18,~~~~
 ~\bar{\varphi}_{2}^{-1}(u)=+\infty ,  ~\forall u\geq 18.$$  Let  $\bar{h}(\t)=54\t^{-0.4}, ~\forall\t\in (0,1]$. Thus (\ref{6Fe22}) holds.

\vspace*{4pt}\noindent{\bf Step 2.} {M{\scriptsize ATLAB}} code. Next we specify the {M{\scriptsize ATLAB}} code for calculating the truncated EM  approximation $Z(t)$:
\vspace{-1.2em}
\begin{lstlisting}%[firstline=1, lastline=23]%[firstline=1, lastline=19, float,caption=A floating example]

%MATLAB code for calculating the truncated EM approximation Z(t)
clear all;
T=100; dt=2^(-9); Gam=[-4 4;0.2 -0.2]; c=expm(Gam*dt); h=54*dt^(-2/5);
Z=zeros(2,T/dt+1,'double'); r=zeros(1,T/dt+1,'double'); Z(:,1)=[1;1];
r(1)=2; dB=sqrt(dt)*randn(2,T/dt); v=h/18; %Obviously, v>norm(Z(:,1));
for n=1:T/dt
    if r(n)==1
        Z(:,n+1)=Z(:,n)+2.5*Z(:,n)*(1-norm(Z(:,n)))*dt+...
            [-1  sqrt(2);sqrt(2) 1]*norm(Z(:,n))^(3/2)*dB(:,n);
    else
        Z(:,n+1)=Z(:,n)+([1;2]-Z(:,n))*dt+...
            [0.2 -0.5;1 0.4]*norm(Z(:,n))*dB(:,n);
    end
    if rand<=c(r(n),1)
        r(n+1)=1;
        if norm(Z(:,n+1))>v
            Z(:,n+1)=v*Z(:,n+1)/norm(Z(:,n+1));
        end
    else
        r(n+1)=2;
    end
end
\end{lstlisting}
\begin{figure}[!ht]
  \centering
\includegraphics[width=12cm,height=6.3cm]{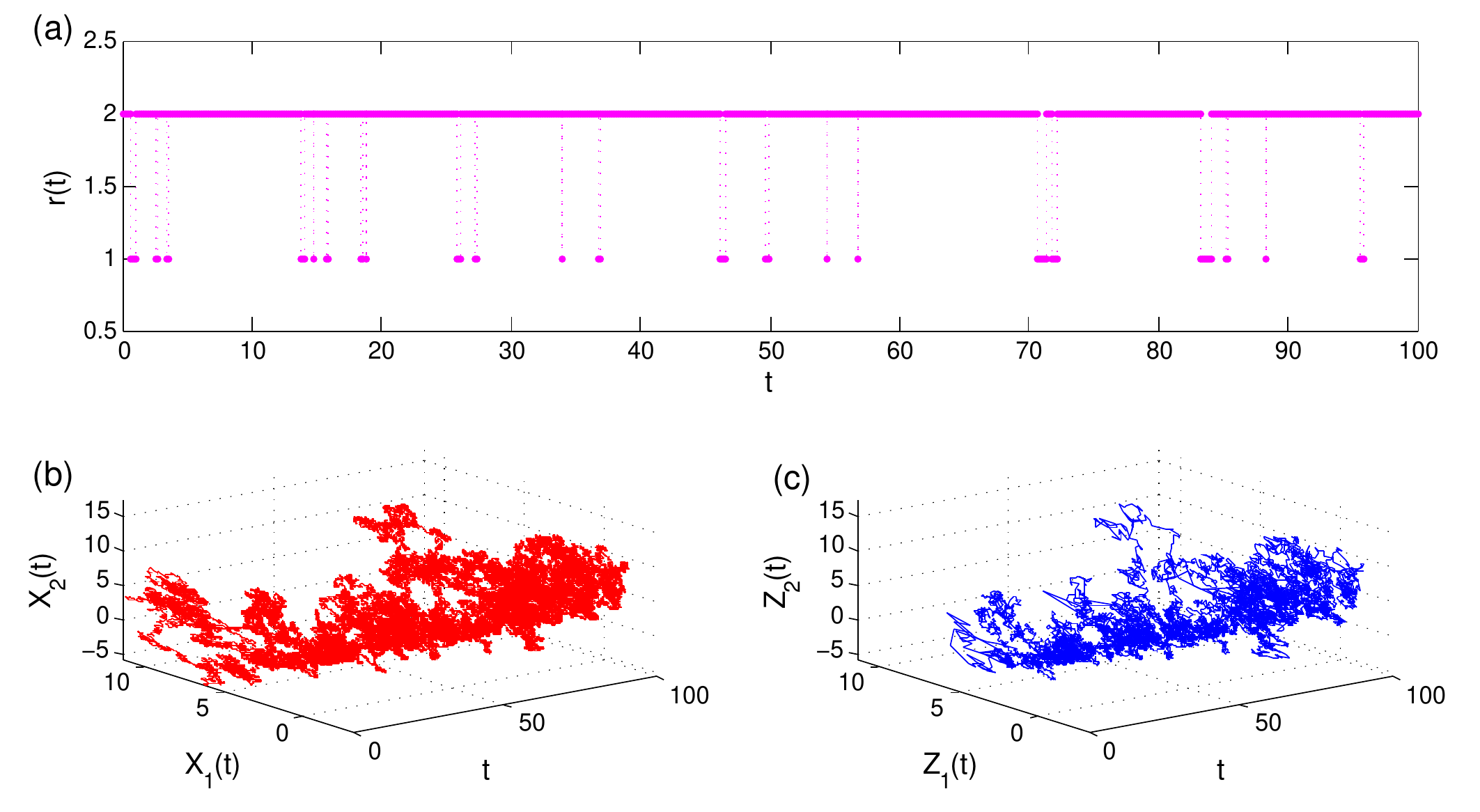}
  \caption{(a)  Computer simulation of a single path of  Markov chain $r(t)$.  (b) A sample path of exact solution $X(t)$ in 3D settings. (c) A sample path of numerical solution $Z(t)$ in 3D settings. The red trajectory represents the exact solution (i.e. the numerical solution of Scheme \eqref{6FY_0} with $\triangle=2^{-18}$) while  the blue trajectory represents the numerical solution of Scheme \eqref{6FY_0} with $\triangle=2^{-9}$.}
  \label{figure1}
\end{figure}
\begin{figure}[!ht]
  \centering
\includegraphics[width=14cm,height=8.5cm]{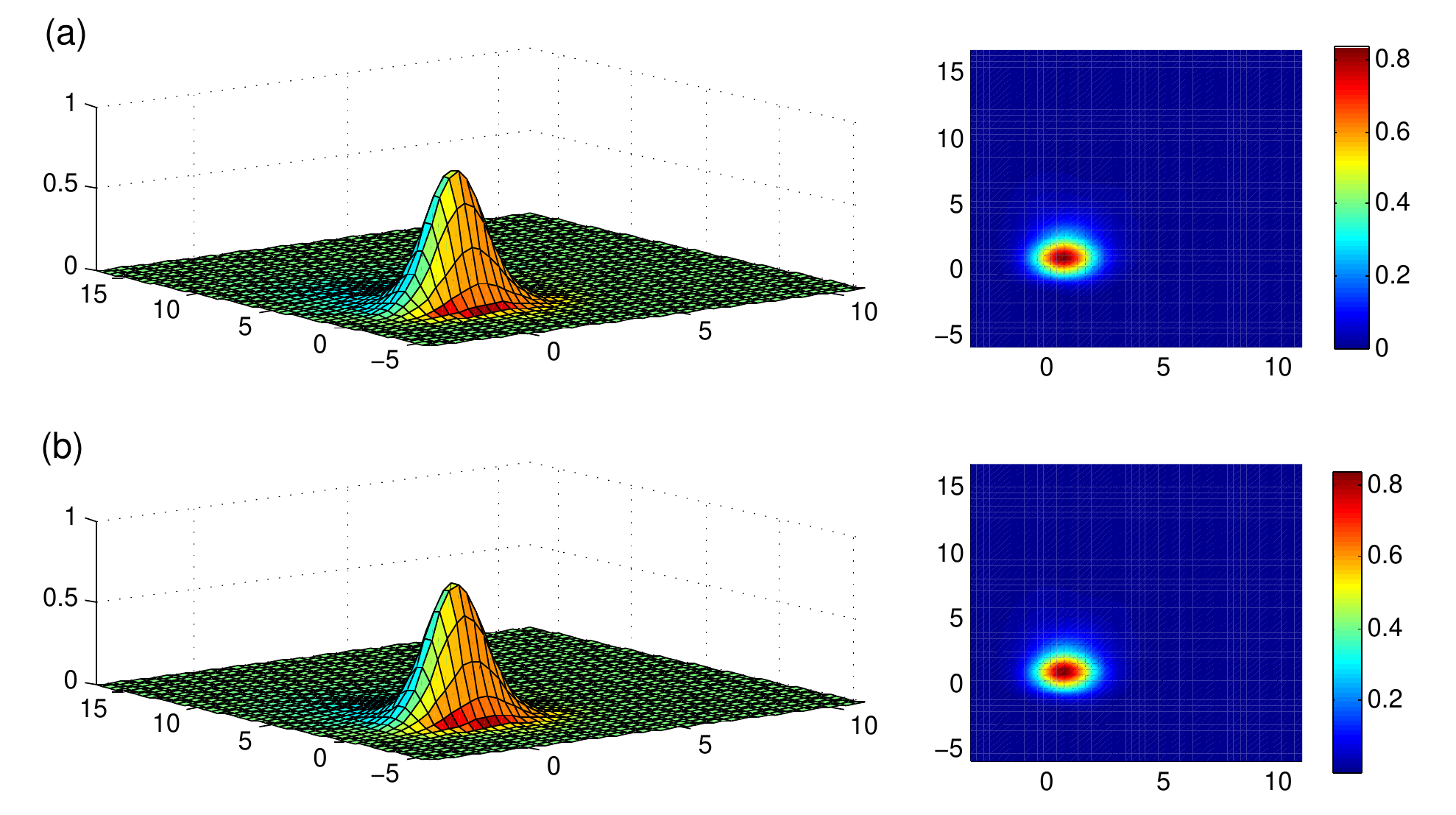}
  \caption{(a)  The empirical density of  $\mu^{\t}$ in 2D and 3D settings.  (b) The   empirical density of  $\mu$ in 2D and 3D settings.   }
  \label{figure2}
\end{figure}
\vspace*{4pt}\noindent{\bf Step 3.} MATLAB simulations.
  Without the closed-form,
the more precise numerical solution $X(t)=(X_1(t), X_2(t))^T$ with $\triangle=2^{-18}$   is a good substitute of the  exact solution.  We compare it with the other numerical solution $ Z(t)=(Z_1(t), Z_2(t))^T$  with  stepsize $\triangle=2^{-9}$.
Figure \ref{figure1} depicts the paths of the Markov chain (Figure \ref{figure1}(a)), $X(t)$ (Figure \ref{figure1}(b)) and $Z(t)$ (Figure \ref{figure1}(c)) for $t\in [0, 100]$.
With 51201 iterations, Figure \ref{figure2}(a) depicts the empirical density of  $\mu^{\t}$ in 2D and 3D settings  while Figure \ref{figure2}(b) depicts the empirical density of  $\mu$ in 2D and 3D settings. It is evident to see that these two density pictures   are very similar.  To  support the theoretical results deeply, we further plot the empirical cumulative distribution function (ECDF) of $ Z(t)$ with the blue dashed line  and the ECDF of  $ X(t)$ with the red solid line.
 To measure the similarity quantitatively, we use the Kolmogorov-Smirnov test  \cite{Massey} to test the alternative hypothesis that  the exact   and numerical invariant measures are from different distributions against the null hypothesis that they are from the same distribution for each component. With 2\% significance level, the  Kolmogorov-Smirnov test
indicates that we cannot reject the null hypothesis. So
the numerical invariant measure  approximates the underlying exact invariant measure very well.
\begin{figure}[!ht]
  \centering
\includegraphics[width=15cm,height=7cm]{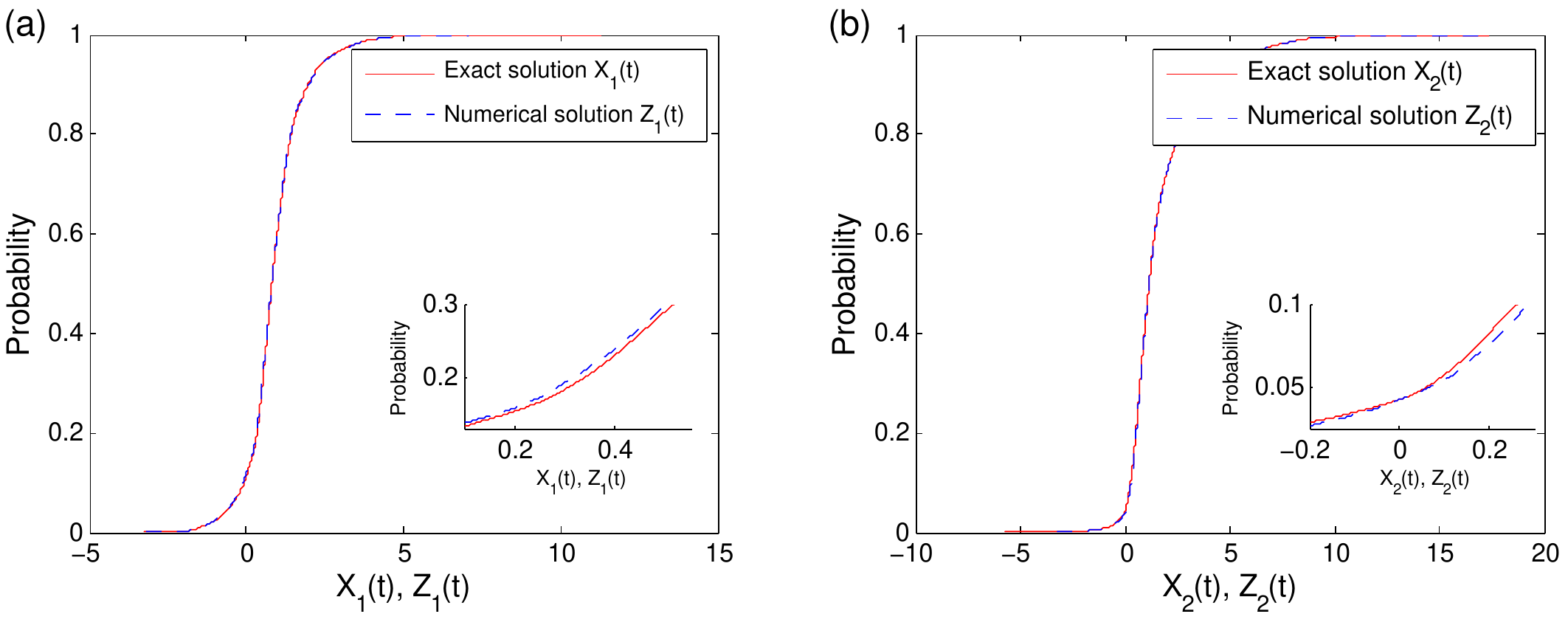}
  \caption{(a)  The ECDFs of $X_1(t)$, $Z_1(t)$.  (b) The ECDF of $X_2(t)$, $Z_2(t)$.   The red solid line represents the exact solution of the switching system while the blue dashed line represents the numerical solution of the switching system.}
  \label{figure3}
\end{figure}

To further illustrate the validity,  we reanalyze  the existence of the invariant measure and its approximation of SDS \eqref{exyhf4.1} . 

\begin{expl}\la{exp5.3}{\rm  Let $r(t)$ be a Markov chain with the state space $\SS=\{1, 2\}$ and the generator  $$\Gamma=\left(
  \begin{array}{ccc}
    -\gamma  & \gamma  \\
    3 & -3
  \end{array}
\right)
$$ for some $\gamma >0$. Its unique stationary distribution $\pi=\left(\pi_1, \pi_2\right)\in \mathbb{R}^{1\times 2}$ is given by
$
 \pi_1= {3}/({3+\gamma }), ~\pi_2= {\gamma }/({3+\gamma }).
$
 Consider the scalar hybrid cubic SDS \eqref{exyhf4.1}  with the initial value $(x_0, \ell)=(0.5, 2)$, and  coefficients
$$
a(1)=1,~b(1)=-1,~\sigma(1)=2;~a(2)=2,~b(2)=0,~\sigma(2)=-1.
$$
 Thus,
  Assumption \ref{a1} holds with any negative constant $\alpha_1 $ and $\alpha_2=3+\bar{p}$ for any $\bar{p}>0$.  Moreover, one observes 
Assumption \ref{a5} holds  with $\beta_1=4\rho-2$, $\beta_2=3+\rho$ for any $\rho>0$.
 Let $\rho=0.1$,  then we know
$
\pi\alpha=\pi_1 \alpha_1+\pi_2 \alpha_2<0,~\pi \beta=\pi_1 \beta_1+\pi_2 \beta_2<0
$
holds  with $\gamma \in(0, 1.548]$. It follows  from Theorem \ref{yth3.1} that the  exact solutions  of  SDS \eqref{exyhf4.1} admit  a unique invariant measure $\mu\in \mathcal{P}(\RR \times \SS)$.
Moreover,  by Theorems \ref{yth3.2} and \ref{yth3.3}, the unique numerical invariant measure $\mu^{\t}\in \mathcal{P}(\RR \times \SS)$ of the truncated EM scheme exists and   converges to $\mu$ in the Wasserstein metric.

Next we begin to construct the explicit scheme to  approximate the underlying invariant measure of  SDS \eqref{exyhf4.1}.
For any $i\in \mathbb{S}$, compute
\begin{align*}
 \sup_{|x|\vee|y|\leq u, x\neq y}  \Big(\frac{ |f(x,i)-f(y,i)|}{|x-y|} \vee\frac{ |g(x,i)-g(y,i)|^2}{|x-y|^2}\Big)
\leq  (3|b_i|u^2+|a_i|)\vee \sigma^2_i, ~~~\forall ~ u\geq 1.
\end{align*}
Thus, for any $u\geq 1$, we choose
$ \bar{\varphi}_1(u)=3u^2+1, ~\bar{\varphi}_2(u)=2,$ which implies  $\bar{\varphi}_1^{-1}(u)=\big((u-1)/3\big)^{1/2},~ \forall~u\geq 4,~
\bar{\varphi}_2^{-1}(u)=+\infty,~~\forall~u\geq 2.$   Let $\bar{h}(\t)=6\t^{-0.4}.$
Clearly, (\ref{6Fe21}) and (\ref{6Fe22}) hold for any $\t\in (0,1]$. Then,
  we give the   truncated EM scheme
\begin{align}  \la{N-EME}
\left\{
\begin{array}{ll}
Z_0=x_0,~~r_0=\ell,&\\
\tilde{Z}_{k+1}=Z_k+[a(r_k)Z_{k}+b(r_k)(Z_{k})^3]\t +\sigma(r_k)Z_k\t B_k,~~~~~~~ \\
Z_{k+1}=\dis\Big(|\tilde{Z}_{k+1}|\wedge   \bar{\varphi}^{-1}_{r_{k+1}}(\bar{h}(\t)) \Big) \frac{\tilde{Z}_{k+1}}{|\tilde{Z}_{k+1}|}
\end{array}
\right.
\end{align}
for any $k=0, 1,\dots, \bar{N}-1.$
Let $\gamma=1.5$, $T=100$ and stepsize  $\triangle=10^{-4}$. We implement  Scheme \eqref{N-EME} in the numerical experiments.
 We simulate  $100$ paths by using  {M{\scriptsize ATLAB}}.
 On the  computer running at Intel Core i3-4170 CPU 3.70 GHz,  the  runtime of  the  truncated EM scheme \eqref{N-EME} is  about 22.137421 seconds while the the  runtime of  the backward EM scheme (\ref{yhfBE1}) is about 38.229964  seconds   on the same computer.
Thus we know that the speed of the  truncated EM scheme \eqref{N-EME} is    1.727 times faster than that of the backward EM scheme.
  Figure \ref{figure4} depicts 10   paths
  of the numerical solution of Scheme \eqref{N-EME}.
  Figure \ref{figure5}(a) depicts the path of the Markov chain,
Figure \ref{figure5}(b) further compares  the path of the exact solution $X(t)$ with that of  the  numerical solution $Z(t)$ while Figure \ref{figure5}(c) compares  the ECDF  of the  exact solution with that of the  numerical solution.  The similarity between the paths as well as the distributions  is significant.  Thus  the numerical invariant measure approximates  the underlying one very well.
\begin{figure}[!ht]
  \centering
\includegraphics[width=13cm,height=6cm]{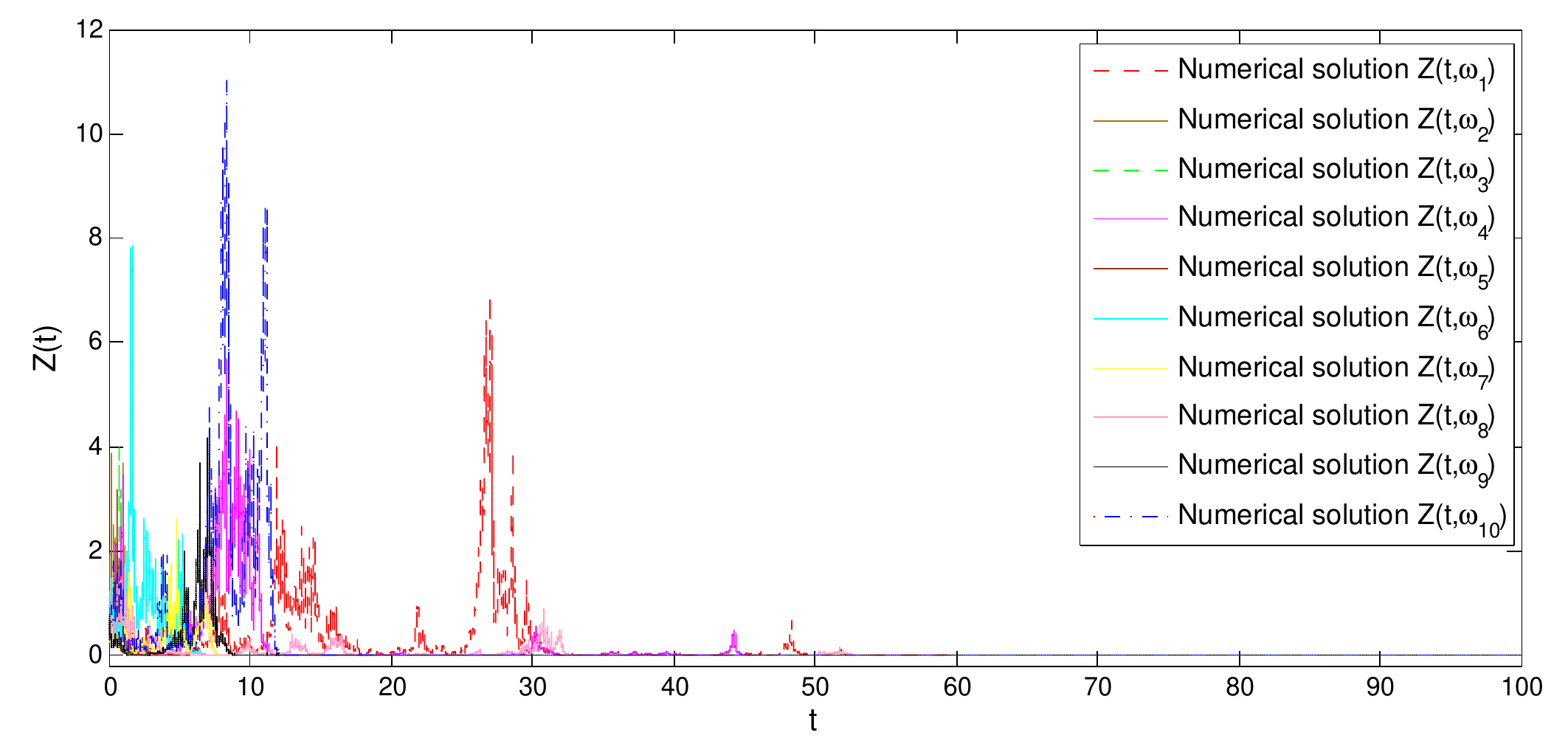}
  \caption{ 10 trajectories of the numerical solution of Scheme \eqref{N-EME} with $x_0=0.5$, $\ell=2$ and stepsize  $\triangle=10^{-4}$.}
  \label{figure4}
\end{figure}
\begin{figure}[!ht]
  \centering
\includegraphics[width=13cm,height=9cm]{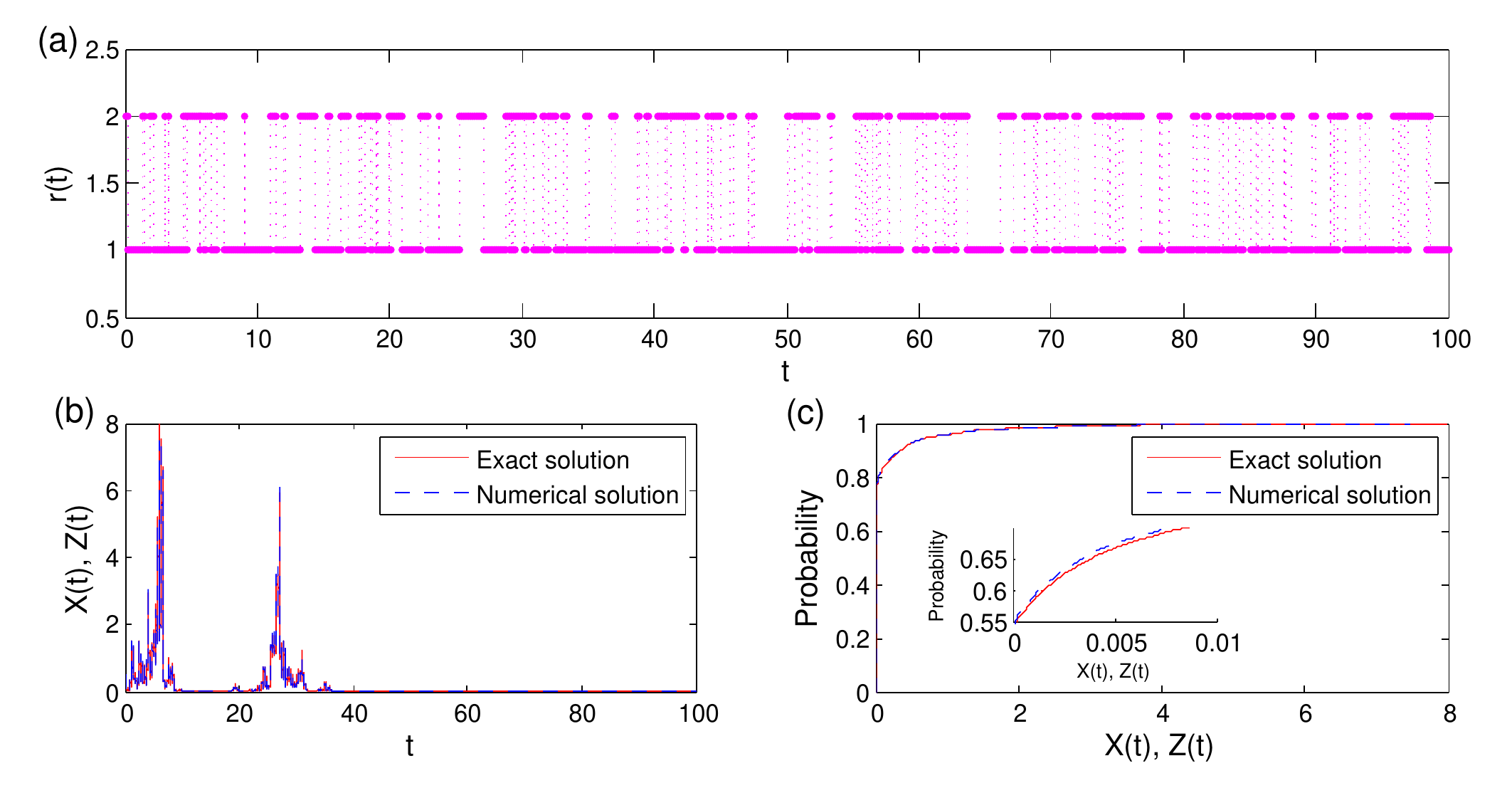}
  \caption{ (a)  Computer simulation of a sample path of  Markov chain $r(t)$.  (b) Sample paths of the exact solution (the red solid line) and numerical solution (the blue dashed line). (c) ECDFs for the exact solution (the red solid line) and numerical solution (the blue dashed line). }
  \label{figure5}
\end{figure}

}
\end{expl}

\section{Concluding Remarks}\la{s6}

This paper investigates the approximation methods for
 the SDSs without globally Lipschitz continuous  coefficients.
A novelty
is to construct two   explicit  schemes  approximating the dynamical properties of SDSs.   
In finite horizon, by one scheme, we show the   boundedness of the numerical solutions, obtain  the convergence in the $p$th moment and estimate the rate of convergence. 
 On the other hand, in infinite horizon we use the other  scheme
to approximate  the underlying invariant measure of exact solutions  in   the Wasserstein distance. Moreover, this scheme is also suitable to realize the stability of SDSs.
 Our exploiting schemes perform the dynamical behaviors of exact solutions very well but don't require more restrictions except the structure conditions which guarantee the exact solutions posses some propery such as stability, moment boundedness and ergodicity.
 Some simulation examples  are provided to support the theoretical results and demonstrate the validity of approaches.

\section*{Acknowledgements}

The authors
thank the editors and referee for the
helpful comments and suggestions.

\bibliography{jde_arxiv}

\end{document}